\documentclass[12pt,reqno]{amsart}
\usepackage[colorlinks,citecolor=red,pagebackref,hypertexnames=false]{hyperref}

\makeatletter
\def\@tocline#1#2#3#4#5#6#7{\relax
  \ifnum #1>\c@tocdepth % then omit
  \else
    \par \addpenalty\@secpenalty\addvspace{#2}%
    \begingroup \hyphenpenalty\@M
    \@ifempty{#4}{%
      \@tempdima\csname r@tocindent\number#1\endcsname\relax
    }{%
      \@tempdima#4\relax
    }%
    \parindent\z@ \leftskip#3\relax \advance\leftskip\@tempdima\relax
    \rightskip\@pnumwidth plus4em \parfillskip-\@pnumwidth
    #5\leavevmode\hskip-\@tempdima
      \ifcase #1
       \or\or \hskip 1em \or \hskip 2em \else \hskip 3em \fi%
      #6\nobreak\relax
    \dotfill\hbox to\@pnumwidth{\@tocpagenum{#7}}\par
    \nobreak
    \endgroup
  \fi}
\makeatother

\usepackage{tikz,amsthm,amsmath,amstext,amssymb,amscd,epsfig,euscript, mathrsfs, dsfont,pspicture,multicol,graphpap,graphics,graphicx,times,enumerate,subfig,sidecap,wrapfig,color}
\usepackage{ hyperref}
\usepackage{enumerate}

 \numberwithin{equation}{section}

\def\bB{{\mathbb{B}}}

\def\bR{{\mathbb{R}}}
\def\bS{{\mathbb{S}}}

\def\bN{{\mathbb{N}}}

\def\cH{{\mathcal{H}}}

\def\cA{{\mathscr{A}}}

\def\cC{{\mathscr{C}}}

\def\cF{{\mathscr{F}}}
\def\cG{{\mathscr{G}}}
\def\csH{{\mathscr{H}}}

\def\cM{{\mathscr{M}}}

\def\cP{{\mathscr{P}}}

\def\cS{{\mathscr{S}}}

\def\one{\mathds{1}}

\def\ve{\varepsilon}

\renewcommand{\d}{{\partial}}
\def\lec{\lesssim}
\def\gec{\gtrsim}

\def\ext{\mathop\mathrm{ext}} 	
 	
\DeclareMathOperator{\diam}{diam}
\def\Tan{\mathop\mathrm{Tan}} 					%tangent measures
\def\BMO{\mathop\mathrm{BMO}} 					%BMO
\def\VMO{\mathop\mathrm{VMO}} 					%VMO
\def\Lip{\mathop\mathrm{Lip}} 						%lipschitz constant 
\def\dim{\mathop\mathrm{dim}} 					%dimension	
\def\dist{\mathop\mathrm{dist}} 						%distance
\def\supp{\mathop\mathrm{supp}}					%support
\def\loc{\mathop\mathrm{loc}}						%locally
\renewcommand{\div}{\mathop\mathrm{div }}			%divergence
\def\warrow{\rightharpoonup}								%weak limit

\newcommand{\ps}[1]{\left( #1 \right)}
\newcommand{\bk}[1]{\left[ #1 \right]}
\newcommand{\ck}[1]{\left\{#1 \right\}}
\newcommand{\av}[1]{\left| #1 \right|}

\newcommand{\ceil}[1]{\left\lceil #1 \right\rceil}
\newcommand{\cnj}[1]{\overline{#1}}

\def\Xint#1{\mathchoice
{\XXint\displaystyle\textstyle{#1}}%
{\XXint\textstyle\scriptstyle{#1}}%
{\XXint\scriptstyle\scriptscriptstyle{#1}}%
{\XXint\scriptscriptstyle\scriptscriptstyle{#1}}%
\!\int}
\def\XXint#1#2#3{{\setbox0=\hbox{$#1{#2#3}{\int}$ }
\vcenter{\hbox{$#2#3$ }}\kern-.58\wd0}}

\def\avint{\Xint-}
\def\grad{\nabla}

\theoremstyle{plain}
\newtheorem{theorem}{Theorem}

\newtheorem{corollary}[theorem]{Corollary}

\newtheorem{lemma}[theorem]{Lemma}

\newtheorem{proposition}[theorem]{Proposition}

\theoremstyle{definition}

\newtheorem{definition}[theorem]{Definition}

\numberwithin{equation}{section}
\numberwithin{theorem}{section}

\newcommand\eqn[1]{\eqref{e:#1}}
\def\Claim{ {\bf Claim: }}
\newcommand\Theorem[1]{Theorem \ref{t:#1}}
\newcommand\Lemma[1]{Lemma \ref{l:#1}}
\newcommand\Corollary[1]{Corollary \ref{c:#1}}

\newcommand\Proposition[1]{Proposition \ref{p:#1}}

  \DeclareFontFamily{U}{mathb}{\hyphenchar\font45} 
\DeclareFontShape{U}{mathb}{m}{n}{
      <5> <6> <7> <8> <9> <10> gen * mathb
      <10.95> mathb10 <12> <14.4> <17.28> <20.74> <24.88> mathb12
      }{}
\DeclareSymbolFont{mathb}{U}{mathb}{m}{n}

\DeclareMathSymbol{\toitself}      {3}{mathb}{"FD}  %*

\newcommand{\vv}{\vspace{2mm}}
\newcommand{\vvv}{\vspace{4mm}}

\newcommand{\dv}{\mathop{\rm div}}
\def\R{\mathbb{R}}
\def\vphi{\varphi}
\def\om{\Omega}
\def\hm{\omega}

  \newtheorem{main}{Theorem}
    \newtheorem{mainprop}{Proposition}

\begin{document}

\title[Tangent measures of elliptic harmonic measure]{Tangent measures of elliptic harmonic measure and applications}

\author{Jonas Azzam}
\address{School of Mathematics, University of Edinburgh, JCMB, Kings Buildings,
Mayfield Road, Edinburgh,
EH9 3JZ, Scotland.}
\email{j.azzam ``at" ed.ac.uk}

\author[Mihalis Mourgoglou]{Mihalis Mourgoglou}
\address{Departamento de Matem\`aticas, Universidad del Pa\' is Vasco, Aptdo. 644, 48080 Bilbao, Spain and\\
Ikerbasque, Basque Foundation for Science, Bilbao, Spain.}
\email{michail.mourgoglou@ehu.eus}

\thanks{Both authors were supported by the ERC grant 320501 of the European Research Council (FP7/2007-2013). M.M. was supported  by the Basque Government through IKERBASQUE and the BERC 2014-2017 program, and by Spanish Ministry of Economy and Competitiveness MINECO: BCAM Severo Ochoa excellence accreditation SEV-2013-0323 and MTM2014-53850.}

\newcommand{\mih}[1]{}%\marginpar{\color{red} \scriptsize \textbf{Mi:} #1}}
\newcommand{\jonas}[1]{}%\marginpar{\color{magenta} \scriptsize \textbf{Jonas:} #1}}

\maketitle

\begin{abstract}
	Tangent measure and blow-up methods, are powerful tools for understanding the relationship between the infinitesimal structure of the boundary of a domain and the behavior of its harmonic measure. We introduce a method for studying tangent measures of elliptic measures in arbitrary domains associated with (possibly non-symmetric) elliptic operators in divergence form whose coefficients have vanishing mean oscillation at the boundary. In this setting, we show the following for domains $ \om \subset \bR^{n+1}$:
	\begin{enumerate}
		\item We extend the results of Kenig, Preiss, and Toro \cite{KPT09} by showing mutual absolute continuity of interior and exterior elliptic measures for {\it any} domains implies the tangent measures are a.e. flat and the elliptic measures have dimension $n$.
		\item We generalize the work of Kenig and Toro \cite{KT06} and show that $\VMO$ equivalence of doubling interior and exterior elliptic measures for general domains implies the tangent measures are always elliptic polynomials.
		\item  In a uniform domain that satisfies the capacity density condition and whose boundary is locally finite and has a.e. positive lower $n$-Hausdorff density, we show that if the elliptic measure is absolutely continuous with respect to $n$-Hausdorff measure then the boundary is rectifiable. This generalizes the work of Akman, Badger, Hofmann, and Martell \cite{ABHM17}.
		\end{enumerate}
	
	Finally, we generalize one of the main results of \cite{Bad11} by showing that if $\omega$ is a Radon measure for which all tangent measures at a point are harmonic polynomials vanishing at the origin, then they are all homogeneous harmonic polynomials.
\end{abstract}

\tableofcontents

\section{Introduction}
\subsection{Background} 
%The harmonic measure for a domain $\Omega\subset \bR^{n+1}$ with pole at $x\in \Omega$ is a measure $\omega_{\Omega}^{x}$ that computes the probability that Brownian motion starting at $x$ will first hit $\d\Omega$ in some set. Let $\Omega^{+}$ be a connected open set with exterior $\Omega^{-}=\ext(\Omega):=(\cnj{\Omega})^{c}$ and let $\omega^{\pm}$ denote harmonic measures for these two domains. Note that $\omega^{\pm}$ are both measures on the same set $\d\Omega^{+}$. 
In this paper, we study how the relationships between the elliptic measures, $\omega^{+}$ and $\omega^{-}$, of two complementary domains, $\Omega^{+}$ and $\Omega^{-}$ respectively, dictate the geometry of their common boundaries. In \cite{KPT09}, Kenig, Preiss, and Toro showed that if $\Omega^{\pm}$ are both nontangentially accesible (or NTA) domains in $\R^{n+1}$ and the interior and exterior harmonic measures are mutually absolutely continuous, then at every point of the common boundary except for a set of harmonic measure zero, $\d\Omega^{+}$ looks flatter and flatter as we zoom in. We will not define NTA but refer the reader to its inception in \cite{JK82}.  Recently, the authors of the current paper, along with Tolsa \cite{AMT16}, as well as with Tolsa and Volberg \cite{AMTV16}, showed  that additionally the boundary is $n$-rectifiable in the sense that, off a set of harmonic measure zero, the boundary is a union of Lipschitz images of $\bR^{n+1}$, and in fact $\Omega^{+}$ and $\Omega^{-}$ need not be NTA but just connected. 

These are, however, almost everywhere phenomena, so it is interesting to ask what assumptions we need on $\omega^{\pm}$ to guarantee some nice limiting behavior of our blowups at {\it every} point. In \cite{KT06}, Kenig and Toro showed that if  $\Omega^+$ is 2-sided NTA and $\log \frac{d\omega^{-}}{d\omega^{+}}\in \VMO(d\omega^{+})$, then as we zoom in on any point of the boundary for a particular sequence of scales, $\d\Omega^{+}$ begins to look more and more like the zero set of a harmonic polynomial (see Section \ref{s:vmo} for the definition of $\VMO$). In \cite{Bad11}, Badger proved that if the blowups of harmonic measure for an NTA domain are harmonic polynomials, they are in fact homogeneous polynomials, and later in \cite{Bad13} investigated the topological properties of the sets where the blowups were polynomials of particular degrees.

%
%Engelstein has also shown in \cite{E15} that higher regularity in $\frac{d\omega^{-}}{d\omega^{+}}$ implies better regularity of $\d\Omega^{+}$ if $\Omega^{+}$ is two-sided NTA.

To explain these results in more detail, we need to discuss what we mean by ``blowups" and what it means for these to look like not necessarily one object but any one of a class of objects as we zoom in on harmonic measure. There are two ways we can consider this. Firstly, we can look at the Hausdorff convergence of rescaled copies of the support of a measure as we zoom in. To do this, we follow the framework of Badger, Engelstein, and Toro \cite{BET17}.
\begin{definition}
Let $A\subset \bR^{n+1}$ be a set. For $x\in A$, $r>0$, and $\cS$ a collection of sets, define
\[
\Theta_{A}^{\cS}(x,r)= \inf_{S\in \cS} \max\ck{\sum_{a\in A\cap B(x,r)} \frac{\dist(a,x+S)}{r}, \hspace{-12pt}\sum_{z\in (x+S)\cap B(x,r)} \frac{\dist(z,A)}{r}}.\]
We say $x\in A$ is a {\it $\cS$ point} of $A$ if $\lim_{r\rightarrow 0} \Theta_{A}^{\cS}(x,r)=0$. We say $A$ is {\it locally bilaterally well approximated by $\cS$} (or simply $LBWA(\cS)$) if for all $\ve>0$ and all compact sets $K\subset A$, there is $r_{\ve,K}>0$ such that $\Theta_{A}^{\cS}(x,r)<\ve$ for all $x\in K$ and $0<r<r_{\ve,K}$.
\end{definition}

Thus, for $x\in A$ to be a $\cS$ point means that, as we zoom in on $A$ at the point $x$, the set $A$ resembles more and more like an element of $\cS$ (though that element may change as we zoom in). 

Secondly, we can look at the weak convergence of rescaled copies of the measure itself. To do this, we follow the framework of Preiss in \cite{Pr87}. For $a\in\bR^{n+1}$ and $r>0$, set
\[T_{a,r}(x) = \frac{x-a}{r}.\]
Note that $T_{a,r}(B(a,r))=B(0,1)$. Given a Radon measure $\mu$, the notation $T_{a,r}[\mu]$ is the image measure of $\mu$ by $T_{a,r}$.
That is,
$$T_{a,r}[\mu](A) = \mu(rA+a),\qquad A\subset\bR^{n+1}.$$
Here and later, for a function $f$ and a measure $\mu$, we write $f[\mu]$ to denote the push-forward measure measure $f[\mu](A)=\mu(f^{-1}(A))$.
\begin{definition}
We say that $\nu$ is a {\it tangent measure} of $\mu$ at a point $a\in\bR^{n+1}$ if
$\nu$ is a non-zero Radon measure on $\bR^{n+1}$ and there are sequences $c_{i}>0$ and $r_{i}\downarrow 0$ so that $c_i\,T_{a,r_{i}}[\mu]$ converges weakly to $\nu$ as $i\to\infty$ and write $\nu\in \Tan(\mu,a)$.
\end{definition}

That is, $\nu$ is a tangent measure of $\mu$ at a point $\xi$ if, as we zoom in on $\mu$ at $\xi$ for a sequence of scales, the rescaled $\mu$ converges weakly to $\nu$.

The collections of measures and sets that we will consider are associated to zero sets of harmonic functions. 
Let $H$ denote the set of harmonic functions vanishing at the origin, $P(k)$ denote the set of harmonic polynomials $h$ of degree $k$ such that $h(0)=0$ and $F(k)$ the set of homogeneous polynomials of degree $k$. For $h\in H$, we define 
\[
\Sigma_{h}=\{h=0\}, \;\;\; \Omega_{h}=\{h>0\}, \]
and  
\[\csH=\{\omega_{h}: h\in H\}, \; \cP(k)=\{\omega_{h}:h\in P(k)\}, \; \cF(k)=\{\omega_{h}: h\in F(k)\}, \]
where
\[
\omega_{h}= -\nu_{\Omega_{h}} \cdot \grad h d\sigma_{\Sigma_{h}}.\]
Also set
\[
\cP_{\Sigma}(k)=\{\Sigma_{h}: h\in P(1)\cup\cdots \cup P(k)\},\;\; \cF_{\Sigma}(k)=\{\Sigma_{h}:h\in F(k)\}\,
\]
and
\[
 \cH_{\Sigma}=\{\Sigma_{h}:h\in H\}.\]
Here $\nu_{\Omega_{h}}(x)$ stands for the measure theoretic  unit  outward normal of $\Omega_{h}$ at $x \in \d^* \Omega_{h}$, the reduced boundary of $\Omega_{h}$. Since $h$ is a harmonic function and thus, real analytic, which implies that $\Sigma_{h}$ is an $n$-dimensional real analytic variety, $\Omega_h$ is a set of locally finite perimeter and one can prove that $\cH^{n}(\d \Omega_{h} \setminus \d^* \Omega_{h})=0$, where $\cH^{n}$ stands for the $n$-Hausdorff measure. Notice now that $\nu_{\Omega_{h}}(x)$ is defined at $\cH^{n}$-almost every point of $\Sigma_{h}$ and $\sigma_{\Sigma_h}$ is the usual surface measure. For a detailed proof of this see  \cite[p. 21]{AMT16}.

%We also let
%\[\cH'=\{\Sigma_{h}: h\in H\}, \;\;\; \cP(k)'=\{\Sigma_{h}:h\in P(k)\}, \;\;\; \cF(k)'=\{\Sigma_{h}: h\in F(k)\}.\]

We summarize the best results to date. We first mention a result by the authors, Tolsa, and Volberg.

\begin{theorem}
	Let $\Omega^{\pm}\subset \bR^{n+1}$ be two disjoint domains and $\omega^{\pm}=\omega^{L_{A},x_{\pm}}_{\Omega^{\pm}}$ for some $x_{\pm}\in \Omega^{\pm}$. If $\omega^{\pm}$ are mutually absolutely continuous on $E$, then for $\omega^{\pm}$-a.e. $\xi\in E$, $\Tan(\omega^{\pm},\xi)\subset \cF(1)$ and $\omega^{+}|_{E}$ can be covered up to a set of $\omega^{+}$-measure zero by $n$-dimensional Lipschitz graphs. Furthermore, if $\d\Omega^{\pm}$ are CDC, then $\lim_{r\rightarrow 0}\Theta_{\d\Omega^{+}}^{\cF_{\Sigma}(1)}(\xi,r)=0$ for $\omega^{+}$-a.e. $\xi\in E$. 
\end{theorem}

This was originally shown by Bishop, Carleson, Garnett, and Jones for simply connected planar domains \cite{BCGJ88}.  Later, Kenig, Preiss and Toro showed that, under the same assumptions, provided that the domain is also 2-sided locally NTA, it holds that $\dim \omega^{+}=n$ (but not that $\omega^{+}$ is rectifiable).

Below we summarize the results so far in the situation when $\Omega$ is 2-sided NTA and the interior and exterior harmonic meausres are $\VMO$ equivalent, which brings together results and techniques from Badger \cite{Bad11,Bad13} and Kenig and Toro \cite{KT06}. 

\begin{theorem}\label{t:bet}
Let $\Omega^{+}\subset \bR^{n+1}$ and $\Omega^{-}=\ext(\Omega^{+})$ be NTA domains, and let $\omega^{\pm}$ be the harmonic measure in $\Omega^{\pm}$ with pole $x^{\pm}\in \Omega^{\pm}$. Assume that $\omega^{+}$ and $\omega^{-}$ are mutually absolutely continuous and $f:=\frac{d\omega^{-}}{d\omega^{+}}$ satisfies $\log f\in \VMO(d\omega^{+})$. Then, there exists $d \in \bN$ (depending on $n$ and the NTA constants) such that the boundary $\d\Omega^{+}$ is $LBWA(\cP_{\Sigma}(d))$ and may be decomposed into sets $\Gamma_{1},...,\Gamma_{d}$ satisfying the following.
\begin{enumerate}
%\item For $1\leq k\leq d$, $\Gamma_{k}$ is the set of $\cF(k)'$ points for $\d\Omega$,
\item For $1\leq k\leq d$, $\Gamma_{k}
=\{\xi\in \d\Omega^{+}: \Tan(\omega^{+},\xi)\subset \cF(k)\}$.
\item $\Gamma_{1}\cup\cdots \Gamma_{d}= \d\Omega^{+}$.
\item $\lim_{r\rightarrow 0}\Theta_{\d\Omega^{+}}^{\cF_{\Sigma}(k)}(\xi,r)=0$ for $\xi\in \Gamma_{k}$. 
%\item \cite{BET17} For $1\leq k\leq d$, $U_{k}:=\Gamma_{1}\cup \cdots \cup \Gamma_{k}$ is an open dense subset of $\d\Omega^{+}$
%$U_{k}\in LBWA(\cP(k)')$ (in particular, $\d\Omega^{+}\in LBWA(\cP(k)')$), and $\Gamma_{k+1}\cup \cdots \cup \Gamma_{d}$ is closed. 
%\item \cite{BET17} $\d\Omega^{+}\backslash \Gamma_{1}$ has upper Minkowski dimension at most $n-1$.
%\item \cite{BET17}  The ``even singular set" $\Gamma_{2}\cup \Gamma_{4}\cup\cdots$ has Hausdorff dimension at most $n-2$.
%\item \cite{BET17}  When $n\geq 2$, $\d\Omega^{+}\backslash \Gamma_{1}$ has Newtonian capacity zero.
\end{enumerate}
\end{theorem}

The work of \cite{BET17} studies the geometric structure of the set as well as the tangent measure structure using the conclusions of the results above. We refer to their work for more details.

%In fact, they are able to replace $\cF(k)'$ and $\cP(k)'$ with smaller classes of measures that are boundaries of two-sided NTA domains, but we state it this way for the sake of exposition. 

%For the non-VMO case, if we just assume mutual absolute continuity of $\omega^{\pm}$ on the boundary of an NTA domain $\Omega^{+}$ with NTA complement $\Omega^{-}$, Kenig, Preiss, and Toro first showed this implied that $\dim \omega^{\pm}\leq n$ \cite{KPT09} and the tangent measures are almost everywhere flat measures, that is, $\Tan(\omega^{\pm},\xi)\subset \cF(1)$ for $\omega^{\pm}$-a.e. $\xi\in \d\Omega^{\pm}$. In \cite{AMT16} and \cite{AMTV16}, the authors, Tolsa, and Volberg showed that mutual absolute continuity implied $n$-rectifiability of the measure.
There are two common threads to all these results: the first is that tangent measures of harmonic measures that are mutually absolutely continuous typically correspond to measures supported along zero sets of harmonic polynomials, and secondly, the use of tangent measures, specifically the connectivity techniques introduced by Preiss in \cite{Pr87}. Using this last technique guarantees, roughly, that if all the tangent measures of a harmonic measure at a point are harmonic polynomials, then they must all be of the same degree (see \cite{KPT09} and \cite{Bad11}). 

\subsection{Blowups of elliptic measures}
In this paper, our objective is to recreate some parts of these results to a class of elliptic measures. Admittedly, there are more results that could be generalized to this setting, like Tsirelson's theorem (using the method of Tolsa and Volberg \cite{TV16-prep}), but we content ourselves with the present results to convey the flexibility of the method.

Let  $\Omega \subset \R^{n+1}$ be open and $A=(a_{ij})_{1 \leq i,j \leq n+1}$ be a matrix with real measurable coefficients in $\Omega$.
We say that $A$ is a {\it  uniformly elliptic matrix} with constant $\Lambda \geq 1$ and write $A\in \cA$ if it satisfies the following conditions:
\begin{align}\label{eqelliptic1}
&\Lambda^{-1}|\xi|^2\leq \langle  A\xi,\xi\rangle,\quad \mbox{ for all $\xi \in\R^{n+1}$},\\
&\langle A \xi,\eta \rangle  \leq\Lambda |\xi| |\eta|, \quad \mbox{ for all $\xi, \eta \in\R^{n+1}$}.\label{eqelliptic2}
\end{align}
 Notice that the matrix is {\it possibly non-symmetric,} and has variable coefficients. We say $L_{A}$ is a {\it uniformly elliptic operator} on $\Omega$, if $L_{A}$ is given by 
$$L_A = -{\rm div} (A(\cdot)\nabla).$$
We will let $\omega_{\Omega}^{A,x}$ denote the $L_{A}$-(elliptic) harmonic measure in $\Omega$ with pole at $x$ (see \cite{HKM} for the definition).  It is clear that the transpose matrix of $A$, which we denote $A^T$, is also uniformly elliptic on $\om$.  Finally, a function $u: \Omega \to \R$ that satisfies the equation $Lu=0$ in the weak sense is called {\it $L$-harmonic}. We will denote $\cC$ the subclass of $\cA$ consisting of matrices with constant entries. 

%Fix $0<\lambda<\Lambda <\infty$ and let $\cA$ be the set of $(n+1)\times (n+1)$ matrices $A$ such that \jonas{Do we need to assume that the matrices $A$ are real?}
%\[
%\lambda |\xi|^{2} \leq  A\xi \cdot \xi\leq \Lambda |\xi|^{2}  \;\; \mbox{ for all }\xi\in \R^{n+1}.\]

%Whenever we say $A$ is an elliptic matrix, we mean that it belongs to $\cA$, and so all results should be understood to have constants depending implicitly on $\lambda$ and $\Lambda$. For a domain $\Omega\subset \R^{n+1}$, we say $L_{A}$ is an elliptic operator on $\Omega$, we mean that $L_{A}$ is given by 
%\[
%L_{A}=-\div A\grad \]
%for a function $A$ with $A(x)\in\cA$ for all $x\in \Omega$. We will let $\omega_{\Omega}^{A,x}$ denote the $L_{A}$-(elliptic) harmonic measure in $\Omega$ with pole at $x$ (see \cite{HKM} for the definition). 

To make sense of tangent measures of an elliptic measure at a point $\xi$ in its support, we need to assume that the coefficients $A$ do not oscillate too much there on small scales. 

\begin{definition}
Let $\Omega\subset \bR^{n+1}$ and and $L_{A}$ an elliptic operator on $\Omega$. For a compact set $E\subset \d\Omega$,  we will say that the coefficients of $L_A$ have {\it vanishing mean oscillation on $K$} with respect to $\Omega$ (or just $L_A \in \,\,\VMO(\om, K)$)  if
\begin{equation}\label{e:boundarylimitK}
\lim_{r\rightarrow 0} \sup_{\xi\in K}\frac{1}{r^{n+1}} \inf_{C\in \cC}\int_{B(\xi,r)\cap \Omega} |A(x)-C|dx=0.
\end{equation}
We also say the coefficients of $L_A$ have {\it $\VMO$ at $\xi\in \d\Omega$} if 
\begin{equation}\label{e:boundarylimit}
\lim_{r\rightarrow 0} \frac{1}{r^{n+1}} \inf_{C\in \cC}\int_{B(\xi,r)\cap \Omega} |A(x)-C|dx=0.
\end{equation}
\end{definition}

Much like the harmonic case, the tangent measures we will obtain are supported on zero sets of elliptic polynomials associated with an elliptic operator with constant coefficients. For a constant coefficient matrix $A$ with real entries, we will denote by $H_{A}$ the set of $L_{A}$-harmonic functions $u$ vanishing at zero, i.e. those functions $u$ for which
\[
\int A\grad u \grad \vphi dx=0 \mbox{ for all }\vphi\in C_{c}^{\infty}(\bR^{n+1}) \mbox{ and } u(0)=0. \]
We also let $P_{A}(k)$ denote the set of $L_{A}$-harmonic polynomials of degree $k$ vanishing at the origin, and $F_{A}(k)\subset P_{A}(k)$ the subset of homogeneous $L_{A}$-harmonic polynomials of degree $k$. When $A=I$, we will simply write $F(k), P(k)$ and $H$ in place of $F_{A}(k),P_{A}(k)$ and $H_{A}$. 

For $h\in H_{A}$, we will write 
\[d\omega_{h}^{A}=-\nu_{\Omega_{h}} \cdot A \nabla h \,d\sigma_{\Sigma_{h}},\]
where $\sigma_{S}$ stands for the surface measure on a surface $S$ and $\nu$ is the outward normal vector at $x \in \d^* \Omega_{h}$, the reduced boundary of $\Omega_{h}$. Again, when $A$ is the identity, we will drop the subscripts and, for example, write $\omega_{h}$ in place of $\omega_{h}^{A}$. For $\cS\subset \cC$, we write 
\[\mathscr{H}_{\cS}=\{\omega_{h}^{A}: h\in H_{A}, \; A\in \cS\}, \;\;\; \cP_{\cS}(k)=\{\omega_{h}^{A}:h\in P_{A}(k), \; A\in \cS\}, \]
\[ \cF_{\cS}(k)=\{\omega_{h}^{A}: h\in F_{A}(k), \; A\in \cS\},\]
\[
\mathscr{H}_{A}=\mathscr{H}_{\{A\}}, \;\; \cP_{A}=\cP_{\{A\}}, \;\; \cF_{A}=\cF_{\{A\}}
\]
and define $\mathscr{H}_{\cS,\Sigma},\cP_{\cS,\Sigma},$ and $\cF_{\cS,\Sigma}$ as we did before. Observe that $\cF_{\cC}(1) = \cF_{A}(1)=\cF(1)$ for any $A\in \cC$.
%%Here $\nu_{\Omega_{h}}(x)$ stands for the measure theoretic  unit  outward normal of $\Omega_{h}$ at $x \in \d^* \Omega_{h}$, the reduced boundary of $\Omega_{h}$. 
%Since $h$ is an $A$-harmonic function and thus, real analytic, which implies that $\Sigma_{h}$ is an $n$-dimensional real analytic variety, $\Omega_h$ is a set of locally finite perimeter and one can prove that $\cH^{n}(\d \Omega_{h} \setminus \d^* \Omega_{h})=0$, where $\cH^{n}$ stands for the $n$-Hausdorff measure. Notice now that $\nu_{\Omega_{h}}(x)$ is defined at $\cH^{n}$-almost every point of $\Sigma_{h}$ and $\sigma_{\Sigma_h}$ is the usual surface measure. For a detailed proof of this see  \cite[p. 21]{AMT16}.  When $A$ is the identity, this reduces to
%\[
%\omega_{h}= -\nu_{\Omega_{h}} \cdot \grad h d\sigma_{\Sigma_{h}}\]
%

%We also let
%\[\cH'=\{\Sigma_{h}: h\in H\}, \;\;\; \cP(k)'=\{\Sigma_{h}:h\in P(k)\}, \;\;\; \cF(k)'=\{\Sigma_{h}: h\in F(k)\}.\]

Our results also recover some LBWA properties implied in previous results if we consider domains satisfying the  Capacity Density Condition (CDC), whose complements also satisfy the CDC (see Definition \ref{CDC} below) and whose associated elliptic measures are doubling. Examples of domains satisfying these conditions are NTA domains or,  by \cite[Theorem 3.1]{Mar79},  any domain $\Omega$ for which there is $s>n-1$ such that $\cH^{s}_{\infty}(B(\xi,r)\cap d\Omega)/r^{s}\geq c>0$ for all $\xi\in \d\Omega$ and $r>0$ is a CDC domain. 
%The CDC is also equivalent to $\Delta$-regularity \cite{Anc86}. 

%Our first result generalizes the results of \cite{KT06} and \cite{Bad11} to elliptic measures.

Our first result extends the work of \cite{KPT09} to the elliptic case, and for domains beyond NTA. 
\begin{main}
	\label{newKPT}
	Let $\Omega^{\pm}\subset \bR^{n+1}$ be two disjoint domains and $L_{A}$ an elliptic operator on $\Omega^{+}\cup \Omega^{-}$ that is in $\VMO(\om^+ \cup \om^-,\xi)$ at $\omega^{+}$-almost every $\xi\in E\subset \d\Omega^{+}\cap \d\Omega^{-}$ with respect to either $\Omega^{\pm}$. Let $\omega^{\pm}=\omega^{L_{A},x_{\pm}}_{\Omega^{\pm}}$ for some $x_{\pm}\in \Omega^{\pm}$. If $\omega^{\pm}$ are mutually absolutely continuous on $E$, then for $\omega^{\pm}$-a.e. $\xi\in E$, $\Tan(\omega^{\pm},\xi)\subset \cF(1)$ and $\dim \omega^{\pm}|_{E}=n$. Furthermore, if $\d\Omega^{\pm}$ are CDC, then $\lim_{r\rightarrow 0}\Theta_{\d\Omega^{+}}^{\cF_{\Sigma}(1)}(\xi,r)=0$ for $\omega^{+}$-a.e. $\xi\in E$. 
\end{main}

Kenig, Preiss, and Toro originally showed this if $\Omega^{\pm}$ were both NTA domains, and the dimension was computed by estimating the Hausdorff dimension directly from above and then using the monotonicity formula of Alt, Cafarelli, and Friedman \cite{ACF84} to estimate it from below. The latter is not available for $L$-harmonic functions  when $L$ satisfies the $VMO$ condition above. For this reason, we use instead  the fact that the tangent measures are all flat, which forces $\omega^{\pm}$ to decay like a planar $n$-dimensional Hausdorff measure on small scales.

%\begin{theorem}\label{t:lbwaH}
%Let $\Omega^{+}$ be a $\Delta$-regular domain with $\Delta$-regular connected complement $\Omega^{-}=\ext(\Omega^{+})$. If $\omega^{-}\in VA_{\infty}(d\omega^{+})$ and $\omega^{+}$ is $C$-doubling, then $\d\Omega^{+}$ is $LBWA(\cP_{d})$ for some $d$ depending on $C$, the $\Delta$-regularity constants, and $n$. Moreover, for each $\xi\in \d\Omega$, there is $k\leq d$ so that $\Tan(\xi,\omega^{+})\subset \cF(k)$ and $\xi$ is a $\cF(k)$ point.
%\end{theorem}
%
%Domains with doubling  harmonic measures include include NTA domains whose compliments are NTA just as in \Theorem{bet}  (see \cite[Lemma 4.2]{JK82} and \cite[Lemma 1]{B87}), but also $\Delta$-regular semi-uniform domains (see \cite[Theorem 1.2]{AH08}). 
%
%
%In \cite{BET17}, \Theorem{bet} is a consequence of their main result, which is a decomposition theorem for sets that are LBWA($\cP(d)$) (see \Theorem{betmain} below). By the results of Badger, Kenig, and Toro, domains satisfying the conditions of the theorem are LBWA($\cP(d)$) for some $d$ depending on the NTA constants and $d$. Using this decomposition theorem, the following corollary is almost immediate.

Assuming a $\VMO$ condition on the interior and exterior  elliptic measures, we can also obtain the results of \cite{KT06} and \cite{Bad11} for elliptic measures on non-NTA domains. We first state a pointwise version of these. 

\begin{main}\label{t:main}
	Let $\Omega^{+}\subseteq \bR^{n+1}$, $\Omega^{-}:=\ext(\Omega^{+})$ be its exterior, and let $L_{A}$ be a uniformly elliptic operator in $\om^+ \cup \om^-$. Denote $\omega^{\pm}$ the $L_{A}$-harmonic measures of $\Omega^{\pm}$ with poles at some points $x^{\pm}\in \Omega^{\pm}$, and assume  that $\omega^{\pm}$ are mutually absolutely continuous with  $f=\frac{d\omega^{-}}{d\omega^{+}}$. If for a fixed $ \xi\in \d\Omega^+ \cap \d\Omega^+ $ it holds that $L_A \in \VMO(\om^+ \cup \om^-, \xi)$, 
	\begin{equation}\label{e:vmo}
	\lim_{r\rightarrow 0} \ps{\avint_{B(\xi,r)} fd\omega^{+}}\exp\ps{-\avint_{B(\xi,r)}\log fd\omega^{+}}=1,
	\end{equation}
	%\lim_{r\rightarrow 0} \avint_{B(\xi,r)} \av{\log \frac{d\omega^{-}}{d\omega^{+}}(\xi) -\avint_{B(\xi,r)}\log \frac{d\omega^{-}}{d\omega^{+}}d\omega^{+}}d\omega^{+}=0,\end{equation}
	and $\Tan(\omega^{+},\xi)\neq \varnothing$, then $\Tan(\omega^{+},\xi)\subset \cF_{\cC}(k)$ for some $k$ and
	\begin{equation}
	\label{e:wdoubling}
	\limsup_{r\rightarrow 0} \frac{\omega^{+}(B(\xi,2r))}{\omega^{+}(B(\xi,r))}<\infty.
	\end{equation}
	If $\Omega^{\pm}$ have the CDC, then additionally
		\[
	\lim_{r\rightarrow 0}\Theta_{\d\Omega^{+}}^{\cF_{\cC,\Sigma}(k)}(\xi,r)=0.\] 
\end{main}

\vvv

It is well known that $\Tan(\omega^{+},\xi)\neq\varnothing$ whenever $\omega^{+}$ satisfies the pointwise doubling condition \eqn{wdoubling}. In our situation, however, we do not assume that, but we get it for free since $\cF_{\cC}(k)$ is compact (see \cite[Lemma 4.10]{Bad11} for the harmonic case and \Theorem{compactdouble} below).

One might have guessed that a pointwise version of \Theorem{bet} would have assumed instead that 
\[\lim_{r\rightarrow 0} \avint_{B(\xi,r)}\av{f-\avint_{B(\xi,r)}\log fd\omega^{+}}d\omega^{+}=0,\]
but this is not enough to imply \Theorem{main}. However, under certain conditions they are equivalent. We will discuss this matter in depth in Section \ref{s:vmo} below.\\

Next, we state a global version.
\begin{main}\label{t:main2}
Let $\Omega^{\pm}\subset \bR^{n+1}$ be two disjoint domains in $\bR^{n+1}$ with common boundary, and let $L_{A}$ be a uniformly elliptic operator in $\om^+ \cup \om^-$ such that $L_A \in \VMO(\om^+ \cup \om^-, \xi)$ at every $\xi\in \d\Omega^+ \cap \d\Omega^+ $. Denote $\omega^{\pm}$ the $L_{A}$-harmonic measures of $\Omega^{\pm}$ with poles at some points $x^{\pm}\in \Omega^{\pm}$.
If $\omega^{+}$ is $C$-doubling, $\omega^{\pm}$ are mutually absolutely continuous, and $\log f=\log \frac{d\omega^{-}}{d\omega^{+}}\in \VMO(d\omega^{+})$, then there is $d$ depending on $n$ and the doubling constant so that, for every compact subset $K\subseteq \d\Omega^{+}$, 
%may be decomposed into sets $\Gamma_{1},...,\Gamma_{d}$, where 
%$$\Gamma_{k}=\{\xi\in \d\Omega^{+}: \Tan(\omega^{+},\xi)\subset  \cF_{\cC}(k)\},$$
 %and $d$ depends only on $n$ and $C$. In fact, for each compact set $K\subset \Gamma_k$,\mih{In the last part, is it really $\mathscr{H}_{\cC}(k)$ or is it $\mathscr{F}_{\cC}(k)$?}\jonas{It would be stronger to say it about F...ok}
\begin{equation}\label{klimsup}
\lim_{r\rightarrow 0}\sup_{\xi\in K} d_{1}(T_{\xi,r}[\omega^{+}],\mathscr{P}_{\cC}(d)) =0.
\end{equation}
If additionally $\Omega^{\pm}$ are CDC domains, then for any compact set $K\subseteq \Gamma_k$,
\[
\lim_{r\rightarrow 0}\sup_{\xi\in K} \Theta_{\d\Omega^{+}}^{\cP_{\cC,\Sigma}(d)}(\xi,r).
\]
That is, $\d\Omega^{+}\in LBWA(\cP_{\cC,\Sigma}(d))$. 
\end{main}
See Section \ref{s:tan} for the definition of $d_{1}(\cdot,\mathscr{P}_{\cC}(d))$, which is essentially a distance between measures and the set $\mathscr{P}_{\cC}(d)$. Many things can be concluded from this LBWA property, see \cite{BET17}.

%\begin{theorem}\label{t:main}
%Let $\Omega^{+}$ be a $\Delta$-regular domain in $\bR^{n+1}$ such that $\Omega^{-}:=\ext\Omega^{+}$ is also a $\Delta$-regular domain. Let $\omega^{\pm}$ denote the harmonic measure of $\Omega^{\pm}$ with pole at some $x^{\pm}\in \Omega^{\pm}$. Let $\xi\in \supp \omega^{+}$ and assume $\omega^{\pm}$ are mutually $A_{\infty}$-equivalent and 
%\begin{equation}\label{e:vmo}
%\lim_{r\rightarrow 0} \avint_{B(\xi,r)} \av{\log \frac{d\omega^{-}}{d\omega^{+}}(\xi) -\avint_{B(\xi,r)}\log \frac{d\omega^{-}}{d\omega^{+}}d\omega^{+}}d\omega^{+}=0,\end{equation}
%If $\Tan(\omega^{+},\xi)\neq \varnothing$, then $\Tan(\omega^{+},\xi)\subset \cF(k)$ for some $k$. In particular, it follows that
%\begin{equation}
%\label{e:wdoubling}
%\limsup_{r\rightarrow 0} \frac{\omega^{+}(B(\xi,2r))}{\omega^{+}(B(\xi,r))}<\infty,
%\end{equation}
%\end{theorem}

The proof of \Theorem{main} involves some useful lemmas about tangent measures that may be of independent interest. Specifically, we refer the reader to \Lemma{newpreiss}. \\

%Observe that by the generalized Gauss-Green's identity, since $h$ is harmonic,
%
%\begin{align*}
%\int \vphi d\omega_{h}
%& =\int_{\Sigma_{h}} -\vphi \nu_{\Omega_{h}}\cdot \grad h d\sigma_{\Sigma_{h}}
%=\int_{\Sigma_{h}} \ps{h \nu_{\Omega_{h}}\cdot \grad \vphi  -\vphi \nu_{\Omega_{h}}\cdot \grad h }d\sigma_{\Sigma_{h}}\\
%& =\int_{\Omega_{h}} (h\triangle \vphi - \vphi\triangle h)dx
%=\int_{\Omega_{h}} h\triangle \vphi dx\end{align*}
%
%where $\sigma_{\Sigma}$ is $n$-dimensional surface measure on $\Sigma$ and $\nu_{h}$ is the outward pointing normal measure theoretic normal, which exists from the fact that $\Omega$ has locally finite perimeter. 

Over the course of working on this manuscript, we also resolved a question left open in \cite{Bad11} (see the discussion on page 16 of \cite{Bad11}).

\begin{mainprop}\label{p:pcompact}
The $d$-cone $\cP(k)$ has compact basis for each $k\in (0,n]$.
\end{mainprop}

See Section \ref{s:tan} for the definition of compact bases. A consequence of this result is that we can improve on the following result of Badger.

\begin{theorem}[\cite{Bad11} Theorem 1.1] \label{t:badtheorem} Let $\Omega\subset \bR^{n+1}$ be an NTA domain with harmonic measure $\omega$ and let $\xi\in \d\Omega$. If $\Tan(\omega,\xi)\subset \cP(d)$, then $\Tan(\omega,\xi)\subset \cF(k)$ for some $k\leq d$.
\end{theorem}

In the proof of this result, Badger relied on the NTA assumption to conclude that $\Tan(\omega,\xi)$ was compact. By using \Proposition{pcompact} (whose proof is rather short), the compactness of homogeneous harmonic polynomials (to which much of the proof of \Theorem{badtheorem} is dedicated), and a connectivity theorem of Preiss, we can improve this by showing that, to get the same conclusion, no a priori information about the geometry of $\omega$ is needed: it need not have been a harmonic measure, let alone one for an NTA domain.

\begin{mainprop}\label{p:taninf}
Let $\omega$ be a Radon measure in $\bR^{n+1}$ and $\xi\in \R^{n+1}$ such that  $\Tan(\omega,\xi)\subset \cP(k)$ for some integer $k$. If $\Tan(\omega,\xi)\cap \cF(k)\neq\varnothing$ for some integer $k$, then $\Tan(\omega,\xi)\subset \cF(k)$.
\end{mainprop}

\subsection{Rectifiability and harmonic measure for uniform domains}
The blowup arguments we use also have an application to studying the relationship between rectifiability and harmonic measure, a subject in which there have been a flurry of results in the last few years.  For simply connected planar domains, the problem of when harmonic measure is absolutely continuous with respect to $\cH^{1}$ is classical. Bishop and Jones showed in \cite{BJ90} that, if $\Omega$ is simply connected, $\omega_{\Omega}^{x}\ll \cH^{1}$ on the subset of any Lipschitz curve intersecting $\d\Omega$. Conversely, Pommerenke showed in \cite{Pom86} that if $\omega_{\Omega}\ll \cH^{1}$ on a subset $E\subset \d\Omega$, then that set can be covered by Lipschitz graphs up to a set of harmonic measure zero. In fact, a much earlier result of the Riesz brothers says that any Jordan domain has harmonic measure and $\cH^{1}$ mutually absolutely continuous if and only if the boundary is rectifiable (see \cite{RR16} or \cite[Chapter VI.1]{Harmonic-Measure}).

In higher dimensions, the problem is trickier. There are some simple examples of simply connected domains  $\Omega\subset \R^{n+1}$ with $n$-rectifiable boundaries of finite $\cH^{n}$-measure so that either $\omega_{\Omega}\not \ll \cH^{n}$ or $\cH^{n}\not \ll \omega_{\Omega}$ (see \cite{Wu86}, \cite{Zie74}). David and Jerison showed in \cite{DJ90} that mutual absolute continuity occured for NTA domains with Ahlfors-David regular boundaries. Building on that, Badger showed in \cite{Bad12} that $\cH^{n}\ll \omega_{\Omega}$ if $\Omega$ was an NTA domain whose boundary simply had locally finite $\cH^{n}$ measure, although we showed with Tolsa that the converse relation $\omega_{\Omega} \ll \cH^{n}$ could be false for such domains \cite{AMT17}. 

However, in \cite{AHMMMTV16}, along with Hofmann, Martell, Mayboroda, Tolsa, and Volberg, we showed that for {\it any}  domain $\Omega\subset \R^{n+1}$ and $E\subset \d\Omega$ with $\omega_\Omega(E)>0$ and $\cH^{n}(E)<\infty$, if $\omega\ll \cH^{n}$ on $E$, then $E$ may be covered up to $\omega$ measure zero by Lipschitz graphs. By a theorem of Wolff, harmonic measure in the plane lies on a set of $\sigma$-finite $\cH^{1}$-measure, and so the assumption that $\cH^{1}(E)<\infty$ is unnecessary in this case (although very necessary in higher dimensions due to the existence of Wolff snowflakes).  With Akman, we developed a converse for domains $\Omega\subset \R^{n+1}$ with {\it big complements}, meaning 
\begin{equation}
\cH^{n}_{\infty}(B(\xi,r)\backslash\Omega)\geq cr^{n} \mbox{ for all }\xi\in \d\Omega \mbox{ and } 0<r<\diam \d\Omega.
\end{equation} 
We showed that, for such domains, $\omega_{\Omega}\ll \cH^{n}$ on the subset of any $n$-dimensional Lipschitz graph \cite{AAM16}, and hence, for these domains, we know that absolute continuity is equivalent to rectifiability of harmonic measure (versus rectifiability of the boundary). 

There are fewer positive results concerning absolute continuity and rectifiability of {\it elliptic} harmonic measures. Even in the case of the half plane, without some extra assumptions on the behavior of the elliptic coefficients, elliptic measure can be singular \cite{CFK81,Swe92,Wu94}, and some sort of Dini condition on the coefficients near the boundary is needed \cite{FJK84,FKP91}. In \cite{KP01}, for example, Kenig and Pipher considered the following condition.

\begin{definition}\label{d:KP}
	Let $\delta(x)=\mbox{dist}(x,\partial\Omega)$. We will say that an elliptic operator $L=-\div {A}\grad$ satisfies the {\it Kenig-Pipher condition (or KP-condition)} if ${A}=(a_{ij}(x))$ is a uniformly elliptic real matrix that has distributional derivatives such that
\begin{equation}\label{eq:KP}
	\ve_{\Omega}^{{L}}(z):=\sup\{\delta(x)|\grad a_{ij}(x)|^{2}: x\in B(z,\delta(z))/2,\;\; 1\leq i,j\leq n+1\}
\end{equation}
	is a Carleson measure in $\Omega$, by which we mean that for all $x\in \d\Omega$ and $r\in (0,\diam \d\Omega)$,
	\[
	\int_{B(x,r)\cap \Omega}\ve_{\Omega}^{{L}}(z)dz\leq Cr^{n}.\]
\end{definition}

In \cite{KP01}, they showed that for Lipschitz domains in $\R^{n+1}$, elliptic operators satisfying  the $KP$-$condition$ gave rise to elliptic measures which were $A_{\infty}$-equivalent to surface measure. In fact, it was proved in \cite{HMT16} that the same result can be obtained under the following more general assumptions on the coefficients:
 \begin{equation}\label{eq:KP-bis}
    (\widetilde{KP})=
    \begin{cases}
      \nabla a_{ij} \in \Lip_{loc}(\om),  \\
      \| \delta_{\om} |\nabla a_{ij}|\|_{L^\infty(\om)}< \infty,\\
      \delta(x)|\grad a_{ij}(x)|^{2}\,\,\text{is a Carleson measure},
    \end{cases},
\end{equation}
for $ 1\leq i,j\leq n+1$. 
Akman, Badger, Hofmann, and Martell observed in \cite[Section 3.2]{ABHM17} that, using the same arguments in \cite{DJ90}, this result can be extended to NTA domains with Ahlfors-David regular boundaries. They used this fact to show that, on a {\it uniform domain} $\Omega$ {(see Definition \ref{d:uniform} below)} with Ahlfors-David regular boundary, if $L_{A}$ is a {\it symmetric} elliptic operator satisfying a local  $L^1$ version of \eqref{eq:KP}, i.e., $A \in \Lip_{\loc}(\om)$ and $\sup\{|\grad a_{ij}(x)|: x\in B(z,\delta(z))/2,\;\; 1\leq i,j\leq n+1\}|$ is a Carleson measure with Carleson constant depending on the ball, then $\cH^{n}\ll \omega^{L}_{\Omega}$ implies $n$-rectifiability of the boundary.%\footnote{In fact, it is enough to ask for an $L^1$ version of \eqref{eq:KP-bis}, i.e., $|\grad a_{ij}(x)|$ to be a Carleson measure.}

\vv

Using our blowup arguments, we can obtain the following improvement. 
\begin{main}\label{t:newABHM}
	Let $\Omega\subset \R^{n+1}$ be a uniform CDC domain so that $\cH^{n}|_{ \d\Omega}$ is locally finite. Let $\omega_{\Omega}^{L_{A}}$ be the harmonic measure associated to a (possibly non-symmetric) elliptic operator satisfying \eqref{eqelliptic1} and \eqref{eqelliptic2}. Let $E\subseteq \d\Omega$ be a set with $\cH^{n}(E)>0$ such that $\cH^{n}\ll \omega_{\Omega}^{L_{A}}$ on $E$ and for $\cH^{n}$-a.e. $\xi\in E$,
	\[
	\theta_{\d\Omega,*}^{n}(\xi,r):= \liminf_{r\rightarrow 0}\frac{ \cH^{n}(B(\xi,r)\cap \d\Omega)}{(2r)^{n}}>0
	\]
	and $A$ has vanishing mean oscillation at $\xi$. Then $E$ is $n$-rectifiable. 
\end{main}
Surprisingly, to get this improvement requires a very different set of techniques than originally considered in \cite{ABHM17}.

Having $\VMO$ coefficients $\cH^{n}$-a.e. on $\d\Omega$ is natural as it is implied by the Carleson condition considered in \cite{ABHM17} and \cite{KP01} by the following proposition:

\begin{mainprop}\label{carlesonprop}
	Let $\Omega\subset \bR^{n+1}$ be a domain and suppose that $A$ is an elliptic matrix satisfying \eqref{eqelliptic1} and \eqref{eqelliptic2} such that $A \in \Lip_{\loc}(\om)$ and, for some ball $B_{0}$ centered on $\d\Omega$,
	
	 \begin{equation}\label{e:notcarl}
	 \int_{B_{0}}\delta(x)|\grad a_{ij}(x)|^{2}dx<\infty.
	 \end{equation}  
	 Then $L_A \in \VMO(\om, \xi)$ for $\cH^{n}$-a.e. $\xi \in B_{0}\cap \d\Omega$. 
\end{mainprop}

\vv

\noindent{\bf Discussion of related results.} Near the completion of this work, we learned that Toro and Zhao simultaneously had proved that $\cH^{n}\ll \omega_{\Omega}$ implies rectifiability of the boundary if $\Omega\subseteq \R^{n+}$ is a uniform domain with Ahlfors-David $n$-regular boundary and the elliptic coefficients are in $W^{1,1}(\Omega)$ \cite{TZ17}. They also exploit the vanishing oscillation of the coefficients at almost every boundary point (which they show is implied by the $W^{1,1}$ condition) in the context of uniform domains, though their proof is distinct by their use of pseudo-tangents and stopping-time arguments.

\vvv

\noindent{\bf Acknowldegements.} This project was started while we were post-docs under X. Tolsa's ERC grant 320501 of the European Research Council (FP7/2007-2013) between 2015 and 2016, and also during mutual visits between their home institutions of Universitat Aut\`onoma de Barcelona, the Basque Center for Applied Mathematics, Universidad del Pa\' is Vasco and the University of Edinburgh. The authors are greatful to these institutions for their hospitality, to Xavier Tolsa for his helpful discussions and encouragement on the project, Jos\'e Conde for his clear presentation of \cite{Bad11} in our working seminar at the UAB, and to Tatiana Toro for her comments on our paper.

\section{Notation}

We will write $a\lesssim b$ if there is $C>0$ so that $a\leq Cb$ and $a\lesssim_{t} b$ if the constant $C$ depends on the parameter $t$. We write $a\approx b$ to mean $a\lesssim b\lesssim a$ and define $a\approx_{t}b$ similarly.

\section{Tangent Measures}\label{s:tan}

\subsection{Cones and Compactness}
 
Given two Radon measure $\mu$ and $\sigma$, we set
$$F_{B}(\mu,\sigma) = \sup_f \int f\,d(\mu-\sigma),$$
where the supremum is taken over all the $1$-Lipschitz functions supported on $B$. 
For $r>0$, we write
\[
F_{r}(\mu,\nu)= F_{\overline B(0,r)},\qquad
F_{r}(\mu)=F_{r}(\mu,0)=\int (r-|z|)_{+} d\mu.\]

{
A set of Radon measures $\cM$ is a {\it $d$-cone} if $cT_{0,r}[\mu]\in \cM$ for all $\mu\in\cM$, $c>0$ and $r>0$. We say a $d$-cone has {\it closed (resp. compact) basis} if its basis $\{\mu\in \cM: F_{1}(\mu)=1\}$ is closed (resp. compact) with respect to the weak topology.}

 For a $d$-cone $\cM$, $r>0$, and $\mu$ a Radon measure with $0<F_{r}(\mu)<\infty$, we define the {\it distance} between $\mu$ and $\cM$ as 
%\xavi{I have modified the definition of $d_r$ to guaranty the validity of the identities in \rf{e:ojF}}
\[
d_{r}(\mu,\cM)=\inf\ck{F_{r}\ps{\frac{\mu}{F_{r}(\mu)},\nu}: \nu\in \cM, F_{r}(\nu)=1 }.
%\mbox{ and }F_{r}(\nu)=1}.
\]

\begin{lemma}[\cite{KPT09} Section 2]\label{preiss}
Let $\mu$ be a Radon measure in $\bR^{n+1}$ and $\cM$ a $d$-cone. For $\xi\in \bR^{n+1}$ and $r>0$,
\begin{enumerate}
\item $\displaystyle T_{\xi,r}[\mu](B(0,s))=\mu(B(\xi,sr))$,
\item $\displaystyle \int f \, d T_{\xi,r}[\mu] = \int f\circ T_{\xi,r} \, d\mu$,
\item $ \displaystyle F_{B(\xi,r)}(\mu)=r F_{1}(T_{\xi,r}[\mu])$,
\item $\displaystyle F_{B(\xi,r)}(\mu,\nu)=rF_{1}(T_{\xi,r}[\mu],T_{\xi,r}[\nu])$,
\item $\displaystyle \mu_{i}\rightarrow \mu$ weakly if and only if $F_{r}(\mu_{i},\mu)\rightarrow 0$ for all $r>0$,
\item $\displaystyle d_{r}(\mu,\cM)\leq 1$,
\item $\displaystyle d_{r}(\mu,\cM)=d_{1}(T_{0,r}[\mu],\cM)$,
\item if $\displaystyle \mu_{i}\rightarrow \mu$ weakly and $\displaystyle F_{r}(\mu)>0$, then $\displaystyle d_{r}(\mu_{i},\cM)\rightarrow d_{r}(\mu,\cM)$. 
\end{enumerate}
\end{lemma}

%\begin{lemma}[\cite{Pr87} Proposition 1.11]
%Let $\{\mu_i\}$ be a sequence of Radon measures such that $\limsup \mu_{i}(B(0,r))<\infty$ for all $r>0$. Then $\mu_{i}$ converges weakly to a measure $\mu$ if and only if $F_{r}(\mu_{i},\mu)\rightarrow 0$ for every $r>0$.
%\end{lemma}

%\begin{definition} \cite[Section 2]{Pr87} 
%\begin{enumerate}[(a)]
%\item A set $\cM$ of non-zero Radon measures in $\bR^{n+1}$ is a {\it cone} if $c\mu\in \cM$ whenever $\mu\in \cM$ and $c>0$.
%\item A cone $\cM$ is a {\it $d$-cone} if $T_{0,r}[\mu]\in \cM$ for all $\mu\in \cM$ and $r>0$.
%\item The {\it basis} of a $d$-cone $\cM$ is the set $\{\Psi\in \cM: F_{1}(\Psi)=1\}$. 
%\end{definition}
%
%%\item The {\it basis} of a $d$-cone $\cM$ is the set $\{\mu\in \cM: F_{1}(\mu)=1\}$.
%\item\end{enumerate}
%\end{definition}

\begin{lemma}[\cite{KPT09} Remark 2.13] \label{l:closed} A $d$-cone $\cM$ of Radon measures in $\bR^{n+1}$ has a closed basis if and only if it is a relatively closed subset of the non-zero Radon measures in $\bR^{n+1}$. 
\end{lemma}
\begin{proof}
One direction is obvious, so suppose $\cM$ has closed basis and $\mu_{i}\in \cM$ converges weakly to some non-zero Radon measure $\mu$. Then $F_{r}(\mu)>0$ for some $r>0$. The set $\{\nu\in \cM:F_{1}(\nu)=1\}$ is closed by assumption, and since $\cM$ is a $d$-cone, the set $\{\nu\in \cF: F_{r}(\nu)=1\}$ is also closed. Hence, since $\mu_{i}/F_{r}(\mu_{i})\rightarrow \mu/F_{r}(\mu)$, $\mu/F_{r}(\mu)\in \cM$, and thus $\mu\in \cM$.
\end{proof}

\begin{lemma}\label{l:dM}
If $\mu$ is a nonzero Radon measure and $\cM$ is a $d$-cone with closed basis, then $\mu\in \cM$ if and only if $d_{r}(\mu,\cM)=0$ for all $r>0$ for which $F_{r}(\mu)>0$. 
\end{lemma}

\begin{proof}
Suppose $d_{r}(\mu,\cM)=0$ for all $r>0$ for which $F_{r}(\mu)>0$. For $j\in \bN$ large enough, we can find a sequence $\mu_{j,k}\in \cM$ such that 
\begin{equation}\label{flims}
F_{j}(\mu_{j,k})=1 \;\; \mbox{and} \;\; \lim_{k\rightarrow\infty} F_{j}\ps{\frac{\mu}{F_{j}(\mu)},\mu_{j,k}}=0.
\end{equation}
In particular, we can pass to a subsequence so that $\mu_{j,k}$ converges weakly in $B(0,j)$ to a measure $\mu_{j}$ supported in $B(0,j)$ with $F_{j}(\mu_{j})=1$. In view of \eqref{flims}, the latter implies that $\mu = F_{j}(\mu)\, \mu_{j} $ in $B(0,j)$, and thus, 
$$ F_{j}(\mu)\, \mu_{j}\warrow \mu.$$ 
Since, $\mu_{j,k}\rightharpoonup \mu_{j}$ and $F_{j}(\mu)\neq 0$ for $j$ large, we can pick $k_{j}$ so that 
\[
F_{j}(\mu_{j,k_{j}},\mu_{j})<\frac{1}{jF_{j}(\mu)}. \]
In particular, for any $r>0$ and $j>r$,
\begin{align*}
F_{r}\ps{\mu_{j,k_{j}} F_{j}(\mu),\mu}
& \leq F_{j}\ps{\mu_{j,k_{j}}F_{j}(\mu),\mu}\\
& \leq F_{j}\ps{\mu_{j,k_{j}}F_{j}(\mu),\mu_{j}F_{j}(\mu)}+F_{j}\ps{\mu_{j}F_{j}(\mu),\mu}\\
& <\frac{1}{j} +F_{j}\ps{\mu_{j}F_{j}(\mu),\mu}\rightarrow 0.
\end{align*}

Thus, $\mu_{j,k_{j}} F_{j}(\mu)  \rightharpoonup \mu$. 
%Again, since $\cM$ is a $d$-cone with closed basis, 
%we can pick $\mu_{j}\in \cM$ so that $F_{j}(\mu_{j})=1$ and $F_{j}(\mu/F_{j}(\mu),\mu_{j})=0$, so in particular, $\mu= F_{j}(\mu)\mu_{j}$ in $B(0,j)$. Thus, $F_{j}(\mu)\mu_{j}\rightarrow \mu$. 
By \Lemma{closed}, $\cM$ is closed, and since we have $\mu_{j,k_{j}} F_{j}(\mu)\in \cM$ for all $j$, this implies $\mu\in \cM$. The other implication is trivial.
\end{proof}

\begin{theorem}[\cite{Pr87} Corollary 2.7] Let $\mu$ be a Radon measure on $\bR^{n+1}$, and $\xi\in \supp \mu$. Then $\Tan(\mu,\xi)$ has compact basis if and only if 
\begin{equation}
\label{e:compactdouble}
\limsup_{r\rightarrow 0} \frac{\mu(B(\xi,2r))}{\mu(B(\xi,r))}<\infty.
\end{equation}
In this case, for any $\nu\in \Tan(\mu,\xi)$, it holds that $0\in \supp \nu$ and 
\[
\frac{\nu(B(0,2r))}{\nu(B(0,r))}\leq \limsup_{\rho \rightarrow 0} \frac{\mu(B(\xi,2\rho))}{\mu(B(\xi,\rho))},\mbox{ for all }r>0.
\] 
\label{t:compactdouble}
\end{theorem}

\begin{lemma} \cite[Theorem 14.3]{Mattila} Let $\mu$ be a Radon measure on $\bR^{n+1}$. If $\xi\in \bR^{n+1}$ and \eqn{compactdouble} holds, then every sequence $r_{i}\downarrow 0$ contains a subsequence such that 
\begin{equation}\label{eq:tang-doubl}
\frac{T_{\xi,r_{j}}[\mu]}{\mu(B(\xi,r_{j}))}  \rightharpoonup \nu,
\end{equation}
for some measure $\nu\in \Tan(\mu,\xi)$.
\end{lemma}\label{l:tanexist}
 Having tangent measures that arise as limits of the form \eqref{eq:tang-doubl} is very convenient, but this limit does not always converge weakly to something. This may happen if $\mu$ is not pointwise doubling at the point $a$. However, all tangent measures are at least dilations of tangent measures arising in this way.

\begin{lemma}\label{l:nicetan}
Let $\mu$ be a nonzero Radon measure, $\xi\in \supp \mu$, and $\nu\in \Tan(\mu,\xi)$. Then there are $\rho_{j}\downarrow 0$ and $\rho,c>0$ so that 
$$
\frac{T_{\xi,\rho_{j}}[\mu]}{\mu(B(\xi,\rho_{j}))}\rightharpoonup c \, T_{0,\rho}[\nu] \quad \textup{and}\quad c\,T_{0,\rho}[\nu](\bB)>0.
$$ 
\end{lemma}

\begin{proof}
Let $c_{j}>0$ and $r_{j}\downarrow 0$ be so that $\mu_{j}=c_{j} T_{\xi,r_{j}}[\mu]\rightharpoonup \nu$. Since $\nu\neq 0$, there is $r>0$ so that $\nu(B(0,r/2))>0$. Thus, for $R>r$,
\begin{align*}
\limsup_{j\rightarrow\infty} \frac{T_{\xi,r_{j}}[\mu](B(0,R))}{\mu(B(\xi,r_{j}r))}
=&\limsup_{j\rightarrow\infty} \frac{c_{j} T_{\xi,r_{j}}[\mu](B(0,R))}{c_{j} T_{\xi,r_{j}}[\mu](B(0,r))}\\
&\leq \frac{\nu(\cnj{B(0,R)})}{\nu(B(0,r))}<\infty.
\end{align*}
Therefore, we can find a subsequence so that 
$$\nu_{j}:=\frac{T_{\xi,r_{j}}[\mu]}{\mu(B(\xi,rr_{j}))}\rightharpoonup \lambda,$$ 
for some measure $\lambda$. Hence, we  have that
$$
T_{0,r}[\nu_{j}]=\frac{T_{\xi,r_{j}r}[\mu]}{\mu(B(\xi,rr_{j}))} \rightharpoonup T_{0,r}[\lambda].
$$
Let $\vphi$ be any compactly supported continuous function so that $\int \vphi d\lambda>0$ and $\int \vphi d\nu>0$. Then
\[
c_{j}\, \mu(B(\xi,rr_{j}))=\frac{\int \vphi\, dc_{j} T_{\xi,r_{j}}[\mu]}{\int \vphi \, d \nu_{j}}\rightarrow \frac{\int \vphi d\nu}{\int \vphi d\lambda}=:a>0.\]
Thus, 
\[\nu = \lim_{j\rightarrow \infty} c_{j}T_{\xi,r_{j}}[\mu]
=\lim_{j\rightarrow \infty} c_{j} \mu(B(\xi,rr_{j}))
 \frac{T_{\xi,r_{j}}[\mu]}{\mu(B(\xi,rr_{j}))}
=a \lambda .\]
In particular, $T_{0,r}[\nu]=a \,T_{0,r}[\lambda]$. Moreover, $a^{-1}T_{0,r}[\nu](\bB)=a^{-1}\nu(r\bB)>0$ by our choice of $r$. Thus, the lemma follows by setting $\rho_{j}=rr_{j}$, $c=1/a$, and $\rho=r$. 
\end{proof}

%\begin{theorem}\label{t:main}
%Let $\omega$ be a Radon measure in $\bR^{n+1}$ and suppose $\Tan(\omega,\xi)\subset \cP(k)$ for some integer $k$. Then $\Tan(\omega,\xi)\subset \cF(k)$.
%\end{theorem}

\begin{proposition}[\cite{Pr87} Proposition 2.2] Let $\cM$ be a $d$-cone. Then $\cM$ has compact basis if and only if for every $\lambda>1$ there is $\tau>1$ such that 
\begin{equation}
F_{\tau r}(\Psi)\leq \lambda F_{r}(\Psi) \mbox{ for every }\Psi\in \cM \mbox{ and }r>0.
\label{e:compact}
\end{equation}
 In this case, $0\in \supp \Psi$ for all $\Psi\in \cM$.
\label{p:compact}
\end{proposition}

\begin{theorem} \label{t:ttt}
\cite[Theorem 14.16]{Mattila}
Let $\mu$ be a Radon measure on $\bR^{n+1}$. For $\mu$-almost every $x\in \bR^{n+1}$, if $\nu\in \Tan(\mu,x)$, the following hold:
\begin{enumerate}
\item $T_{y,r}[\nu]\in \Tan(\mu,x)$ for all $y\in \supp \nu$ and $r>0$.
\item $\Tan(\nu,y)\subset \Tan(\mu,x)$ for all $y\in \supp \nu$.
\end{enumerate}
\end{theorem}

\subsection{Connectivity of cones}
The main tool from \cite{KPT09} and \cite{Bad11} is the following ``connectivity" lemma, which was originally shown in \cite[Corollary 2.16]{KPT09} under the assumption that $\cM$ had compact basis. For our purposes, we need to remove this assumption. 

\begin{lemma}\label{l:newpreiss}
Let $\cF$ and $\cM$ be $d$-cones and assume $\cF$ has compact basis. Furthermore, suppose that there is $\ve_{0}>0$ such that for $\mu\in \cM$, if there is $r_{0}>0$ so that $d_{r}(\mu,\cF)\leq \ve$ for all $r\geq r_{0}$, then $\mu\in \cF$. For a Radon measure $\eta$ and $x\in \supp \eta$, if $\Tan(\eta,x)\subset \cM$ and $\Tan(\eta,x)\cap \cF\neq\varnothing$, then $\Tan(\eta,x)\subset \cF$. %and there is $\alpha>0$ so that $\mu(B(0,r))\leq Cr^{\alpha}$ for all $r\geq r_{0}$.
\end{lemma}

We will first require some lemmas.

\begin{lemma}\label{l:pstep}
Let $\cF$ be a $d$-cone with compact basis. There is $\beta>0$ depending only on $\cF$ so that the following holds. Suppose $\omega$ is a Radon measure in $\bR^{n+1}$, $\xi\in \supp\omega$, $\Tan(\omega,\xi)\cap \cF\neq\varnothing$ and 
$$\limsup_{r\rightarrow 0} d_{r_{0}}(T_{\xi,r}[\omega],\cF)\geq \ve_{0}>0,\,\, \textup{for some}\,\,r_{0}>0.$$
Then for $\ve<\ve_{0}$ small enough, we may find $\mu\in \Tan(\omega,\xi)\backslash \cF$ so that
\begin{enumerate}
\item $d_{r_{0}}(\mu,\cF)=\ve$, 
\item $d_{r}(\mu,\cF)\leq \ve$ for all $r>r_{0}$, and 
\item $\mu(B(0,r))\leq r^{\beta} \, \mu(B(0,4r_0))$ for all $r\geq r_{0}$.
\end{enumerate}
\end{lemma}

This is an adaptation of the proof of \cite[Corollary 2.16]{KPT09}, but with some extra care. 

\begin{proof}
Without loss of generality, we will assume $r_{0}=1$. Let $c_{j}>0$ and $r_{j}\downarrow 0$ be such that $c_{j}T_{\xi,r_{j}}[\omega]\rightarrow \nu\in \cF$. Since $\cF$ is compact, by Proposition \ref{p:compact}, $0\in \supp \nu$ and so $\nu(\bB)>0$.  Thus, by Lemma \ref{preiss} (5), $c_{j}T_{\xi,r_{j}}[\omega](\bB)>0$ for $j$ large. By Lemma \ref{preiss} (8), we have that, given $\ve>0$, for $j$ large enough,
\begin{equation}\label{e:distrj}
d_1(T_{\xi,r_{j}}[\omega],\cF)=d_1(c_{j}T_{\xi,r_{j}}[\omega],\cF)<\ve.
\end{equation}
Note that $0\in \supp T_{\xi,r_{j}}[\omega]$ since $\xi\in \supp \omega$, and so there is no accidental dividing by zero in the definition of $d_{1}$. By assumption, there is also $s_{j}\downarrow 0$ so that
\begin{equation}\label{e:distsj}
d_1(T_{\xi,s_{j}}[\omega],\cF)>\ve.
\end{equation}
We can assume $s_{j}<r_{j}$ by passing to a subsequence. Then by \eqn{distrj} and \eqn{distsj}, let $\rho_{j}\in (s_{j},r_{j})$ be the maximal number such that
\begin{equation}\label{e:distrhoj}
d_1(T_{\xi,\rho_{j}}[\omega],\cF)=\ve.
\end{equation}
Then, by the maximality of $\rho_{j}$,
\begin{equation}\label{e:t>rho}
\sup_{t\in [\rho_{j},r_{j}]} d_1(T_{\xi,t}[\omega],\cF)\leq \ve.
\end{equation}
We claim $\rho_{j}/r_{j}\rightarrow 0$. If not, then since $\rho_{j}/r_{j}\leq 1$, we may pass to a subsequence so that $\rho_{j}/r_{j}\rightarrow t\in (0,1)$, and so 
\[
c_{j} T_{\xi,\rho_{j}}[\omega]=T_{0,\rho_{j}/r_{j}}\bk{c_{j}T_{\xi,r_{j}}[\omega]}\rightarrow T_{0,t}[\mu]\in \cF,\]
which contradicts \eqn{distrhoj}. Thus, $\rho_{j}/r_{j}\rightarrow 0$, and so \eqn{t>rho} implies that for $\alpha\geq 1$, if $j$ is large enough, we have $1\leq  \alpha<r_{j}/\rho_{j}$. If  $\omega_{j}=T_{\xi,\rho_{j}}[\omega]$, then by Lemma \ref{preiss} (7), it holds
\begin{equation}\label{e:ojF}
d_{\alpha}(\omega_{j},\cF)=d_{\alpha}(T_{\xi,\rho_{j}}[\omega],\cF)
=d_{1}(T_{\xi,\alpha \rho_{j}}[\omega],\cF)
\stackrel{\eqref{e:t>rho}}{\leq} \ve,
\end{equation}
which by \eqref{e:distrhoj} implies that
\begin{equation}\label{e:limsupdrwj}
d_{1}(\omega_{j},\cF)=\ve>0 \;\; \mbox{and} \;\; 
\limsup_{j\rightarrow\infty} d_{r}(\omega_{j},\cF)\leq \ve\mbox{ for }r>1.
\end{equation}

For $r\geq 1$, let $\mu_{j,r}\in \cF$ be such that $F_{\tau r}(\mu_{j,r})=1$ and 
\[
F_{\tau r}\ps{\frac{\omega_{j}}{F_{\tau r}(\omega_{j})},\mu_{j,r}}
<\frac{3}{2}d_{\tau r}(\omega_{j},\cF).
\]
By \eqref{e:limsupdrwj}, for $j$ large enough,
\begin{equation}\label{eq:s3Fr2e}
F_{r}\ps{\frac{\omega_{j}}{F_{\tau r}(\omega_{j})},\mu_{j,r}}
\leq F_{\tau r}\ps{\frac{\omega_{j}}{F_{\tau r}(\omega_{j})},\mu_{j,r}}
<\frac{3}{2}d_{\tau r} (\omega_{j},\cF)< 2\ve.
\end{equation}
Since $\cF$ has compact basis, by \Proposition{compact} with $\lambda=2$, there is $\tau>1$ depending only on $\cF$ so that \eqn{compact} holds for $\cM=\cF$.
Thus, if $\ve<1/8$, by the triangle inequality for $F_{r}$ and \eqref{eq:s3Fr2e},
\begin{equation}\label{e:F-e}
\frac{F_{r}(\omega_{j})}{F_{\tau r}(\omega_{j})}
\geq  F_{r}(\mu_{j,r})-2\ve
\geq \frac{1}{2} F_{\tau r}(\mu_{j,r})-2\ve 
=\frac{1}{2}-2\ve>\frac{1}{4}.
\end{equation}
Hence, for any $r\geq 1$,
\[ F_{\tau r}(\omega_{j})\leq 4F_{r}(\omega_{j}).\]
Set $\mu_{j}= \omega_{j}/F_{1}(\omega_{j})$. Then iterating the above inequality and letting $j\rightarrow\infty$, we get that for all $\ell\in \bN$,
\[
\limsup_{j\rightarrow\infty} F_{\tau^{\ell}}(\mu_{j})\leq 4^{\ell}.\]
This implies that we can pass to a subsequence so that $\mu_{j}$ converges weakly to a measure $\mu\in \Tan(\omega,\xi)$. In particular, for $r\geq 1$, since $F_{1}(\mu_{j})=1$, we may compute
\[
d_{1}(\mu,\cF)=\lim_{j\rightarrow\infty} d_{1}(\mu_{j},\cF)=\lim_{j\rightarrow\infty}  d_{1}(\omega_{j},\cF)\stackrel{\eqref{e:limsupdrwj}}{=}\ve,
\]
\[ 
d_{r}(\mu,\cF)=\lim_{j\rightarrow\infty} d_{r}(\mu_{j},\cF)=\lim_{j\rightarrow\infty}  d_{r}(\omega_{j},\cF)\stackrel{\eqref{e:limsupdrwj}}{\leq}\ve,
\]
and 
\begin{equation}\label{eq:sec3iterationFell}
\tau^{\ell}\mu(B(0,\tau^{\ell}))\leq F_{2\tau^{\ell} }(\mu)\leq 4^{\ell} F_{2}(\mu)\mbox{ for all }\ell\in \bN.
\end{equation}
Since $\tau >1$, for any $r \geq 1$, there exists $\ell>0$ such that $\tau^{\ell-1} < r \leq\tau^\ell$. If $\tau \in (1,4)$, \eqref{eq:sec3iterationFell} implies 
$$ \tau^\ell \,\mu(B(0, \tau^\ell)) \leq \tau^\alpha r^{\alpha} \mu(\overline{B(0,2)}),$$
where $\alpha = \frac{1}{\log_4 \tau} \in (1, \infty)$ and we used that $ 4^\ell = \tau^{\ell \alpha}$. Therefore,
$$\mu(B(0,r))\leq \tau^{\alpha-\ell}\, r^{\alpha} \,\mu(\overline{B(0,2)}),$$
and notice that  $\tau^{\alpha-\ell} \leq 1$ whenever $\tau^\ell \geq 4$, i.e., the constant is independent of $\tau$. In the case that $1 \leq r \leq \tau^\ell < 4$, we simply use that $B(0,r)  \subset B(0,4)$ to conclude that 
$$\mu(B(0,r))\leq  \mu({B(0,4)}).$$
If $\tau \geq 4$, then \eqref{eq:sec3iterationFell} trivially gives
$$\tau^ \ell \mu(B(0, \tau^\ell)) \leq 4^\ell \mu(\overline{B(0,2)}) \leq \tau^ \ell\, \mu(\overline{B(0,2)}),$$
which can only be true if $r\leq \tau^\ell \leq 2$. Thus, $B(0,r)  \subset B(0,2)$ and (3) readily follows.
\end{proof}

\begin{corollary}\label{c:pstep}
Let $\cF$ be a $d$-cone with compact basis. There is $\beta>0$ so that the following holds. Suppose $\mu$ is a Radon measure in $\bR^{n+1}$ so that 
\begin{enumerate}
\item $\Tan(\mu,\xi)\cap \cF\neq\varnothing$ and
\item $\Tan(\mu,\xi)\backslash \cF\neq\varnothing$. 
%\item $0\in \supp \nu$ for all $\nu\in \Tan(\mu,\xi)$. 
\end{enumerate}
Then there is $r_{0}>0$ so that for any $\ve>0$ sufficiently small, the conclusion of Lemma \ref{l:pstep} holds. 
%we may find $\nu\in \Tan(\mu,\xi)\backslash \cF$ so that $d_{r_{0}}(\nu,\cF)=\ve$, $d_{r}(\nu,\cF)\leq \ve$ for all $r>r_{0}$, and $\nu(B(0,r))\leq Cr^{\alpha}$ for all $r\geq r_{0}$.
\end{corollary}

\begin{proof}
Let $\nu\in \Tan(\mu,\xi)\backslash \cF$. By \Lemma{dM}, there exists $r_{0}>0$ so that $F_{r_{0}}(\nu)>0$ and $d_{r_{0}}(\nu,\cF)>0$. Let $c_{j}>0$ and $r_{j}\downarrow 0$ be so that $c_{j}T_{\xi,r_{j}}[\mu]\rightarrow \nu$. Then, for $j$ large enough, $d_{r_{0}}(T_{\xi,r_{j}}[\mu],\cF)>d_{r_{0}}(\nu,\cF)/2>0$. The corollary now follows from \Lemma{pstep} with $\ve_0=d_{r_{0}}(\nu,\cF)/2$.
\end{proof}

\begin{proof}[Proof of \Lemma{newpreiss}]
If $\Tan(\eta,x)\backslash \cF\neq \varnothing$, then, by Corollary \ref{c:pstep}, we may find $\mu\in \Tan(\eta,x) \setminus \cF$ and $\ve,r_{0}>0$ so that $d_{r_{0}}(\mu,\cF)=\ve$ and $d_{r}(\mu,\cF)\leq \ve$ for all $r>r_{0}$. By assumption, this implies $\mu\in \cF$, which is a contradiction. Thus, $\Tan(\eta,x)\subset \cF$.
\end{proof}

\section{Elliptic Measures}

\subsection{Uniformly elliptic operators in divergence form}

Let $A$ be a real matrix that satisfies \eqref{eqelliptic1} and \eqref{eqelliptic2}.
%such that $a_{ij} \in L^\infty(\Omega;\R)$, which also satisfies the following {\it ellipticity} condition: for some $\Lambda>0$,\begin{equation*}\Lambda^{-1}|\xi|^2\leq \langle A(x) \xi,\xi\rangle \quad \textup{and}\quad \langle A(x) \xi,\eta \rangle  \leq\Lambda |\xi| |\eta|, \, \mbox{ for all $\xi, \eta \in\R^{n+1}$ and a.e. $x\in\Omega$.}\end{equation*} 
%We do not assume here $A$ to be either locally Lipschitz. %or symmetric.
We consider the second order elliptic operator $L = -\dv A \nabla$ and we say that a function $u \in W^{1,2}_{loc}(\om)$ is a {\it weak solution} of the equation $L u=0$ in $\Omega$ (or just {\it $L$-harmonic}) if 
\begin{equation}\label{eq:solution}
%\label{e:elliptic}
\int A \nabla u \nabla \vphi=0, \;\; \mbox{ 
for all $\vphi \in C^\infty_0(\om)$}.
\end{equation} 
We also say that  $u \in W^{1,2}_{loc}(\om)$ is a {\it supersolution} (resp. \textit{subsolution}) for $ L$ in  $\Omega$ or just {\it $L$-superharmonic} (resp. {\it $L$-subharmonic}) if $\int A \nabla u \nabla \vphi \geq 0$   (resp. $\int A \nabla u \nabla \vphi \leq 0$) for all non-negative $\vphi \in C^\infty_0(\om)$.

\vv
Following  \cite[Section 9]{HKM}, from now on we make the convention that if $\om$ is unbounded, then the point at infinity always belongs to its boundary. So, all the topological notions are understood with respect to the compactified space $\overline{\R}^{n+1}=\R^{n+1} \cup \{\infty\}$. Moreover, the functions $f \in C(E)$, for $E \subset \overline{\R}^{n+1}$ are assumed to be continuous and real-valued. Therefore, all functions in $C(\d \om)$ are bounded even if $\Omega$ is unbounded.

\subsection{Regularity of the domain and Dirichlet problem}

We say that a point $x_0 \in \d \om\setminus \{\infty\}$  is {\it Sobolev $L$-regular} if, for each function $\vphi \in W^{1,2}(\Omega)\cap C(\overline \Omega)$, the $L$-harmonic function $h$ in $\om$ with $h-\vphi \in W^{1,2}_0(\om)$ satisfies
$$\lim_{x \to x_0} h(x)=\vphi(x_0) .$$

\vv
 \begin{theorem}[Theorem 6.27 in \cite{HKM}] \label{teoreg1}If for $x_0 \in \partial \Omega\setminus \{\infty\}$ it holds that
  $$\int_0^1 \frac{\textup{cap}(B(x_0,r) \cap \Omega^c, B(x_0,2r))}{\textup{cap}(B(x_0,r),B(x_0,2r))}\, \frac{dr}{r}=+\infty,$$
 then $x_0$ is Sobolev $L$-regular. Here ${\rm cap}(\cdot, \cdot)$ stands for the variational $2$--capacity of the condenser $(\cdot, \cdot)$ (see e.g. \cite[p. 27]{HKM}).
 \end{theorem}
 
%\begin{definition}
We say that a point $x_0 \in \d \om$  is {\it Wiener regular} if, for each function $f \in C(\d \om; \R)$, the $L$-harmonic function $H_f$ constructed by the Perron's method satisfies 
$$\lim_{x \to x_0} H_f(x)=f(x_0) .$$
See \cite[Chapter 9]{HKM}. 
%\end{definition} 
 
 \vv
 \begin{lemma}[Theorem 9.20 in \cite{HKM}]
 Suppose that $x_0 \in \d \om \setminus \{\infty\}$. If $x_0$ is Sobolev $L$-regular then it is also  Wiener regular.
 \end{lemma}
 
The aforementioned results form \cite{HKM} are only stated for $\Omega$ bounded but in fact they hold for unbounded domains (see discussion in \cite{AGMT17}). Moreover, $\infty$ is a Wiener regular point for each unbounded $\Omega \subset \R^{n+1}$, if $n \geq 2$ (see e.g. Theorem 9.22 in \cite{HKM}).
 
We say that $\Omega$ is Sobolev $L$-regular (resp.\ Wiener regular) if all the points in $\partial\Omega \setminus \{\infty\}$ are Sobolev $L$-regular  (resp.\ Wiener regular).

\begin{definition}\label{CDC}
A domain $\Omega \subset \R^{n+1}$ is called {\it regular} if every point of $\d_\infty\Omega$ is regular (i.e., if the classical Dirichlet problem is solvable in $\Omega$ for the elliptic operator $\mathcal{L}$). For $K\subset \d\Omega$, we say that $\Omega$ has the {\it capacity density condition (CDC)} if, for all $x\in \d\Omega$ and $0<r<\diam \d\Omega$, 
\[\textup{cap}({B}(x,r) \cap \Omega^c, B(x,2r)) \gtrsim r^{n-1}.\] 
\end{definition}

Note that if $n \geq 2$, by Wiener's criterion, domains satisfying the CDC are both Wiener regular and $L$-Sobolev regular.
\vv

Let $\Omega \subset \R^{n+1}$ be Wiener regular and $x\in\Omega$. If $f \in C(\partial \Omega)$, then the map $f \mapsto \overline{H}_f(x)$ is a bounded linear functional on $C(\partial \Omega)$. Therefore, by Riesz representation theorem and the maximum principle, there exists a probability measure $\hm^{x}$ on $\partial \Omega$ (associated to $L$ and the point $x \in \Omega$) defined on Borel subsets of $\d\Omega$ so that
$$ \overline{H}_f(x) =\int_{\partial \Omega} f \, d\hm^{x}, \;\; \mbox{for all $x \in \Omega$.}$$ 
% It follows from \cite[Theorem 11.1]{HKM} that $\hm^{x}(E)=\hm(E, \Omega ;  L)(x)$. Moreover,  
% $$\omega^{x}(\d \Omega)=\overline H_1(x) =1.$$
We call $\omega^x$ the {\it elliptic measure} or {\it  $L$-harmonic measure} associated to $L$ and $x$.\vv

\subsection{Green function and PDE estimates}

\begin{lemma}\label{phiwG}
Let $\Omega\subset\R^{n+1}$ be an open, connected set so that $\partial\Omega$ is Sobolev $L$-regular. 
There exists a Green function $G:\Omega\times \Omega\setminus\{(x,y):x=y\}\to \R$ associated with $L$ which
satisfies the following.
For $0<a<1$, there are are positive constants $C$ and $c$ depending on $a$, $n$ and $\Lambda$ such that for all $x,y\in\Omega$ with $x\neq y$, it holds:
$$0\leq G(x,y)\leq C\,|x-y|^{1-n}$$

\vspace{-4mm}
$$G(x,y)\geq c\,|x-y|^{1-n} \quad \mbox{ if $|x-y|\leq a\,\delta_\Omega(x)$,}$$

\vspace{-3mm}
$$G(x,\cdot)\in C(\overline\Omega \setminus \{x\}) \cap W^{1,2}_{loc}(\Omega \setminus \{x\}) 
\quad \mbox{ and } \quad G(x,\cdot)|_{\partial\Omega} \equiv 0,$$

\vspace{-3mm}
$$G(x,y) = G^T(y,x),$$
where $G^T$ is the Green function associated with the operator $L_{A^T}$, and for every $\vphi\in C_c^\infty(\R^{n+1})$,
\begin{equation}\label{eq:identity-elliptic measure}
\int_{\partial\Omega} \vphi\,d\omega^x - \vphi(x) = - \int_\Omega A^T(y)\nabla_y G(x,y)\cdot \nabla\vphi(y)\,dy,
\quad\mbox{ for a.e. $x\in\Omega$.}
\end{equation}
\end{lemma}

In the statement of \eqref{eq:identity-elliptic measure}, one should understand that the integral on right hand side is absolutely convergent for a.e.\ $x\in\Omega$ and a proof of it can be found in Lemma 2.4 in \cite{AGMT17}. The rest were proved in \cite{GW82} and \cite{HK07}.

\vvv

The lemma below is frequently called Bourgain's Lemma, as he proved a similar estimate for harmonic measure in \cite{Bou87}. 

\begin{lemma}[{\cite[Lemma 11.21]{HKM}}]\label{l:bourgain}
Let $\Omega\subset \bR^{n+1}$ be any domain satisfying the CDC condition,  $x_{0}\in \d\Omega$, and $r>0$ so that $\Omega\backslash B(x_{0},2r)\neq\varnothing$. Then 
\begin{equation}\label{e:bourgain}
\hm_{\Omega}^{L,x}(B(x_{0},2r)\geq c >0 \;\; \mbox{ for all }x\in \Omega\cap B(x_{0},r),
\end{equation}
where $c$ depends on $d$ and the constant in the CDC.
\end{lemma}

\begin{lemma}
For $\Omega\subset \bR^{n+1}$ and the assumptions of Lemma \ref{phiwG}, if $B$ is centered on $\d\Omega$, then
\begin{equation}\label{G<w}
\inf_{z\in 2B} \omega^{L,z}(4B) G(x,y)r_{B}^{n-1}\lec \omega^{L,y}(4B) \mbox{ for }x\in B\cap \Omega \mbox{ and }Y\in \Omega\backslash 2B.
\end{equation}
In particular, for a CDC domain, we have 
\[
 G(x,y)r_{B}^{n-1}\lec \omega^{L,y}(4B) \mbox{ for }x\in B\cap \Omega \mbox{ and }y\in \Omega\backslash 2B.\]
 
\end{lemma}

\begin{proof}
This was originally shown for harmonic measure in \cite{AHMMMTV16}, but we cover the details here.

Suppose $n>1$. By Bourgain's estimate, $\omega^{L,y}(4B)\gec 1$ for $Y\in 2B\cap \Omega$, and so for $Y\in \Omega\backslash 2B$ and $x\in B\cap \Omega$ 
\[
\inf_{z\in 2B} \omega^{L,z}(4B) G(x,y)r_{B}^{n-1} \lec  \frac{\inf_{z\in 2B} \omega^{L,z}(4B) }{|x-y|^{n-1}} r_{B}^{n-1} \lec \inf_{z\in 2B} \omega^{L,z}(4B)\]
and since $G(x,\cdot)$ vanishes on $\d\Omega$, we thus have that, for some constant $C>0$,
\[
\limsup_{y\rightarrow \xi} C\omega^{L,y}(4B)- \inf_{z\in 2B} \omega^{L,z}(4B) G(x,y)r_{B}^{n-1}\geq 0 \;\; \mbox{ for all }\xi\in \d (\Omega\backslash 2B)\]
and so \eqref{G<w} follows from the maximum principle \cite[Theorem 11.9]{HKM}.  The case $n=1$ is similar.
\end{proof}

By an iteration argument using the previous lemma, one can obtain the following lemma. 

\begin{lemma}\label{holder}
Let $\Omega\subsetneq\R^{n+1}$ be open with the CDC. Let $x\in\partial\Omega$ and $0<r<\diam\Omega$.
Let $u$ be a non-negative $L$-harmonic function in $B(x,4r)\cap \Omega$ and continuous in $B(x,4r)\cap \overline\Omega$
so that $u\equiv 0$ in $\partial\Omega\cap B(x,4r)$. Then extending $u$ by $0$ in $B(x,4r)\setminus \overline\Omega$,
there exists a constant $\alpha>0$ such that
\begin{equation}
\label{holder}
u(y)\leq C \,\left(\frac{\delta_\Omega(y)}r\right)^\alpha \!\sup_{B(x,2r)}u
\quad \mbox{for all $y\in B(x,r)$,}
\end{equation}
where $C$ and $\alpha$ depend on $n$, $\Lambda$ and the CDC constant, and $\delta_{\Omega}(y)=\dist(y,\Omega^{c})$. In particular, $u$ is $\alpha$-H\"older continuous in $B(x,r)$.
\end{lemma}

%Equipped with this lemma, the following lemma has the same proof as \cite[Lemma 4.4]{JK82}. 
%\begin{lemma}\label{boundaryharnack}
%If $\Omega$ is uniform and $u$ is a nonnegative $L$-harmonic function vanishing in $\d\Omega \cap 2B$ where $B$ is a ball is centered on $\d\Omega$, then 
%\[
%\sup_{x\in B\cap \Omega} u \lec u(x_{B}).\]
%\end{lemma}

%{\Rd We will need the following well known change of variables theorem, see for example \cite{Ngu10}.}

The following lemma is  standard but we provide a proof for the sake of completeness. 

\begin{lemma}\label{l:cov}
Let $\Omega \subset \R^{n+1}$ be an open set, and assume that $A$ is an elliptic matrix and $\Phi:\R^{n+1} \to \R^{n+1}$ is a bi-Lipschitz map. Set
$$\widetilde A:=  |\det D_{\Phi}| D_{\Phi^{-1}} (A\!\circ\!\Phi) D^T_{\Phi^{-1}}.$$ 
Then $u$ is a weak solution of $L_{A} u =0$ in $\Phi(\Omega)$ if and only if $\tilde{u}=u\circ \Phi$ is a weak solution of $L_{\widetilde A} \widetilde u =0$ in $\Omega$.
\end{lemma}

\begin{proof}
 Let $\vphi\in C_{c}^{\infty}(\R^{n+1})$ and $\vphi=\psi\circ \Phi$. Then by change of variables and the chain rule
\begin{align*}
&\int_{\Phi(\Omega)}  A \grad u \grad \psi= 
\int_{\Omega}  (A\!\circ\!\Phi) \grad u\!\circ\!\Phi \cdot \grad \psi\!\circ\!\Phi\,|\det D_{\Phi}|\\
& =\int_{\Omega}  (A\!\circ\!\Phi) D^T_{\Phi^{-1}}\grad (u\!\circ\!\Phi) \cdot D^T_{\Phi^{-1}}\grad (\psi\!\circ\!\Phi)\,|\det D_{\Phi}| \\
& = \int_{\Omega} |\det D_{\Phi}| D_{\Phi^{-1}}(A\!\circ\!\Phi) D^T_{\Phi^{-1}}\grad (u\!\circ\!\Phi) \cdot \grad (\psi\!\circ\!\Phi)\\
&= \int_{\Omega} \widetilde A \grad \widetilde u \cdot \grad \vphi.
\end{align*}
The lemma readily follows.
\end{proof}

We will usually apply the above lemma when $\Phi(x)=Sx$ for some matrix $S$, in which case 
\begin{equation}\label{e:matrixform}
\tilde{A} = (\det S) S^{-1} (A\circ S) (S^{-1})^{T}.
\end{equation}

\begin{lemma}\label{l:cov2}
With the same assumptions as \Lemma{cov}, and assuming $\Omega$ is a Wiener regular domain, we have that for any set $E \subset \Phi(\d \Omega)=\d \Phi(\Omega)$ and $x\in \Omega$,
\begin{equation}\label{e:pushf}
 \omega_{\Phi(\Omega)}^{L_{A},\Phi(x)}(E)=\omega_{{\Omega}}^{L_{\widetilde A},x}(\Phi^{-1}(E)).
\end{equation}
\end{lemma}

\begin{proof}
Let $\vphi\in C_{c}^{\infty}(\R^{n+1})$. Since the function
\[
v(x)=\int \vphi \,d\omega_{\Phi(\Omega)}^{L,x}\]
 is $L$-harmonic for $x \in \Phi(\Omega)$, by the previous lemma we know that the function
 \[
\widetilde v(x)=\int \vphi \,d\omega_{\Phi(\Omega)}^{L,\Phi(x)}\]
is $L_{\widetilde A}$-harmonic for $x \in \Omega$. If $\xi\in \d\Omega$, then as $x\rightarrow \xi$ in $\Omega$, $\Phi(x)\rightarrow \Phi(\xi)$ in $\Phi(\Omega)$, and so
 \[
\widetilde v(x)=\int \vphi \,d\omega_{\Phi(\Omega)}^{L,\Phi(x)} \rightarrow \vphi(\Phi(\xi)).
\]
Thus, $\widetilde v$ is the $L_{\widetilde A}$-harmonic extension of $(\vphi\circ\Phi)|_{\d \om}$ to $\Omega$, and so
\[
\int_{\d\Phi(\Omega)} \vphi \,d\omega_{\Phi(\Omega)}^{L_A,\Phi(x)}= \int_{\d {\Omega}} \vphi\circ
\Phi\, d\omega_{\om}^{L_{\widetilde A},x},\quad\textup{for all}\,\, x\in \Omega.\]
Since this holds for all such $\vphi$, we get that for any set $E \subset \d \Phi(\Omega) = \Phi(\d \Omega)$,
\[
\omega_{\Phi(\Omega)}^{L_{A},\Phi(x)}(E)=\omega_{{\Omega}}^{L_{\widetilde A},x}(\Phi^{-1}(E)),
\]
which gives the lemma.

\end{proof}

The following lemma will help us relate elliptic harmonic polynomial measures to just harmonic polynomial measures. 

\begin{lemma}\label{l:pushforpol}
	Let $A$ be an elliptic constant matrix, $A_s=(A+A^{T})/2$, and matrix and $S=\sqrt{A_s}$. Let $h\in H_{A}$ and $\widetilde{h}=h\circ S$. Then
	\begin{equation}\label{e:pushforpol}
	\omega_{\widetilde{h}}=(\det S)^{-1}  S^{-1}[\omega_{h}^{A}].
	\end{equation}
\end{lemma}

\begin{proof}
Note that since $L_A$ has constant coefficients, then $L_{A_s}=L_{A}$ by the fact that for $u \in C^2$
\begin{align*}
L_A u&= \sum_{i,j} a_{ij} \d_i  \d_j u= \frac{1}{2}\sum_{i,j} a_{ij} \d_i  \d_j u + \frac{1}{2}\sum_{i,j}a_{ij} \d_j \d_i u\\
&=  \sum_{i,j}\frac{(a_{ij} + a_{ji})}{2} \d_i \d_j u = L_{A_s} u.	
\end{align*}
 Thus, if $h$ is an $L_{A}$-harmonic function, it is also an $L_{A_s}$-harmonic function. Moreover, for any $\psi\in C_{c}^{\infty}(\R^{n+1})$
\begin{align*}
 \int \psi \,d\omega_{h}^{A_s} = \int_{\Omega_{h}} h\, L_{A_s}(\psi )&=  \int_{\Omega_{h}} h\, L_{A}(\psi )=  \int \psi \,d\omega_{h}^{A}.
\end{align*}
In fact, without loss of generality, we may assume that $A=A_s$. 

Recall now that since $A_s$ is a symmetric, positive definite and invertible matrix with constant real entries, then it has a unique real symmetric positive definite square root $S= \sqrt{A_s}$ which is also invertible. Hence, by Lemma \ref{l:cov} and \eqref{e:matrixform} with $A=A_s$, we have that $\tilde A= (\det S) I$ and $\widetilde{h}$ is $L_{(\det S)I}$-harmonic, and thus just harmonic. 

 Let now $\vphi\in C_{c}^{\infty}(\R^{n+1})$ and $\psi \circ S=\vphi$. By integration by parts and the fact that $S$ is also symmetric, we have that
 	\begin{align*}
	(\det S) \int \vphi \,d\omega_{\widetilde{h}}
	&= (\det S)\int_{\Omega_{\widetilde{h}}} \widetilde{h} \,\Delta \vphi
	=-(\det S) \,\int_{\Omega_{\widetilde{h}}} \grad \widetilde{h}\cdot  \grad \vphi \\
	&= -(\det S) \,\int_{\Omega_{\widetilde{h}}} S^{T}\grad h\circ S \cdot S^{T} \grad \psi\circ S \\
	&=-\int_{S^{-1}(\Omega_{h})} SS^{T}\grad h\circ S \cdot  \grad \psi\circ S \\
	&=-\int_{\Omega_{h}} A_s \grad h \cdot  \grad \psi = \int_{\Omega_{h}} h\, L_{A_s}(\psi )\\
	&=  \int_{\Omega_{h}} h\, L_{A}(\psi )= \int \psi \,d\omega_{h}^{A} 
	=\int \vphi \,dS^{-1}[\omega_{h}^{A}].
	\end{align*}
\end{proof}
%
%As a corollary of  Lemmas \ref{l:cov} and \ref{l:cov2} we have the following.
%\begin{corollary}\label{cor:A(x0)=id}
%Let $\Omega \subset \R^{n+1}$ be an open set, and assume that $A$ is a uniformly elliptic matrix with real entries.  Let $A_s= (A + A^*)/2$  be the symmetric part of $A$ and for a fixed point $y_0 \in \om$ define $S= \sqrt{A_s(y_0)}$. If  
%\[
%\tilde{A}(\cdot) = S^{-1} (A\circ S)(\cdot) S^{-1},
%\]
%then $\tilde{A}$ is uniformly elliptic,  $\tilde A_s(z_0) = Id$ if $z_0 = S^{-1}y_0$ and $u$ is a weak solution of $L_{A} u =0$ in $\Omega$ if and only if $\tilde{u}=u\circ S$ is a weak solution of $L_{\tilde A} \tilde u =0$ in $S^{-1}(\Omega)$ . 
%
%Assuming $\Omega$ is a Wiener regular domain, we have that for any set $E \subset \d \Omega$ and $x\in \Omega$,
%\begin{equation}\label{e:pushf}
% \omega_{\Omega}^{L_{A},x}(E)=\omega_{S^{-1}(\Omega)}^{L_{\widetilde A},S^{-1}x}(S^{-1}(E)).
%\end{equation}
%\end{corollary}

\vv

Let us recall some simple facts from linear algebra which help us understand how the geometry of $\Omega$ is affected by the linear transformation above. Note that $S$ is orthogonally diagonalizable since it is symmetric, which means that it represents a linear transformation with scaling in mutually perpendicular directions. Hence $S^{-1}$ is a special bi-Lipschitz change of variables that takes balls to ellipsoids, where eigenvectors determine directions of semi-axes, eigenvalues determine lengths of semi-axes and its maximum eccentricity is given by $\sqrt{(\lambda_{\max} / {\lambda_{\min}})}$ (where $\lambda_{\max}$  are $\lambda_{\min}$ 
are the maximal and minimal eigenvalues of $S^{-1}$),
which is in turn bounded below by $\sqrt{\Lambda}^{-1}$ and above by $\sqrt{\Lambda}$. In particular, $S^{-1}(\d \Omega)=\d (S^{-1}(\Omega))$, $\Lambda^{-1/2} \leq \|S^{-1} \|\leq   \Lambda^{1/2}$, i.e., $S^{-1}$ distorts distances by at most a constant depending on ellipticity. .

%From now on we will assume that the entries of any uniformly elliptic matrix $A$ satisfy the following: There exists $\theta: [0, \infty] \to [0, \infty]$ which is  non-decreasing and satisfies $\lim_{r \to 0} \theta(r)= \theta(0)=0$, such that $$|a_{i,j}(x)-a_{i,j}(y)| \leq \theta(|x-y|).$$

\subsection{The main blow-up lemma}

We now introduce the main tool of this paper, which is a variant of previous blow-up arguments, first introduced by Kenig and Toro for NTA domains \cite{KT06}, then extended to CDC domains in \cite{AMT16}. Both these cases applied to harmonic measure, but it can be extended to elliptic measures with a $\VMO$ condition on the coefficients.

\begin{lemma}\label{l:azmoto}
	Let $\Omega^+\subset \bR^{n+1}$ be a CDC domain, $K\subset \d\Omega^{+}$ a compact set, $\xi_{j}\in K$ and $L=-\div A \nabla$ be a uniformly elliptic operator in $\Omega^+$ such that 
\begin{equation}\label{e:boundarylimitK2}
	\lim_{r\rightarrow 0} \sup_{\xi\in K} \frac{1}{r^{n+1}} \inf_{C\in \cC}\int_{B(\xi,r)\cap \Omega^+} |A(x)-C|dx=0.
	\end{equation}

	%which is continuous at the boundary $\partial \Omega$. %and $a_{i,j}(\xi)=\delta_{ij}$, where $\delta_{ij}$ stands for Kronecker delta. 
	Let $\omega^{\pm}$ be the elliptic measures for $\Omega^\pm$. and
	$\omega_\infty^+\in \Tan(\omega^+,\xi)$, with
	$c_{j}\geq 0$, and $r_{j}\rightarrow 0$ such that $\omega_{j}^+=c_{j}T_{\xi_{j},r_{j}}[\omega^+]\rightarrow \omega_{\infty}^+$. Let $\Omega_{j}^\pm=T_{\xi_{j},r_{j}}(\Omega^\pm)$. Then there is a subsequence and a closed set $\Sigma\subset \R^{n+1}$ such that 
	\begin{enumerate}[(a)]
		\item $\d\Omega_{j}^+\cap E\rightarrow \Sigma\cap E$ in the Hausdorff metric for any compact set $E$.
		\item $\Sigma^{c}=\Omega_{\infty}^+\cup \Omega_{\infty}^-$ where $\Omega_{\infty}^+$ is a nonempty open set and $\Omega_{\infty}^-$ is also open but possibly empty. Further, they
		satisfy that for any ball $B$ with $\cnj{B}\subset \Omega_{\infty}^\pm$, a neighborhood of $\cnj{B}$ is contained in $\Omega_{j}^\pm$ for all $j$ large enough.
		\item $\supp \omega_{\infty}^+\subset \Sigma$. 
		\item Let $u^+(x)=G_{\Omega^+}(x,x^+)$ on $\Omega^+$ and $u^+(x)=0$ on $(\Omega^+)^{c}$. Set
		\[u_{j}^+(x)=c_{j}\,u^+(xr_{j}+\xi_{j})\,r_{j}^{n-1}.\]
		Then $u_{j}^+$ converges locally uniformly in $\bR^{n+1}$ and in $W_{\textup{loc}}^{1,2}(\R^{n+1})$ to a nonzero function $u_{\infty}^+$ which is continuous in $\R^{n+1}$, vanishes in $(\om^+_\infty)^c$ and satisfies
		\begin{equation}
		u_{\infty}^+(y)\lec  \omega^+_{\infty}({\overline B(x,4r)})\,r^{1-n},
		\label{e:u<wr}
		\end{equation}
		for $x\in \Sigma$, $r>0$, and $y\in B(x,r)\cap \Omega^+_{\infty}$. {Moreover, there is $A^+_{0}$ a constant elliptic matrix so that if $L^+_{0}=-\dv A^+_{0} \nabla$, then 
			\begin{align}\label{e:ibp}
			\int \vphi \,d\omega_{\infty}^{+} = \int_{\R^{n+1}}  u_{\infty}^{+}\, L_{0}^{+} \vphi, \quad \textup{for any}\,\, \vphi \in C^\infty_c( \R^{n+1}).
			\end{align}
		}
	\end{enumerate}
	Suppose now that $\Omega^-=\R^{n+1}\setminus \overline{\Omega^+}$,
	so that $\partial\Omega^+=\partial\Omega^-$ and $\Omega^-$ is also connected and CDC. Define analogously $\omega_{j}^{-}$, $u^-$, $u_j^-$, and $u^-_\infty$. Assume that $A$ is uniformly elliptic in $\Omega^{+}\cup \Omega^{-}$, \eqref{e:boundarylimitK2} holds for $\Omega^{+}\cup \Omega^{-}$ and $\omega_{j}^{-}$ converges weakly to $\omega_{\infty}^{-}=c\omega_{\infty}^{+}$ for some number $c\in (0,\infty)$. %(which happens, for example, if $\xi\in \Gamma$ where $\Gamma$ is as in \Lemma{samew}).
	% Suppose now that $\Omega^-$ is also connected and $\Delta$-regular. Let $E$ and $\Gamma$ be as above and suppose additionally that $\xi\in \Gamma$. Define analogously $u^-$, u_j^-$ and $u^-_\infty$. 
	Then $\Omega_{\infty}^-\neq\varnothing$ and for a suitable subsequence, (d) holds for $u_{j}^-$, $u_{\infty}^-$, and
	$\Omega^-_\infty$. Furthermore, if we set $u_{\infty}= u_{\infty}^+ - c^{-1} u_{\infty}^-$, then
	\begin{enumerate}[(a)]
		\setcounter{enumi}{4}
		\item $u_{\infty}$ extends to a continuous function on $\bR^{n+1}$ which satisfies $L_{0} u_\infty =0$ in $\R^{n+1}$.
		\item $\Sigma=\{u_{\infty}=0\}$, with $u_\infty>0$ on $\Omega_\infty^+$ and $u_\infty<0$ on $\Omega_\infty^-$. Further, $\Sigma$ is a real analytic variety of dimension $n$.
		\item $d\omega_{\infty}^+=-\frac{\partial u_\infty}{\partial \nu_{A_{0}}}\,d\sigma_{\d\Omega_{\infty}^+}$, where $\sigma_{S}$ stands for the surface measure on a surface $S$ and $\frac{\partial}{\partial \nu_{A_{0}}}=\nu \cdot A_{0} \nabla$ is the outward co-normal derivative.
	\end{enumerate}
	\label{l:blowup1}
\end{lemma}

\begin{proof}
The proof of this lemma can be found in \cite{AMT16} for harmonic measure for the case that $K=\{\xi\}$ (i.e. so that \eqref{e:boundarylimit} holds). The proof for general $K$ is essentially the same in this setting with minor modifications. Here we shall only record the required modifications (some of which are quite substantial) for the $K=\{\xi\}$ case in order for the same proof to work for any elliptic measure as well. In this case, $\xi_{j}=\xi$ for all $j$. We set 

\[
A_j(x):=A(r_j x+ \xi),  \,\, u_j^\pm(x):=c_j r_j^{n-1} u^\pm( r_j x+ \xi)\]
and
\[
  \vphi_j(x) := \vphi\ps{\frac{x-\xi}{r_j}}.
\]
Without loss of generality we can only work with $u^+$ since the results for $u^-$ can be proved analogously. 

Notice now that for $j$ large enough, the pole $x^+ \not \in \supp(\vphi_j)$. In fact, for any ball $B$ centered at the boundary of $\om_j$, we can find $j_0\in \mathbb N$, such that for all $j \geq j_0$, $x^+ \not \in T_{\xi, r_j} (B) $. Moreover, for $x\in B\cap\Omega_j$ and $j$ large enough,
\begin{align}\label{e:uij<w}
u_{j}^+(x) &
=c_{j}\,r_{j}^{n-1}\, u^+(r_{j}x +\xi)\\
&
\stackrel{\eqref{G<w}}{\lec} c_{j}\,r_{j}^{n-1}(r_{j}r_{B})^{1-n}\, \omega^+(4r_{j} B+\xi)
=r_{B}^{1-n}\,\omega_{j}^+(4B)\nonumber.
\end{align}

\noindent {\bf Proof of (b):} We only need to prove the existence of $B \subset \om_j^+$ for large $j \in \mathbb N$. Suppose there is no such ball.  Let $\vphi$ be any continuous compactly supported non-negative function for which $\int \vphi\, d\omega_{\infty}^+\neq 0$, and let $M>0$ be so that $\supp \vphi \subset B(0,M)$. Thus, there must be $x_0 \in B(0,M)\cap \supp \omega_{\infty}^+$. We set 
$$\delta_{j}:=\sup\{\dist(x,(\Omega_{j}^{+})^{c}):x\in B(0,2M)\},$$ 
which goes to zero by assumption. For $x \in B(0, 2M)$ and $j\in \mathbb N$, let $\zeta_{j}(x)\in (\Omega_{j}^{^+})^c$ be closest to $x$ so 
that $|x-\zeta_{j}(x)|\leq \delta_{j} \leq 2M$ (the second inequality holds because $0 \in \d \om_j^+$). It also holds that for all $x\in B(0,2M)$,  $|x-x_{0}|\leq |x|+|x_{0}|< 3M$.

Notice now that for any $j$ big enough,  $u_j^+$ is a solution in $B(0,2M) \cap \om_j^+$ and a subsolution in  $B(0,2M)$. Moreover, if $x\in\Omega_j^+$, then $\zeta_j(x) \in\d\Omega_j^+$. Thus, for $j$ large, by Cauchy-Scwharz, Caccioppoli's inequality in $B(0,M)$ (which also holds for subsolutions) and the fact that $u_j^+$ and $\vphi$ are supported in $\Omega_j^+$ and $B(0,M)$ respectively, 
\begin{align*}
&0<\int \vphi\, d\omega_{j}^+ =\int_{\Omega_{j}^+} A_j \nabla u_{j}^+\cdot\nabla \vphi \lesssim_{\lambda, \Lambda, n, M} \|\nabla \vphi\|_\infty \left(\int_{ B(0, 2M)} |u_j^+|^2 \right)^{1/2}\\
&\stackrel{\eqref{holder}}{\lec} \left( \int_{\Omega_j^+ \cap B(0,2M)} \left(\,\sup_{B(\zeta_{j}(x),2 M)} u_{j}^+\right)^2  \ps{\frac{x-\zeta_{j}(x)}{2M}}^{2\alpha}\,dx \right)^{1/2} \\ 
&\stackrel{\eqn{uij<w}}{\lec} \left( \int_{\Omega_j^+ \cap B(0,2M)} \left[\, \omega_{j}^+(B(\zeta_{j}(x),8 M)) \,(2M)^{1-n}\right]^2 \,dx \right)^{1/2} \ps{\frac{\delta_{j}}{2M}}^{\alpha}\\
&\lec (2M)^{\frac{n+1}{2}}\omega_{j}^+(B(x_{0},13M)) \,(2M)^{1-n}\ps{\frac{\delta_{j}}{2M}}^{\alpha},
\end{align*}
and thus
\begin{align*}
0<\int \vphi \,d\omega_\infty^+
& \lec_{\lambda, \Lambda, n, M,\vphi} \ps{\limsup_{j\rightarrow\infty}  \omega_{j}^+(B(x_{0},13M))}  \lim_{j}\delta_{j}^{\alpha}\\
&\leq \omega^+_{\infty}(\overline B(x_{0},13M)) \cdot 0=0,
\end{align*}
which is a contradiction. Thus, there is $B\subset \Omega_{j}$ for all large $j$ (after passing to a subsequence).

\noindent {\bf Proof of (d):} Arguing as in \cite{AMT16}, there exists $u_\infty^+$ which is continuous in $\R^{n+1}$ and vanishes on $(\om^+_\infty)^c$ such that (after passing to a subsequence) $u_j^+ \to u^+_\infty$ uniformly on compact sets of $\R^{n+1}$. Moreover, it is not hard to see that $u_j^+ \in W^{1.2}(B)$ for large $j$. Indeed,  by \eqref{e:uij<w}, it is clear that 
\begin{equation}\label{eq:sobolev1-uj}
\|u_j^+\|_{L^2(B)}^2 \lesssim r_{B}^{3-n}\,[\omega_{j}^+(4B)]^2,
\end{equation}
while by  Caccioppoli's inequality and \eqref{e:uij<w}, 
\begin{align}
\int_B |\nabla u_j^+|^2 
%&\lesssim  c_j^2 r_j^{2n-(n+1)} (r_j r_{B})^{-2} \int_{\xi+r_j B} |u^+|^2 
& \lec  r_{B}^{-2}\int_B |u_j^+|^2  \lesssim r_{B}^{-2} [r_{B}^{1-n}\,\omega_{j}^+(4B)]^2 r_{B}^{n+1} =r_{B}^{1-n}\,[\omega_{j}^+(4B)]^2. \label{eq:sobolev2-uj}
\end{align}
%Moreover,  $u_j^+ \in L^2(B)$ since $u^+$ vanishes at the center of the ball $r_jB+\xi$. To prove this we may assume, without any loss of generality, that $B$ is centered at $0$. Then
%\begin{align*}
%\int_B |u_j^+|^2 &=  c_j^2 r_j^{2(n-1)} (r_j r_{B})^{-(n+1)} \int_{B(\xi,r_j r_{B})}  |u^+-u^+(\xi)|^2 \notag\\
%&= c_j^2 r_j^{2(n-1)} (r_j r_{B})^{-(n+1)}  \int_{\mathbb S^n} \int_0^{r_j r_{B}} |u^+(\rho \theta+\xi)-u^+(\xi)|^2 \,d\rho d\theta \notag\\
%&= c_j^2 r_j^{2(n-1)} (r_j r_{B})^{-(n+1)} \int_{\mathbb S^n} \int_0^{r_j r_{B}} \left|\int_0^\rho \partial_s u^+(s \theta+\xi) \,ds\right|^2 \,d\rho\, d\theta\notag\\
%&\lesssim  c_j^2 r_j^{2(n-1)} (r_j r_{B})^{-(n+1)}  \int_{\mathbb S^n} \int_0^{r_j r_{B}} \rho \int_0^{r_j r_{B}} |\partial_s u^+(s \theta+\xi)|^2 \,ds \,d\rho\, d\theta\notag\\
%&\lesssim c_j^2 r_j^{2(n-1)} (r_j r_{B})^{-(n-1)}  \int_{\mathbb S^n} \int_0^{r_j r_{B}} |\partial_s u^+(s \theta+\xi)|^2 \,ds \, d\theta\notag\\
%&\leq r_{B}^2 \int_B |\nabla u_j^+|^2 \,dx  \stackrel{\eqref{eq:sobolev-uj}}{\lec} r_{B}^{3-n}\,[\omega_{j}^+(4B)]^2.
%\end{align*}
In view of \eqref{eq:sobolev1-uj} and \eqref{eq:sobolev2-uj} we have that 
\begin{align*}
\limsup_{j \to \infty} \| u_j^+\|_{W^{1,2}(B)} &\lesssim r_{B}^{\frac{1-n}{2}}(1+r_{B})\,  \limsup_{j \to \infty} \omega_{j}^+(4B)\\
&\leq r_{B}^{\frac{1-n}{2}}(1+r_{B})\, \omega_{\infty}^+(\overline{4B})<\infty.
\end{align*}
Therefore, by \cite[Theorem 1.32]{HKM}, $u^+_\infty \in W^{1,2}_{\textup{loc}}(\R^{n+1})$ and there exists a further subsequence of $u_j^+$ that converges weakly to $u^+_\infty$ in $W^{1,2}_{\textup{loc}}(\R^{n+1})$. 
%By Rellich?Kondrachov theorem, the unit ball of $W^{1,2}(B)$ is relatively compact in $L^{2}(B)$ and thus, there exists a convergent subsequence of $u_j$ in $B$. 

Notice that
$$-\int_{\om_j^+}A_j \nabla u_j^+ \cdot \nabla \vphi= \int \vphi \,d\hm_j^+.$$
 Indeed, by a change of variables, and letting $\vphi_{j}=\vphi\circ T_{\xi,r_{j}}$ and $\vphi_{j}=\vphi\circ T_{\xi,r_{j}}$ 
\begin{align*}
\int \vphi\,d\hm_j^+&=c_j \int \vphi_j \,d\hm^+=\int_{\om^+}A \nabla u^+ \cdot \nabla \vphi_j \\
&=c_j r_j^{n}\int_{\om_j^+}A( r_j x +\xi)  \nabla u^+(r_j x+\xi) \cdot \nabla \vphi(x) dx\\
&=\int_{\om_j^+}A_j  \nabla u_j^+ \cdot \nabla \vphi.
\end{align*}

Let $C_{j,k}$ be a constant elliptic matrix so that 
\[
\lim_{j} (kr_{j})^{-1-n}\int_{B(\xi,kr_{j}) \cap \Omega^+} |A-C_{j,k}|=0.\]
By a diagonalization argument and compactness, we may pass to a subsequence so that for each $k$, $C_{j,k}$ converges to a uniformly elliptic matrix $C_{k}$ with constant coefficients. It is not hard to check that we must in fact have that $C_{k}=A^+_{0}$ for some fixed matrix $A^+_{0}$ (using the fact that $\inf \delta_{j}>0$).
Thus, we have
\begin{equation}
\label{e:AMlimbis}
\lim_{j} (Mr_{j})^{-1-n}\int_{B(\xi,Mr_{j})\cap \Omega^+} |A-A^+_{0}|=0 \;\; \mbox{ for all }M\geq 1.
\end{equation}
To see the ellipticity of $A^+_0$ is pretty easy but we show the details for completeness. Note that since $A$ is uniformly elliptic for a.e. $x \in \om^+$, then for $\xi \in \R^{n+1}$,
\begin{align*}
\Lambda^{-1} |\xi|^2 \leq A(x)\xi \cdot \xi = (A(x)-A^+_0)\,\xi \cdot \xi + A^+_0\,\xi \cdot \xi.
\end{align*}
Then, if we take  averages over $B(\xi,Mr_{j}) \cap \Omega$, use the existence of corkscrew balls in $\om_j$ for large $j$ proved in (b) and then take limits as $j \to \infty$, by \eqref{e:AMlimbis} we have
\begin{align*}
\Lambda^{-1} |\xi|^2 \leq  A^+_0\,\xi \cdot \xi.
\end{align*}
The upper bound follows by a similar argument and the proof is omitted. 

 We will now estimate the difference 
\begin{equation}\label{e:diff}
\int_{\om_j^+}A_j  \nabla u_j^+ \cdot \nabla \vphi -\int_{\om_\infty^+}A^+_{0}  \nabla u_\infty^+ \cdot \nabla \vphi,
\end{equation}
for sufficiently large $j$. 

To this end, let $\supp(\vphi) \subset B(0,M)$. Note that
%By continuity of $A$ on $\partial \Omega$, there exists $\delta(\varepsilon)>0$ such that for every $z \in B(\xi, \delta)$, it holds that $|a_{i,j}(z)-a_{i,j}(\xi)|<\varepsilon$. Choose now $j$ so large that $B(\xi, r_j M) \subset B(\xi, \delta)$ and also $\|1_{\om_j^+}-1_{\om_\infty^+}\|_{L^\infty(B(0,M))}< \varepsilon$ and note that 
\begin{align*}
|\eqn{diff}| & \leq \left| \int_{\om_j^+} (A(r_j x+\xi)-A^+_{0})  \nabla u_j^+ \cdot \nabla \vphi \right| \\
& \qquad + \av{\int_{B(0,M)}  (\nabla u_j^+\one_{\Omega_{j}} - \nabla u^+_\infty \one_{\Omega_{\infty}}) \cdot A_{0}^{+,*} \nabla \vphi }  \leq I_{1}+I_{2}.
%\leq 2\varepsilon \|\nabla \vphi\|_{L^\infty} \int_{B(0,M)}|\nabla u_j^+|.
\end{align*}
Since $u_j$ and $u_\infty$ are supported in $\om_j^+$ and $\om_\infty^+$ respectively, by the weak convergence of $\nabla u_j $ to $\nabla u_\infty $ in $L^2(B(0,M))$, we have that $I_{2}\rightarrow 0$. On the other hand, since $A$ and $A^+_0 \in L^\infty(\Omega)$,
\begin{align*}
I_{1}
& \leq  ||\nabla u_j^+ ||_{L^2(B(0,M))} || \nabla \vphi ||_{\infty} \ps{\int_{B(0,M)\cap \Omega_{j}^{+}}|A(r_j x+\xi)-A^+_{0}|^{2}dx}^{1/2}\\
& \stackrel{\eqref{eq:sobolev2-uj}}{\lesssim }_{\Lambda}\, M^{\frac{1-n}{2}} \omega_{\infty}^{+}(\cnj{B(0,4M)})  \ps{\frac{1}{r_{j}^{1+n}}\int_{B(0,Mr_{j})\cap \Omega^+}|A(x)-A^+_{0}|dx}^{1/2}\\
& \stackrel{\eqn{AMlimbis}}{\rightarrow 0 }.
\end{align*}
Thus, combining the above estimates and taking $j\rightarrow \infty$, we infer that
$$-\int_{\om_\infty^+}A^+_{0} \nabla u_\infty^+ \cdot \nabla \vphi= \int \vphi \,d\hm_\infty^+. $$
In particular, $u^+_\infty$ is a continuous weak solution of 
$$L^+_{0} w= - \dv A^+_{0} \nabla w=0 \,\, \textup{in}\,\, \om^+_\infty.$$ Since $L^+_{0}$ is a second order elliptic operator with constant coefficients, $u^+_\infty$ is real analytic in $\om^+_\infty$. Thus, by definition of $u^+_\infty$ and since the gradient of its extension by zero is the extension by zero of the gradient (see Proposition 9.18 in \cite{Br}), we have that 
 $$\int_{\om_\infty^+}A^+_{0} \nabla u_\infty^+ \cdot \nabla \vphi = \int_{\R^{n+1}} A^+_{0} \nabla u_\infty^+ \cdot \nabla \vphi.$$
We now use the divergence theorem along with the fact that $\supp(\nabla \vphi) \subset B(0.M)$ and obtain %it is a strong solution of the equation $L_{0} u^+_\infty =0$.
\begin{align*}
\int \vphi \,d\hm_\infty^+&=-\int_{\R^{n+1}} \dv [u_\infty^+\, A_{0}^{+,*} \nabla \vphi]+ \int_{\R^{n+1}} u_\infty^+\, L^{+,*}_{0} \vphi\\
&= -0 +  \int_{\R^{n+1}} u_\infty^+\, L^{+,*}_{0} \vphi,
\end{align*}
which finishes the proof of (d). The rest of the proof is almost identical since one only uses that $u_\infty$ real analytic in $\R^{n+1}$ and Liouville's theorem for positive solutions of uniformly elliptic equations (see e.g. Corollary 6.11 in \cite{HKM}).

\vv

On may argue similarly in the case of $u_j^-$. Notice that in this case, we will obtain a constant coefficient uniformly elliptic matrix $A_0$ such that \begin{equation}
\label{e:AMlimbis-}
\lim_{j} (Mr_{j})^{-1-n}\int_{B(\xi,Mr_{j})\cap( \Omega^+ \cup \Omega^-)} |A-A_{0}|=0 \;\; \mbox{ for all }M\geq 1.
\end{equation}
\end{proof}
%We may extend $A$ to $\R^{n+1} \setminus [\om^+ \cup \om^-]$ by $A_0$ and thus, we have a globally defined uniformly elliptic matrix such that
%\begin{equation}
%\label{e:AMlimglobal}
%\lim_{j} (Mr_{j})^{-1-n}\int_{B(\xi,M r_{j})} |A-A_{0}|=0 \;\; \mbox{ for all }M\geq 1.
%\end{equation} 
%and Harnack's inequality for solutions of second order uniformly elliptic equations. That is, for any positive solution in $\om$, and $s \in (0, \frac{n+1}{n-1})$,$$\sup_B u \lesssim \left(\avint_B u^s \right)^{1/s} \lesssim \inf_{4B} u,$$ for every ball $B$ such that $4B \subset \Omega$. 
%\vvv
%We now state an almost monotonicity formula for uniformly elliptic equations which coefficients satisfy the double Dini condition with respect to the point $0$, i.e., 
%$$|A(x)-A(0)| \leq \theta(|x|),$$ 
%for all $x \in \R^{n+1}$, where
%\begin{equation}\label{eq:Dini}
%\int_0^1 r \int_0^r \frac{\theta(\rho)}{\rho} \,d\rho \,dr= -\int_0^1 \theta(r) \log r \,\frac{dr}{r} <\infty.
%\end{equation}
%
% \begin{theorem}[Theorem III, \cite{MP}]
% Suppose we have  two continuous functions $u^\pm$ in the unit ball $\bB$ that satisfy
% $$u^\pm \geq 0, \quad 	L u^\pm \geq -1, \quad u^+ \cdot u^- =0 \,\, \textup{in}\,\, \bB.$$
% Then, there exists $r_\theta \in (0,1)$ such that the functional 
% $$\gamma(r):= r^{-4} \int_{r \bB} \frac{|\nabla u^+|^2}{|x|^{n-1}}\, dx \int_{r \bB} \frac{|\nabla u^-|^2}{|x|^{n-1}}\, dx$$
% satisfies
% $$\gamma(r) \leq C_\theta \left(1+ \|u^+\|^2_{L^2(\bB)} + \|u^-\|^2_{L^2(\bB)} \right)^2, \quad \textup{for all} \,\, r \leq r_\theta.$$
% \end{theorem}

\vvv 

Now we prove a slightly weaker version of this result in the next two lemmas. Again, this is based on the details in the proof of \cite[Lemma 5.3]{AMTV16}, but with some adjustments for elliptic measure.

 \begin{lemma}\label{l:azmotovo}
Let $\Omega \subset \bR^{n+1}$ be a domain. 
Let $\xi_{j}\in \d\Omega$ and $L=-\div A \nabla$ be a uniformly elliptic operator in $\Omega$ such that \eqref{e:boundarylimitK} holds with $K=\{\xi_{j}\}$ and, if $\omega=\omega_{\Omega}^{L_{A},x_{0}}$ is its $L_{A}$-harmonic measure with pole at $x_{0}\in \Omega$, there is $r_{j}\rightarrow 0$ and $c_{j}>0$ so that 
\[
\omega_{j}:= c_{j}T_{\xi_{j},r_{j}}[\omega] \rightarrow \omega_{\infty} %\in \Tan(\omega,\xi),\]
\]
\begin{equation}\label{e:om-density}
\liminf_{j} \frac{|\Omega\cap B(\xi_{j},r_{j})|}{r_{j}^{n+1}}>0,
\end{equation}
and 
\begin{equation}\label{CDCatxi}
\omega^{z}(B(\xi_{j},r_{j}/2))\gec 1 \mbox{ for all }j \mbox{ and }z\in B(\xi_{j},r_j)\cap \Omega.
\end{equation}
Then there is a subsequence such that the following hold:
 If $u(x)=G_{\Omega^{}}(x,x_{0})$ on $\Omega$ and $u(x)=0$ on $\Omega^{c}$, and 
\[u_{j}(x)=c_{j}\,u(xr_{j}+\xi_{j})\,r_{j}^{n-1},\]
then $u_{j}$ converges in $L^{2}_{loc}(\frac{1}{2}\bB)$ to a nonzero function $u_{\infty}$ which is $L_{A_{0}}$-harmonic in $\{x:u_{\infty}>0\}\cap \frac{1}{2}\bB$,  for constant uniformly elliptic matrix $A_{0}$, and  such that 
\begin{equation}\label{e:ufinbound}
||u_{\infty}||_{L^{2}(\frac{1}{2}\bB)} \lec \omega_{\infty}(\cnj{B(0,2)}),
\end{equation} 
and for any $\vphi \in C^\infty_c( \frac{1}{2}\bB)$,
 \begin{align}\label{e:ibp}
 \int \vphi \,d\omega_{\infty} = \int_{\R^{n+1}}  u_{\infty}\, L_{A_{0}}\vphi.
 \end{align}

 If $\xi=\xi_{j}$ and $A$ is continuous at $\xi$, then $A_{0}$ is just the value of $A$ at $\xi$.
\end{lemma}

\begin{proof}Recall that we denote $\bB=B(0,1)$. Again, to simplify notation, we'll just prove the case when $\xi_{j}=\xi\in \d\Omega$. 

%Let $\omega_{\infty}\in \Tan(\omega_{1},\xi)$,
%$c_{j}\geq 0$, and $r_{j}\rightarrow 0$ be such that $\omega^{j}_{1}=c_{j}T_{\xi,r_{j}}[\omega^{+}]\rightarrow \omega_{\infty}$. %{\color{blue} 
%As $\omega_{\infty}\neq 0$, there is $R>0$ so that $\omega_{\infty}(B(0,R))\neq 0$. 
By \eqref{CDCatxi}, without loss of generality, we can scale the $c_{j}$ so that %we will assume $R=1/4$, and we can pick $c_{j}$ so that
\begin{equation}\label{e:w>0}
\omega_{\infty}\bigl(\tfrac{1}{4}\bB\bigr)=1.
\end{equation}
%}

Let $\Omega_j=T_{\xi,r_{j}}(\Omega)$.  
%Let $u_{}(x)=G_{\Omega_{1}}(x,p_1)$ on $\Omega_{1}$ and $u_{1}(x)=0$ on $(\Omega_{1})^{c}$ (since we are assuming Wiener regularity, this is continuous). Set
%\[u_{j}^{+}(x)=c_{j}\,u_{1}(xr_{j}+\xi)\,r_{j}^{n-1}.\]
%Define $u_{2}$ and $u^{j}_{2}$ similarly.
%
%Without loss of generality, by passing to a subsequence we may assume that 
%\begin{equation}\label{e:om+nondeg}
%\cH^{n+1}(B(\xi,r_{j})\backslash \Omega_{1})\geq \frac{r_{j}^{n+1}}{2}.
%\end{equation}
%Thus, for $z\in B(\xi,r_{j})$,
%\[
%\omega_{\Omega_{1}}^{z}(B(\xi,\delta^{-1}r_{j}))
%\gec  \frac{\cH^{n+1}(B(\xi,r_{j})\backslash \Omega_{1}))}{r^{n+1}}
%\gec 1.\]
By \eqref{CDCatxi} and \eqref{G<w},
\begin{equation}\label{e:Green-lowerbound2}
 \omega(B(\xi,2r_{j}))\gtrsim r_{j}^{n-1}\, u(x)\quad\mbox{
 for all $x\in B(\xi,r_{j})\cap\Omega_1$,}
 \end{equation}
 and so, 
 \begin{equation}\label{e:Green-lowerbound3}
 \omega_{j}(2\bB)\gtrsim \, u_{j}(x)\quad\mbox{
 for all $x\in \bB\cap\Omega_1^j$,}
 \end{equation}

By Caccioppoli's inequality for $L$-subharmonic functions and the uniform boundedness of $u$ in $\bB$, we deduce that, for $i=1,2$,
\begin{align*}
\limsup_{j\rightarrow\infty}
\|\nabla u_{j}\|_{L^{2}(\frac12 \bB)}
&\lesssim \limsup_{j\rightarrow\infty} \|u_{j}\|_{L^2(\bB)}\\
&\lesssim \limsup_{j\rightarrow\infty}  \omega_{j}(2\bB) 
\leq \omega_{\infty}(\cnj{2\bB}).
\end{align*}

%See (3.7) of \cite{KPT09} for a similar argument.
By the Rellich-Kondrachov theorem, the unit ball of the Sobolev space $W^{1,2}(\frac12 \bB)$ is relatively
compact in $L^2(\frac12 \bB)$, and thus there exists a subsequence of the functions $u_{j}$ which
converges {\em strongly} in $L^2(\frac12 \bB)$ to another function $u_{\infty}\in L^2(\frac12 \bB)$. This and the above inequality imply \eqref{e:ufinbound}.

% Let $C_{j,k}$ be a constant elliptic matrix so that 
%\[
%\lim_{j} (kr_{j})^{1-n}\int_{B(\xi,kr_{j})} |A(x)-C_{j,k}|^{2}\,dx=0.\]
%By a diagonalization argument and compactness, we may pass to a subsequence so that for each $k$, $C_{j,k}$ converges to an elliptic matrix $C_{k}$. Notice now, that since each $C_k$ is constant matrix, we must in fact have that there exists a matrix $C$ so that $C_{k}=C$ for any $k$. Thus, we have
By the same diagonalization argument as in the proof of the previous lemma (although using \eqref{e:om-density} instead of $\inf \delta_j>0$ that we used in the previous lemma), we can pass to a subsequence so that, for some uniformly elliptic matrix $A_0$ with constant coefficients,
\begin{equation}
\label{e:AMlim}
\lim_{j} (Mr_{j})^{-1-n}\int_{B(\xi,Mr_{j}) \cap \om} |A(x)-A_0|=0, \;\; \mbox{ for all }M\geq 1.
\end{equation}
It easy to check that
\begin{equation*}
\int \vphi\,  d\omega_{j} = \int A_j \nabla u_{j} \cdot \nabla \vphi \,dx,
\end{equation*}
for any $C^\infty$ function $\vphi$ compactly supported in $\frac{1}{2}\bB$. Then passing to a limit, it follows that
\begin{equation}\label{eq302}
\int \vphi\,  d\omega_{\infty} = \int  A_{0} \nabla u_\infty\cdot \nabla \vphi \,dx,\,\,\, \text{for any}\,\, \vphi \in C^\infty_c(\tfrac{1}{2}\bB).
\end{equation}

\end{proof}

\begin{theorem}\label{azmotovo2}
Let $\Omega^{\pm} \subset \bR^{n+1}$ be disjoint domains. 
Let $\xi_{j}\in \d\Omega^{+}\cap \d\Omega^{-}$ and $L=-\div A \nabla$ be a uniformly elliptic operator in $\om^+ \cup \om^-$ such that \eqn{boundarylimit} holds at $\xi$ with respect to both $\Omega^{+} \cup \Omega^{-}$. If $\omega^{\pm}=\omega_{\Omega^{\pm}}^{L_{A},x^\pm}$ is the $L_{A}$-harmonic measure with pole at $x^{\pm}\in \Omega^{\pm}$, and if there is $r_{j}\rightarrow 0$ and $c_{j}>0$ so that 
\[
\omega_{j}^{+}:= c_{j}T_{\xi_{j},r_{j}}[\omega^{+}] \rightarrow \omega_{\infty} %\in %\Tan(\omega^{+},\xi)\]
\]
and 
\[
\omega_{j}^{-}:= c_{j}T_{\xi_{j},r_{j}}[\omega^{-}] \rightarrow c\,\omega_{\infty} 
%\in \Tan(\omega^{-},\xi), 
\]
for some constant  $c>0$, then there is a subsequence such that the following hold. If $u^{\pm}(x)=G_{\Omega^{\pm}}(x,x^{\pm})$ on $\Omega^{\pm}$,  $u(x)=0$ on $(\Omega^{\pm})^{c}$ and 
\[u_{j}^{\pm}(x)=c_{j}\,u^{\pm}(xr_{j}+\xi_{j})\,r_{j}^{n-1},\]
then $u_{j}:=u_{j}^{+}-c^{-1}u_{j}^{-}$ converges in $L^{2}(\tfrac{1}{2} \bB)$ to a nonzero function $u_{\infty}$, which is $L_{A_{0}}$-harmonic in $\tfrac{1}{2} \bB$ for some constant uniformly elliptic matrix $A_{0}$, and moreover,
\begin{equation}\label{zerozet}
\tfrac{1}{2}\bB\cap \supp \omega_{\infty}=\{u_{\infty}=0\}\cap \tfrac{1}{2} \bB
\end{equation}
and \eqref{e:ufinbound} and \eqref{e:ibp} hold. If $\xi_{j}=\xi$ and  $A$ is continuous at $\xi$, then $A_{0}$ is just the value of $A$ at $\xi$.

By applying this result to the sequences $c_{j} T_{\xi_{j},ar_{j}}[\omega^{\pm}]$ for all $a>0$, we see that $u_{\infty}$ extends to a $L_{A_{0}}$-harmonic function on $\mathbb{R}^{n+1}$ so that for $r>0$,
\begin{equation}\label{e:ufinbound2}
||u_{\infty}||_{L^{2}(B(0,r))} \lec r^{1-n}\,\omega_{\infty}(\cnj{B(0,4r)}),
\end{equation} 
and for any $\vphi \in C^\infty_c(\R^{n+1})$,
 \begin{align}\label{e:ibp2}
 \int \vphi \,d\omega_{\infty} = \int_{\R^{n+1}}  u_{\infty}\, L_{A_{0}}\vphi.
 \end{align}

\end{theorem}

\begin{proof}
	The proof is mostly the same as the proof of \cite[Lemma 5.3]{AMTV16}, but we provided some of the details here to show the differences. Again, we assume $\xi_{j}=\xi$. %By a change of variables, we may assume that $A_{0}=I$. 
	Note that since $\Omega^+$ and $\Omega^-$ are disjoint, we may assume without loss of generality that 
\[
|B(\xi,r_{j}/8)\backslash \Omega^+|\geq \frac{|B(\xi,r_{j}/8)|}{2}\]
and so Bourgain's estimate implies 
\[
\omega^{+,z}(B(\xi, r_{j}/2))\gec 1 \mbox{ for all }z\in B(\xi,r_{j})).\]
Hence, the conclusions of Lemma \ref{l:azmotovo} apply to $\omega=\omega^{+}$, $\Omega=\Omega^{+}$ and $u=u^{+}$. In particular,  \eqref{e:Green-lowerbound3} in our scenario is 
 \begin{equation}\label{e:Green-lowerbound4}
\omega_{j}^{+}(2 \bB)\gtrsim \, u_{j}^{+}(x)\quad\mbox{
 for all $x\in \bB\cap\Omega_1^j$.}
 \end{equation}

Again, by rescaling, we can assume that $\omega_{\infty}(\frac{1}{4} \bB)=1$.

Observe now that for any non-negative $\vphi \in C^\infty_c( \frac{1}{2}\bB)$ with $\phi = 1$ in $ \frac{1}{4}\bB$, by Cauchy-Schwartz and Caccioppoli's inequality (since $u_j^\pm$ is positive and $L$-harmonic in $\bB \cap \Omega_j^\pm$ and zero in $\bB \setminus \Omega_j^\pm$) we have that
\begin{align*}
1 = &\omega_{\infty}(\frac{1}{4} \bB)
 \leq \int \vphi \,d\omega_{\infty}
=\int A_{0} \nabla u_{\infty}^{+} \cdot \nabla \vphi  \,dx\\
& =\lim_{j} \int_{\Omega_{j}^{+}} A_j \nabla u_{j}^{+} \cdot \nabla \vphi \, dx\\
& \leq \|A\|_{L^\infty} \|\nabla \vphi \|_{L^{\infty}(\bB)} \lim_{j} \int_{\Omega_{j}^{+} \cap \frac{1}{2} \bB} |\nabla u_{j}^{+}|\\
& \lec \|A\|_{L^\infty} \|\nabla \vphi \|_{L^{\infty}(\bB)}  \lim_{j} \left(\int_{\Omega_{j}^{+} \cap  \bB} |u_{j}^{+}|^2 \right)^{1/2}\\
& \lesssim \lim_{j} \ps{\int_{\bB\cap \Omega_{j}^{+} \cap \{u_{j}^{+}>t\}} |u_{j}^{+}|^2 \, dx
+ \int_{\bB\cap \Omega_{j}^{+} \cap \{u_{j}^{+}\leq t\}} |u_{j}^{+}|^2 \,dx}^{1/2}\\
& \lesssim\liminf_{j} \ps {| \{x\in \bB\cap \Omega_{j}^{+} : u_{j}^{+}>t\} |^{1/2} \cdot ||u_{j}^{+}||_{L^{\infty}(\bB \cap \Omega_{j}^{+})}}
+ t  \\
& \stackrel{\eqn{Green-lowerbound4}}{\lec} \liminf_{j} \ps{|\{x\in \bB\cap \Omega_{j}^{+} : u_{j}^{+}>t\}|^{1/2}\, \omega_{\infty}\left(\cnj{2\bB}\right)
+ t },
\end{align*}
and so, for $t$ small enough,
\[
|\bB\cap \Omega_{j}^{+}|\geq |\{x\in \bB\cap \Omega_{j}^{+}: u_{j}^{+}(x)>t\}|\gec \omega_{\infty}(\cnj{2\bB})^{-2}.\]
In particular,
\begin{equation}\label{e:om-nondeg}
|B(\xi,r_{j})\backslash \Omega^{-}|\geq |B(\xi,r_{j})\cap \Omega^+|\gec r_{j}^{n+1}\omega_{\infty}(\cnj{2\bB})^{-2}.
\end{equation}

Thus, by the same arguments as earlier in proving \eqn{Green-lowerbound3}, we have that for $j$ large,
\begin{equation}\label{e:Green-lowerbound4}
 \omega_{j}^{-}(B(\xi,2r_j))\gtrsim \, u^{-}_{j}(x)\,\omega_{\infty}(\cnj{2\bB})^{-2}, \quad\mbox{
 for all $x\in B(\xi,r_j)\cap\Omega^-$.}
 \end{equation}
 
Thus, we can apply Lemma \ref{l:azmotovo} and can pass to a subsequence so that $u_{j}^{-}$ converges in $L^{2}(\frac{1}{2}\bB)$ to a function $u_{\infty}^{-}$. Hence, $u_{j}^{+}-c^{-1}u_{j}^{-}\rightarrow u_{\infty}^{+}-c^{-1}u_{\infty}^{-}=:u_{\infty}$ and
 \begin{equation}\label{e:intphu-}
 c\int \vphi \, d\omega_{\infty} = \int L_{A_{0}^*}\vphi  \,u_{\infty}^{-}\,dx, \quad \textup{for any}\,\, \phi \in C^\infty_c(\tfrac{1}{2}\bB).
\end{equation}
In particular, we can show that $u_{\infty}$ is $L_{A_0}$-harmonic in $\tfrac{1}{2}\bB$, and the rest of the proof is exactly as in \cite{AMTV16} starting from equation (5.15).

\end{proof}

\vvv

\section{Harmonic Polynomial Measures}
\subsection{Preliminaries}
In this section, we review and collect some lemmas that will help us work with the quantities $\omega_{h}^{A}$.

\begin{lemma}\label{l:dilating}
Let $h\in H_{A}$ and $r>0$. Then
\begin{equation}
\label{e:trw}
T_{0,r}[\omega_{h}^{A}]= r^{n-1} \omega_{h\circ T_{0,r}^{-1}}^{A}
\end{equation}
and 
\begin{equation}\label{e:htr}
F_{r}(\omega_{h}^{A}) = r^{n} F_{1}(\omega_{h\circ T_{0,r}^{-1}}^{A}).
\end{equation}
\end{lemma}

\begin{proof}
By \Lemma{pushforpol}, it suffices to prove this in the case that $h\in H$. 
Note that if $h$ is a harmonic function and $\vphi\in C_{c}^{\infty}(\bR^{n+1})$, then
\begin{align*}
\int \vphi\, d T_{0,r}[\omega_{h}] = \int \vphi\circ T_{0,r}& \,d\omega_{h}\\
=\int h\,\Delta(\vphi\circ T_{0,r}) & \,dx=r^{-2} \int h\, \Delta\vphi \circ T_{0,r}  \,dx\\
 =r^{n-1}\int &h \circ T_{0,r}^{-1} \,\Delta \vphi\, dx = r^{n-1} \int \vphi \,d \omega_{h\circ T_{0,r}^{-1}},
\end{align*}
and so \eqn{trw} follows. Moreover, by Lemma \ref{preiss} (3),
\begin{equation}
F_{r}(\omega_{h})
=rF_{1}(T_{0,r}[\omega_{h}])
\stackrel{\eqn{trw}}{=} r^{n} F_{1}(\omega_{h\circ T_{0,r}^{-1}}).
\end{equation}
\end{proof}

\begin{lemma}\label{l:rn-1+k}
Let $h\in F_{A}(k)$ and $r>0$. Then
\begin{equation}
\label{e:rn-1+k}
F_{r}(\omega_{h}^{A})=r^{n+k}F_{1}(\omega_{h}^{A}).
\end{equation}
\end{lemma}

\begin{proof}
Note that since $h$ is homogeneous of degree $k$, 
\[
h\circ T_{0,r}^{-1}(x)=h(rx)=r^{k}h(x),\]
and thus, by \eqn{htr},
\[
F_{r}(\omega_{h}^{A})
=r^{n} F_{1}(\omega_{h\circ T_{0,r}^{-1}}^{A})
=r^{n} F_{1}(\omega_{r^{k} h}^{A})
=r^{n+k} F_{1}(\omega_{ h}^{A}).\]
\end{proof}

The following is an immediate consequence of \Lemma{dilating}

\begin{lemma}[Lemma 4.1 \cite{Bad11}] \label{l:Pdcone} $\cF_{A}(k)$, $\cP_{A}(k)$, and $\mathscr{H}_{A}$ are $d$-cones. Hence, so are $\cF_{\cS}(k), \cP_{\cS}(k)$, and $\mathscr{H}_{\cS}$ for any $\cS\subset \cC$.
\end{lemma}

\begin{lemma}\label{l:uc}
Let $A_{j}\in\cC$ converge to a matrix $A\in\cC$ and let $h_{j}\in H_{A_{j}}$ converge uniformly on compact subsets to some $h\in H_{A}$. Then $\omega_{h_{j}}^{A_{j}}\rightarrow \omega_{h}^{A}$ weakly.
\end{lemma}

\begin{proof}
First we will deal with the case that $A_{j}=A=I$ for all $j$. 

We first claim that, since $h$ and $h_{j}$ are harmonic, $\one_{\Omega_{h_{j}}}\rightarrow \one_{\Omega_h}$ a.e.. Indeed, if $\one_{\Omega_{h}}(x)=1$, then $h(x)>0$, and by uniform convergence, $h_{j}(x)>0$ for all large $j$, and so $\one_{\Omega_{h_j}}(x)=1$ for all large $j$; similarly, if $\one_{\Omega}(x)=0$, then either $x\in \d\Omega_{h}$ (which has measure zero) or $h_{j}(x)<0$ for all large $j$, in which case $\one_{\Omega_{h_{j}}}(x)=0$ for all large $j$. Thus, $\one_{\Omega_{h_{j}}}\rightarrow \one_{\Omega_{h}}$ pointwise everywhere in $(\d\Omega_{h})^{c}$ and thus a.e. in $\R^{n+1}$.  In particular, $h_{j}\one_{\Omega_{j}}\rightarrow h\one_{\Omega}$ a.e.. Hence, for $\vphi\in C_{c}^{\infty}(\bR^{n+1})$, by the dominated convergence theorem,
\[
\lim_{j\rightarrow\infty} \int \vphi \,d\omega_{h_{j}}
 =\lim_{j\rightarrow\infty}  \int_{\Omega_{h_{j}}}h_{j} \, \Delta \vphi 
 =\int_{\Omega_{h}} h\,\Delta \vphi 
=\int \vphi\, d\omega_{h},\]
which implies $\omega_{h_{j}} \rightharpoonup \omega_{h}$ as $j \to \infty$.

Now we handle the general case. Let ${A}_{j,s}=(A_{j}+A_{j}^{T})/2$, and $S_{j}=\sqrt{{A_{j,s}}}$, and define ${A_s}$ and $S$ similarly. Let $\tilde A_j$ and $\tilde A$ be defined as in \eqref{e:matrixform}, and let $\tilde{h}=h\circ S$ and $\tilde{h}_{j}=h_{j}\circ S_{j}$. Since $\sqrt{\cdot}$ is continuous on the set of real symmetric matrices, $\tilde{h}_{j}\rightarrow \tilde{h}$ uniformly on compact subsets and both are harmonic. Thus, $\omega_{\tilde{h}_{j}}\warrow \omega_{\tilde{h}}$, and so
\[
\lim_{j\rightarrow\infty} \omega_{h_{j}}^{A}
\stackrel{\eqn{pushforpol}}{=}\lim_{j\rightarrow\infty}(\det S_{j}) S_j[\omega_{\tilde{h}_{j}}]
=(\det S) S[\omega_{\tilde{h}}] \stackrel{\eqn{pushforpol}}{=} \omega_{h}^{A}.\]

\end{proof}

%
%
%\begin{lemma}\label{l:hdcone}
%The set $\cH$ is a $d$-cone.
%\end{lemma}
%
%\begin{proof}
%If $h\in H$, then the Taylor series $h_{k}$ for $h$ converges uniformly on compact subsets to $h$ and $h_{k}\in P_{k}$, and so $\omega_{h_{k}}\in \cP(k)$. By \Lemma{Pdcone}, $T_{0,r}[\omega_{h_{k}}]\in \cP(k)$ as well for every $k$. Since $h_{j}\circ T_{0,r}^{-1}\rightarrow h\circ T_{0,r}^{-1}$ uniformly as well, by \Lemma{uc}, $\omega_{h_{j}\circ T_{0,r}^{-1}}\rightarrow \omega_{h\circ T_{0,r}^{-1}}$. By \eqn{trw}, this implies 
%\end{proof}

\begin{lemma}\label{l:hbound}
If $A\in\cC$ and $h\in P_{A}(k)$ for some $k \in \bN$, then
\begin{equation}\label{e:hbound}
||h||_{L^{\infty}(\bB)}\lec_{k,\Lambda} F_{1}(\omega_{h}^{A}).
\end{equation}
\end{lemma}

\begin{proof}
Suppose instead that there exist $A_{j}\in\cC$  and $h_{j}\in P_{A_{j}}(k)$ for which $||h_{j}||_{L^{\infty}(\bB)} >j F_{1}(\omega_{h_{j}}^{A_{j}})$. Without loss of generality, we may assume $||h_{j}||_{L^{\infty}(\bB)} =1$, and thus $F_{1}(\omega_{h_{j}}^{A_{j}})\rightarrow 0$. Using Cauchy estimates,  $ \{h_{j}\}_{j=1}^\infty$ forms a normal family in $\bB$ and thus, we can pass to a subsequence so that $h_{j}$ converges uniformly on compact subsets of $\bB$ and so that $A_{j}$ converges to some $A\in\cC$. Since all $h_{j}$ are polynomials of order $k$, we know that the coefficients of $h_{j}$ converge, which, in turn, implies that $h_{j}$ converges to some function $h\in \cP_{\cC}(k)$ uniformly on compact subsets of $\bR^{n+1}$. By \Lemma{uc}, $\omega_{h_{j}}^{A_{j}}\rightarrow \omega_{h}^{A}$. In particular, 
\[F_{1}(\omega_{h}^{A})=\lim_{j\rightarrow\infty} F_{1}(\omega_{h_{j}}^{A_{j}})=0.\]
Thus, $0\not\in \supp \omega_{h}$. We will now show that in fact $0\in \supp \omega_{h}^{A}$ in order to get a contradiction. 

First, by \Lemma{pushforpol}, we can assume without loss of generality that $A=I$ and $\omega_{h}^{A}=\omega_{h}$. Secondly, notice that as $h_{j}\in \cP_{\cC}(k)$, $h\in \cP(k)$ and so $h(0)=0$. By Lojasiewicz's structure theorem for real analytic varieties (see e.g. \cite[Theorem 6.3.3,  p.168]{KP}), if $U$ is a small enough neighborhood of a point $0\in \Sigma_{h}$, we have that 
$$U \cap \Sigma_{h}= V^{n} \cup V^{n-1} \cup \dots \cup V^0,$$
where $V^0$ is either the empty set or the singleton $\{0\}$ and for each $i\in \{1, \dots, n\}$, we may write $V^i$ as a finite, disjoint union $V^i = \bigcup_{j=1}^{N_k} \Gamma^i_j$, of $i$-dimensional real analytic submanifolds. Further, for each $1 \leq i \leq n-1$, 
$$U \cap \overline {V^i} \supset V^{i-1} \cup \dots \cup V^0.$$
Moreover,  for $1 \leq k \leq n$ and $1 \leq j \leq N_k$, $U \cap \d \Gamma^i_j$ is a union of sets of the form $\Gamma^\ell_m$, for $1 \leq \ell <i$ and $1 \leq m \leq N_\ell$ and possibly $V^0$. 

By the main result in \cite{CNV15}, $\dim \{\grad h=0\}\leq n-1$, and thus $V^{n}\cap\{\grad h=0\}$ is a closed set of relatively empty interior in $V^{n}$, so in particular, 
\[\cnj{V^{n}\backslash \{\grad h=0\}}\cap U=\cnj{V^{n}}\cap U=\Sigma_{h}\cap U\ni 0.\]
For $\zeta\in U\cap V^{n}\backslash  \{\grad h=0\}$, the derivative of $h$  at $\zeta$ tangent to $V^{n}$ is always zero, as $h$ is zero on $V^{n}$, which forces $\grad h$ to be perpendicular to $V^{n}$. Since the normal derivative is nonzero, 
\[
U\cap V^{n}\backslash  \{\grad h=0\}\subset \ck{\zeta\in U\cap V^{n}:\frac{\partial h}{\partial \nu}\neq 0}\subset U\cap V^{n}\cap \supp \omega_{h}.\]
Thus, $0\in U\cap \cnj{V^{n}\backslash \{\grad h=0\}}\subset \supp\omega_{h}$, which gives us the contradiction and concludes the proof.\\

\end{proof}

\subsection{Proof of Proposition \ref{p:pcompact}}

\Proposition{pcompact} will follow from the following more general result.

\begin{lemma}\label{l:ccompactbasis}
	Let $\cS\subset\cC$ be closed (hence compact). Then  $P_{\cS}(k)$ and $\cF_{\cS}(k)$ have compact basis. 
\end{lemma}
\begin{proof}
Let $h_{j}\in P_{A_{j}}(k)$ with $A_{j}\in \cS$ and assume $\cF(\omega_{h_{j}}^{A_{j}})=1$. Then by \eqn{hbound} and Cauchy estimates, we can  bound each coefficient of the polynomials $h_{j}$ uniformly, and then pass to a subsequence so that $A_{j}\rightarrow A\in \cS$ and $h_{j}$ converges on compact subsets of $\bR^{n+1}$ to a function $h\in P_{A}(k)\subset P_{\cS}(k)$. By Lemma \ref{l:uc}, we have that $\omega_{h_{j}}\rightarrow \omega_{h}$, which implies that $\cP_{\cS}(k)$ has compact basis. The proof for $\cF_{\cS}(k)$ is similar. 
\end{proof}

As a corollary, we show  the following stronger version of \eqn{hbound}.

\begin{corollary} For $h\in P_{\cC}(k)$ and $r>0$,
\begin{equation}\label{e:hbound2}
||h||_{L^{\infty}(r\bB)}\approx_{k} r^{-n}F_{r}(\omega_{h}).
\end{equation}
\end{corollary}

\begin{proof}
Let $h\in P_{\cC}(k)$ and $\vphi\in C_{c}^{\infty}(\bR^{n+1})$ be such that $\one_{\frac{1}{2}\bB}\leq \vphi \leq \one_{\bB}$. Since $\cP_{\cC}(k)$ has compact basis by Lemma \ref{l:ccompactbasis}, we can estimate
\begin{align*}
F_{1}(\omega_{h})
& \stackrel{\eqn{compact}}{\lec}
F_{1/2}(\omega_{h})
\leq \int \vphi d\omega_{h}
=\int_{\Omega_{h}}h\, \Delta \vphi 
\leq ||\Delta\vphi||_{\infty} \int_{\bB} |h|
\\ & \lec ||h||_{L^{\infty}(\bB)}
 \stackrel{\eqn{hbound}}{\lec} F_{1}(\omega_{h}).
\end{align*}
For $r\neq 1$, by the previous inequalities we have
\[
F_{r}(\omega_{h})
\stackrel{\eqn{htr}}{=}
r^{n} F_{1}(\omega_{{h}\circ T_{0,r}^{-1}})
\approx r^{n} ||h\circ T_{0,r}^{-1}||_{L^{\infty}(\bB)}
\approx r^{n} ||h||_{L^{\infty}(r\bB)}.\]
\end{proof}

\subsection{Proof of \Proposition{taninf}}

%In \cite{Bad11}, \Lemma{KPT} could not be used directly with $\cF=\cF(k)$ and $\cM=\cP(k)$ as it was not known whether $\cP(k)$ had compact basis. Instead, Badger used the fact that harmonic measure for NTA domains is doubling, which implies that the tangent measures at each point form compact cones. In the absence of the doubling property, however, one cannot use this. 

\begin{lemma}\label{l:taylor}
Let $h\in H_{A}$ and
\[
h(x)=\sum_{j=m}^{\infty}\sum_{|\alpha|=j}\frac{D^{\alpha} h(0)}{\alpha!}x^{\alpha}
=\sum_{j=m}^{\infty} h_{j}(x)\]
be its Taylor series, where $m>0$, which converges uniformly to $h$ on compact subsets of $\bR^{n+1}$. Then $\Tan(\omega_{h}^{A},0)=\{c\,\omega_{h_{m}}^{A}:c>0\}$.
\end{lemma}

\begin{proof}
For notational convenience, we will just consider the case $A=I$, the general case is identical. Note that as $r\rightarrow 0$, $r^{-m} h\circ T_{0,r}^{-1}\rightarrow h_{m}$ uniformly on compact subsets of $\bR^{n+1}$. Indeed, fix $R>0$. Then the series
\[
r^{-m}\sum_{j=m}^{\infty}\sum_{|\alpha|=j}\frac{D^{\alpha} h(0)}{\alpha!}(rx)^{\alpha} 
=\sum_{j=m}^{\infty}\sum_{|\alpha|=j}\frac{D^{\alpha} h(0)}{\alpha!}x^{\alpha} r^{|\alpha|-m}
\]
converges uniformly to $r^{-m} h\circ T_{0,r}^{-1}$ on compact subsets of  $B(0,R)$, provided $r$ is small enough. In fact, by Cauchy estimates,
$$|D^{\alpha}h(0)|\lec_{n}|\alpha|^{|\alpha|},$$ 
and since there exists a constant $C>1$ such that $\frac{k^{k}}{k!}\lec C^{k}$, then, for $x\in B(0,R)$ and $r\in(0,\frac{1}{CR})$, we have that
\begin{multline*} 
\av{r^{-m} h\circ T_{0,r}^{-1}(x)-h_{m}(x)}
 \leq \sum_{j=m+1}^{\infty}\sum_{|\alpha|=j}\av{\frac{D^{\alpha} h(0)}{\alpha!}}R^{|\alpha|}r^{|\alpha|-m} \\
 \lec_{n,m} \sum_{j=m+1}^{\infty} C^{j}R^{j}r^{j-m}
\lec r^{-m} (CRr)^{m+1}= (CR)^{m+1} r
\stackrel{r \downarrow 0}{\longrightarrow} 0.
\end{multline*}
%
%one can show that $\av{\sum_{|\alpha|=j}\frac{D^{\alpha} h(0)}{\alpha!}}$ are bounded in $j$. Hence, for $x\in B(0,R)$,
%\[
%\av{r^{-m} h\circ T_{0,r}^{-1}(x)-h_{m}(x)}\\
%\leq \sum_{j=m+1}^{\infty}\sum_{|\alpha|=j}\av{\frac{D^{\alpha} h(0)}{\alpha!}}R^{|\alpha|}r^{|\alpha|-m} 
%\rightarrow 0 
%\]
%as $r\rightarrow 0$ since $|\alpha|>m$ in this sum. 

Let now
\[\nu_{r}:=r^{-m-n+1}\,T_{0,r}[\omega_{h}]\stackrel{\eqn{trw}}{=} r^{-m}\,\omega_{h\circ T_{0,r}^{-1}}=\omega_{r^{-m} h\circ T_{0,r}^{-1}}.\] 
By \Lemma{uc}, $\nu_{r}\rightharpoonup \omega_{h_{m}}\in \cF(m)$. In particular, every tangent measure of $\omega_{h}$ at zero must be a multiple of this one.
\end{proof}

We now state an interesting consequence of these results: that if a portion of tangent measures of an arbirary Radon measure are harmonic polynomials, then they are all homogeneous polynomials. 

\begin{lemma}
Let $\omega$ be a Radon measure, $\xi\in \supp \omega$, and $k$ be the minimal integer such that $\Tan(\omega,\xi)\cap \cP(k)\neq\varnothing$, then $\Tan(\omega,\xi)\cap \cP(k)\subset \cF(k)$. 
\label{l:mink}
\end{lemma}

We follow the proof in \cite[Lemma 5.9]{Bad11}, which originally supposed that $\omega$ was harmonic measure for an NTA domain. 

\begin{proof}
If $k=1$, then $\cP(1)=\cF(1)$. Now suppose $k>1$ and there is  $h\in P(k)$ non-homogeneous such that $\omega_{h}\in \Tan(\omega,\xi)\cap \cP(k)$. Since $h\in \cP(k)$, we may write
\[
h(x)=\sum_{j=m}^{k}\sum_{|\alpha|=j}\frac{D^{\alpha} h(0)}{\alpha!}x^{\alpha}
=\sum_{j=m}^{k} h_{m}(x),\]
where $m<k$ since $h\in \cP(k)$ is not homogeneous. By Lemma \ref{l:taylor}, $\Tan(\omega_{h},0)=\{c \omega_{h_{m}}:c>0\}\subset \cF({m})$, and since $\Tan(\omega_{h},0)\subset \Tan(\omega,\xi)$ by \Theorem{ttt}, $\Tan(\omega,\xi)\cap \cF({m})\neq\varnothing$, contradicting the minimality of $k$. Thus, $\Tan(\omega,\xi)\cap \cP(k)\subset \cF(k)$. 
\end{proof}

We will also need the following result. 

\begin{lemma}[\cite{Bad11} Lemma 4.7] 
\label{l:bad}
Suppose $h\in P(m)$ for some $m$. There exist $\ve=\ve(n,m,k)>0$ and $r_{0}>0$ so that if $d_{r}(\omega_{h},\cF(k))<\ve$ for all $r\geq r_{0}$, then $m=k$.
\end{lemma}

\begin{proof}[Proof of \Proposition{taninf}]
Suppose $\Tan(\omega,\xi)\subset \cP(k)$. Let $m$ be the minimal integer for which $\Tan(\omega,\xi)\cap \cP(m)\neq\varnothing$, so $m\leq k$. Then, by Lemma \ref{l:mink}, $\Tan(\omega,\xi)\cap \cP(m)\subset \cF(m)$. In particular, $\Tan(\omega,\xi)\cap \cF(m)\neq\varnothing$. Since, by \Proposition{pcompact}, $\cP(k)$ has compact basis, we can use \Lemma{bad} and \Lemma{newpreiss} to conclude $\Tan(\omega,\xi)\subset \cF(m)$.
\end{proof}

\section{Proof of Theorem \ref{newKPT}}

\begin{lemma}\label{logw}
	Let $\cS\subset \cC$ be closed and  $\omega=\omega_{\Omega}^{A,x}$ be an $L_{A}$-harmonic measure where $A\in \cA$ and $L_A \in \VMO(\om, \xi)$ at $\xi\in \supp \omega$. Also assume we have $\Tan(\omega,\xi)\subset\mathscr{H}_{\cS}$. Let $k$ be the smallest integer for which $\Tan(\omega,\xi)\cap \cF_{\cS}(k)\neq \varnothing$. Then $\Tan(\omega,\xi)\subset \cF_{\cS}(k)$. In particular,
	\begin{equation}\label{pointwisedim}
	\lim_{r\rightarrow 0} \frac{\log \omega(B(\xi,r))}{\log r} = n+k-1.
	\end{equation}
\end{lemma}

\begin{proof}
	If $\Tan(\omega,\xi)\not\subset \cF_{\cS}(k)$, then by \Corollary{pstep}, there is $r_{0}>0$ so that for any $\ve>0$ small we may find $\nu\in \Tan(\omega,\xi)\backslash \cF_{\cS}(k)$ so that $d_{r_{0}}(\nu,\cF_{\cS}(k))=\ve $ and $d_{r}(\nu,\cF_{\cS}(k))\leq \ve $ for all $r\geq r_{0}$. Without loss of generality, we can assume  $r_{0}=1$. For each $r>1$, choose $\mu_{r}\in \cF_{\cS}(k)$ such that $F_{r}(\mu_{r})=1$ and 
	\[
	F_{r}\ps{\frac{\nu}{F_{r}(\nu)},\mu_{r}}<2\ve.\]
		Then for $r\geq 1$,
	\begin{align*}
	\frac{F_{r}(\nu)}{F_{2r}(\nu)}
	& =\int(r-|x|)_{+}d\frac{\nu}{F_{2r}(\nu)}
	< 2\ve + \int(r-|x|)_{+}d\mu_{2r}
	=2\ve + F_{r}(\mu_{2r})\\
	& \stackrel{\eqn{rn-1+k}}{=} 2\ve  + 2^{-n-k} F_{2r}(\mu_{2r})
	= 2\ve  + 2^{-n-k}
	=2^{-n-k+\beta},
	\end{align*}
	for some $\beta>0$ that goes to zero as $\ve\rightarrow 0$. Similarly,
	\[\frac{F_{r}(\nu)}{F_{2r}(\nu)}\geq 2^{-n-k-\beta}.\]
	Hence, for $\ell\in \bN$,
	\begin{equation}\label{e:+b}
	2^{\ell(n+k-\beta)} \leq \frac{F_{2^{\ell} r}(\nu)}{F_{r}(\nu)} \leq 2^{\ell(n+k+\beta)}.\end{equation}
	Note that $\nu=\omega_{h}^{A}$ for some $h\in \mathscr{H}_{A}$ by Lemma \ref{azmotovo2} and $A\in \cS$, and so
	\begin{align}
	||h||_{L^{\infty}(2^{\ell}\bB)}
	& \stackrel{\eqn{ufinbound2}}{\lec} 2^{\ell(1-n)} \omega_{h}(B(0,2^{\ell+1})) 
	\leq 2^{-\ell n-1}  F_{2^{\ell+2}}(\omega_{h})   \notag \\
	&  \stackrel{\eqn{+b}}{\leq} 2^{\ell(k+\beta)-1}F_{2^{2}}(\omega_{h}).
	\label{e:hrb<}
	\end{align}
	Let $\alpha$ be a multi-index of length $|\alpha|>k$. Then we can pick $\ve>0$ small enough so that $\beta$ is so small  that  $|\alpha|-k-\beta > 0$ holds. Thus,
	%\begin{align*}
	%|\d^{\alpha}h(0)|
	%& \lec 2^{-\ell |\alpha|} ||h||_{L^{\infty}(2^{\ell}\bB)}
	%\stackrel{\eqn{u<wr}}{\lec} 2^{\ell(-|\alpha|+1-n)} \omega_{h}(B(0,2^{\ell+1})) \\ 
	%& \leq 2^{\ell(-|\alpha|-n)-1}  F_{2^{\ell+2}}(\omega_{h})  
	% \stackrel{\eqn{+b}}{\leq} 2^{\ell(-|\alpha|-1+k+\beta)-1}F_{2^{2}}(\omega_{h})
	%\rightarrow 0
	%\end{align*}
	by Cauchy estimates,
	\begin{align*}
	|\d^{\alpha}h(0)|
	& \lec_\alpha 2^{-\ell |\alpha|} ||h||_{L^{\infty}(2^{\ell}\bB)} 
	\stackrel{\eqn{hrb<}}{\lec} 2^{-\ell(|\alpha|-k-\beta)}F_{2^{2}}(\omega_{h}) 
	\rightarrow 0
	\end{align*}
	as $\ell\rightarrow\infty$, and so $h \in \cP_{A}(k)$.
	
%	 Suppose $h=\sum_{j=1}^{k+1}h_{j}$ where $h_{j}\in \cF_{A}(j)$ and $h_{k+1}\neq 0$. Then for $r>1$ large, if $\ell$ is the integer such that $2^{\ell-1}\leq r<2^{\ell}$,
%	
%	\begin{align*}
%	r^{k+1} ||h_{k+1}||_{L^{\infty}(\bB)}
%	& =||h_{k+1}||_{L^{\infty}(r\bB)}
%	\approx ||h||_{L^{\infty}(r\bB)}
%	\leq ||h||_{L^{\infty}(2^{\ell}\bB)}\\
%	& \stackrel{\eqn{hrb<}}{\lec} 2^{\ell(k+\beta)}F_{2^{2}}(\omega_{h})
%	\lec_{k,\beta} r^{k+\beta}F_{2^{2}}(\omega_{h})
%	\end{align*}
%	which is impossible for $r$ large since $\beta<1$. Thus, we must have $h_{k+1}=0$, and so $\omega_{h}\in \cP_{A}(k)$. 
%	
		 Suppose $h=\sum_{j=1}^{k}h_{j}$. If $\omega_{h}\not\in \cF_{A}(k)$, then there exists $j<k$ such that $h_{j}\neq 0$, and by \Lemma{taylor}, we infer that $\Tan(\omega_{h}^{A},0)$ contains an element of $\cF_{A}(j)$. Since $\omega_{h}^{A}\in \Tan(\omega,\xi)$, we know that $\Tan(\omega_{h}^{A},0)\subset  \Tan(\omega,\xi)$ and thus, $ \Tan(\omega,\xi)\cap \cF_{A}(j)\neq\varnothing$. Hence $ \Tan(\omega,\xi)\cap \cF_{\cS}(j)\neq\varnothing$,  contradicting the minimality of $k$. This proves $\Tan(\omega,\xi)\subset \cF_{\cS}(k)$. 
	
	For the final equality, note that $\Tan(\omega,\xi)\subset \cF_{\cS}(k)$ and so $\Tan(\omega,\xi)$ has compact basis. In particular, by Lemma \ref{l:pstep},
	\[
	\lim_{r\rightarrow 0} d_{1}(T_{\xi,r}[\omega],\cF_{\cS}(k))= 0.\]
%\mih{I just added that this follows from Lemma 3.10. The proof in ``Indeed... contradiction" is unnecessary and should be removed.}\jonas{cool}Indeed, if there was  $r_{j}\downarrow 0$ so that $d_{1}(T_{\xi,r_{j}}[\omega],\cF_{\cS}(k))\geq \ve >0$, then we could pass to a subsequence so that $T_{\xi,r_{j}}[\omega]\rightharpoonup \omega'\in \cF_{\cS}(k)$. Since $\cF_{\cS}(k)$ is compact, $F_{1}(\omega')>0$, and so Lemma \ref{preiss} (8) implies 
%	\[
%	\ve \leq d_{1}(T_{\xi,r_{j}}[\omega],\cF_{\cS}(k))\rightarrow d_{1}(\omega',\cF_{\cS}(k))=0,\]
%	a contradiction. 
	
	Thus, for $\ve>0$, there is $r_{0}>0$ such that for each $r\leq r_{0}$, there exists $\mu_{r}\in \cF_{\cS}(k)$ so that $F_{1}(\mu_{r})=1$ and
	\[
	F_{1}\ps{\frac{T_{\xi,r}[\omega]}{F_{1}(T_{\xi,r}[\omega])},\mu_{r}}<\ve .
	\]
	Setting $\nu_{r}=r^{-1}T_{\xi,r}^{-1}[\mu_{r}]$, this gives $F_{r}(\nu_{r})=1$ and 
	\[
	F_{r}\ps{\frac{\omega}{F_{r}(\omega)},\nu_{r}}<\ve.
	\]
	By the same arguments as earlier, we can show that there exists $\gamma>0$, which goes to zero as $\ve\rightarrow 0$, so that for all $\ell\geq 0$ and $r< 2^{-\ell-1}r_0$,
	\jonas{yeah}
	\begin{equation}
	2^{\ell(n+k-\gamma)} \leq \frac{F_{2^{\ell} r}(\omega)}{F_{r}(\omega)} \leq 2^{\ell(n+k+\gamma)}.
	\end{equation}
	Hence, if we set $d=n+k-1$, we get 
	\begin{align*}
	\omega(B(\xi,2^{\ell}r))
	& =T_{\xi,r}[\omega](B(0,2^{\ell}))
	\leq 2^{-\ell}F_{2^{\ell+1}}(T_{\xi,r}[\omega])\\
	& \leq  2^{(\ell+1)(n+k+\gamma)-\ell}F_{1}(T_{\xi,r}[\omega])\\
	& \leq 2^{\ell (d+\gamma) +n+k+\gamma }  T_{\xi,r}[\omega](B(0,1))\\
	& = 2^{\ell (d+\gamma) +n+k+\gamma }  \omega(B(\xi,r)).
	\end{align*}
	Similarly,
	\begin{align*}
	\omega(B(\xi,r)) 
	& =T_{\xi,r}[\omega] (B(0,1))
	\leq F_{2}(T_{\xi,r}[\omega])\\
	& \leq 2^{-(\ell-1)(n+k-\gamma)} F_{2^{\ell}}(T_{\xi,r}[\omega])\\
	& \leq 2^{-(\ell-1)(n+k-\gamma)+\ell}  \omega(B(\xi,2^{\ell}r))\\
	& = 2^{-\ell (d-\gamma)+n+k-\gamma}\omega(B(\xi,2^{\ell}r)).
	\end{align*}
	For $r<r_{0}/2$, let $\ell\in \bN$ be so that $2^{-\ell-1}r_{0}\leq r\leq 2^{-\ell}r_{0}$. Then
	\begin{align*}
	\omega(B(\xi,r)) \leq \omega(B(\xi,2^{-\ell}r_{0}))
	& \leq 2^{-\ell (d-\gamma)+n+k-\gamma}\omega(B(\xi,r_{0}))\\
	& \leq 2^{1+(n+k-\gamma)} r^{d-\gamma}  \omega(B(\xi,r_{0})).
	\end{align*}
	Hence, recalling that these logs are negative, we conclude
	\[
	\liminf_{r\rightarrow 0} \frac{\log \omega(B(\xi,r))}{\log r} 
	\geq \liminf_{r\rightarrow 0} \frac{\log \ps{2^{1+(n+k-\gamma)}\omega (B(\xi,r_{0}))}}{\log r} + d-\gamma
	= d-\gamma.
	\]
	A similar estimate gives 
	\[
	\limsup_{r\rightarrow 0} \frac{\log \omega(B(\xi,r))}{\log r} \leq d+\gamma.
	\]
	If we let $\gamma\rightarrow 0$, then  \eqref{pointwisedim} follows.

\end{proof}

%\begin{remark}
%	Notice that we may assume that we only have two different uniformly elliptic matrices defined on each in $\Omega_i$, $i=1,2$ which agree just on $\d \Omega_1 \cap\d\Omega_2$. Indeed, we then define a new matrix $A\in \cA$ that agrees with $A_i$ in $\Omega_i \cup [\d \Omega_1 \cap\d\Omega_2]$, for $i=1,2$ and the identity matrix everywhere else. 
%\end{remark}

\begin{proof}[Proof of Theorem \ref{newKPT}]

%\begin{lemma}\label{l:mainlem}
%Let $\Omega_{1},\Omega_{2}\subset \bR^{n+1}$ be disjoint domains and let $\omega_{i}$ be the $L$-harmonic measure in $\Omega_i$. Suppose there is $E\subset \d\Omega_{1}\cap \d\Omega_{2}$ upon which we have $\omega_{1}|_{E}\ll \omega_{2}|_{E}\ll \omega_{1}|_{E}$. Fix $\ve<1/100$ and let $E_{m}$ be the set of $\xi\in E$ such that for all $0<r<1/m$ and $i=1,2$ we have 
%\begin{equation}\label{e:widub}
%\omega_{i}(B(\xi,2r))\leq m\, \omega_{i}(B(\xi,r)),
%\end{equation} 
%\begin{equation}\label{e:ominondeg}
%\cH^{n+1}(B(\xi,r)\cap \Omega_{i})\geq  \frac1m\,r^{n+1},
%\end{equation}
%and
%\begin{equation}\label{e:betasmall}
%\beta_{\omega_{i},i}(\xi,r)<\ve \frac{\omega_{i}(\xi,r)}{r^{n}}
%\end{equation}
%Then 
%\begin{equation}\label{eq:EMM}
%\omega_i \biggl(E\setminus \bigcup_{m\geq1} E_m\biggr)=0.
%\end{equation}
%\end{lemma}

We set 
\begin{align*} 
E^{*}&=\ck{\xi\in E: \lim_{r\rightarrow 0} \frac{\omega^{+}(E\cap B(\xi,r))}{\omega^{+}(B(\xi,r))}= \lim_{r\rightarrow 0} \frac{\omega^{-}(EE\cap B(\xi,r))}{\omega^{-}(B(\xi,r))}=1}\, \textup{and}\\
E^{**}&= \{\xi\in E^*: \eqref{e:boundarylimit} \mbox{ holds}\}. 
\end{align*}
Notice that by \cite[Corollary 2.14 (1)]{Mattila} and because $\omega_{1}$ and $\omega_{2}$ are mutually absolutely continuous on $E$,
\[\omega^{+}(E\backslash E^{**})= \omega^{-}(E\backslash E^{**})=0.\]
Also, set 
\begin{multline*}
\Lambda_{1} =\left\{\xi\in E^{**}\!\!: 0<h(\xi):=\frac{d\omega^{-}}{d\omega^{+}}(\xi)=\lim_{r\rightarrow 0} \frac{\omega^{-}( B(\xi,r))}{\omega^{+}( B(\xi,r))} \right. \\
=\left. \lim_{r\rightarrow 0} \frac{\omega^{-}(E\cap B(\xi,r))}{\omega^{+}(E\cap B(\xi,r))}<\infty\right\}\end{multline*}
%\[
%\Lambda_{2}=\ck{\xi\in E: \lim_{r\rightarrow 0} \frac{\omega^{-}(E\cap B(\xi,r))}{\omega^{+}(E\cap B(\xi,r))}=\infty}\]
%\[
%\Lambda_{3}=\ck{\xi\in E: \lim_{r\rightarrow 0} \frac{\omega^{-}(E\cap B(\xi,r))}{\omega^{+}(E\cap B(\xi,r))}=0}\]
%\[ \Lambda_{4}=\ck{\xi\in E:\lim_{r\rightarrow 0} \frac{\omega^{-}(E\cap B(\xi,r))}{\omega^{+}(E\cap B(\xi,r))} \mbox{ does not exist}}\]
and 
\[
\Gamma = \ck{\xi\in \Lambda_{1}: \xi \mbox{ is a Lebesgue point for $h$ with respect to }\omega^{+}}.
\]
Again, by Lebesgue differentiation for measures (see \cite[Corollary 2.14 (2) and Remark 2.15 (3)]{Mattila}), $\Gamma$ has full measure in $E^{**}$ and hence in $E$.
%\vv
%
%To prove \eqref{eq:EMM}, it suffices to show that for $\omega_{1}$-almost every $\xi\in \Gamma$, we have 
%\begin{equation}\label{e:limdub}
%\limsup_{r\rightarrow 0} \frac{\omega^{+}(B(\xi,2r))}{\omega^{+}(B(\xi,r))}<\infty,
%\end{equation}
%\begin{equation}\label{e:limcomp}
%\liminf_{r\rightarrow 0} \min_{i=1,2} \frac{\cH^{n+1}(B(\xi,r)\cap \Omega_{i})}{r^{n+1}}>0,
%\end{equation}
%and
%\begin{equation}\label{e:limbeta}
%\lim_{r\rightarrow 0} \beta_{\omega_{1},1}(\xi,r) \frac{r^{n}}{\omega^{+}(B(\xi,r))}=0.
%\end{equation}
%We then use some standard measure theory to find our desired sets $E_m$.

%
%\begin{equation}
%\label{e:zerolambdas}
%\omega^{+}(\Lambda_{2})=\omega^{-}(\Lambda_{3})=\omega^{+}(\Lambda_{4})=\omega^{-}(\Lambda_{4})=0
%\end{equation}

%The following is proven in \cite[Lemma 5.8]{AMT}. There we assume a capacity density condition, but the assumption is not used in the proof. 
Next, we record a lemma which was proven in \cite{AMT16} in the case of the harmonic functions in domains that satisfy the CDC condition, but its proof goes through unchanged  for $L$-harmonic functions in general domains.

\begin{lemma}\label{l:samew}
	Let $\xi\in \Gamma$, $c_{j}\geq 0$,  and $r_{j}\rightarrow 0$ be so that $\omega_{j}^{+}=c_{j}T_{\xi,r_{j}}[\omega^{+}]\rightarrow \omega_{\infty}$. Then $\omega_{j}^{-}=c_{j}T_{\xi,r_{j}}[\omega^{-}]\rightarrow h(\xi)\omega_{\infty}$. 
\end{lemma}

\vv
We define
\[
\cF:=\{c\cH^{n}|_{V}: c>0, \;\; V\mbox{ a $d$-dimensional plane containing the origin}\}.\]
It is not hard to show that $\cF$ has compact basis.

%\begin{lemma}\label{l:capf}
%For this proof we will let $\bB=B(0,1)$. Let $\xi\in \d\Omega_{1}\cap \d\Omega_{2}$ be such that there is $h(\xi)\in (0,\infty)$ so that if $r_{j}\downarrow 0$ and $c_{j}>0$ are such that $c_{j} T_{\xi,r_{j}}[\omega^{+}]\rightarrow \omega_{\infty}\in \Tan(\omega_{1},\xi)$, then $c_{j} T_{\xi,r_{j}}[\omega^{-}]\rightarrow h(\xi)\omega_{\infty}$. (This happens for  $\omega_{i}$-a.e. $\xi\in \Gamma$, for example.) Then $\Tan(\omega_{1},\xi)\cap \cF\neq\varnothing$. 
%\end{lemma}

\begin{lemma}\label{l:capf}
	For $\omega^{+}$-a.e.\ $\xi\in\Gamma$, $$\Tan(\omega^{+},\xi)\cap \cF\neq\varnothing.$$ 
\end{lemma}

\begin{proof}
	
	We can pick $\xi\in \Gamma$ so that $\Tan(\omega^{+},\xi)\neq \varnothing$, let $\omega_{\infty}\in \Tan(\omega^{+},\xi)$, so there is $c_{j}>0$ and $r_{j}\downarrow 0$ so that $c_{j}T_{\xi,r_{j}}[\omega^{+}]\rightarrow \omega_{\infty}$. By Lemma \ref{l:samew}, we also have $c_{j}T_{\xi,r_{j}}[\omega^{-}]\rightarrow h(\xi)\omega_{\infty}$. By Lemma \ref{azmotovo2}, \eqref{zerozet} holds.

	In particular, $\frac{1}{2}\bB\cap \supp \omega_{\infty}$ is a smooth real analytic variety, and arguing as in \cite{AMTV16}, for example, one deduces that
	$$d\omega_\infty|_{\frac{1}{2}\bB}= -c_n(\nu_{\Omega_{\infty}^{+}} \!\cdot\! A_{0}\nabla u_{\infty})\, d\cH^n|_{
		\d^*\Omega_{\infty}^{+}\cap \frac{1}{2}\bB},$$
	where $A_{0}$ is the matrix from Lemma \ref{azmotovo2}, $\d^*\Omega_{\infty}^{+}$ is the reduced boundary of $\Omega_{\infty}^{+}=
	\{u_{\infty}>0\}$ and $\nu_{\Omega_{\infty}^{+}}$ is the measure theoretic outer unit normal.
	Hence,  
	$\omega_{\infty}$ is absolutely continuous with respect to surface measure of $\d\Omega_{\infty}^{+}$ in $\frac{1}{2}\bB$. Thus, since the tangent measure at $\cH^n$-almost every point of $\d\Omega_{\infty}^{+}$ is contained in $\cF$,  we can take another tangent measure of $\omega_{\infty}$ that is in $\cF$ and apply Theorem \ref{t:ttt} to conclude the proof.
\end{proof}
\vv

By Lemma \ref{logw}, we also have that $\dim \omega^{+}|_{E}=n$. It remains to show that, if $\Omega^{\pm}$ both have the CDC, then 
$\lim_{r\rightarrow 0}\Theta_{\d\Omega^{+}}^{\cF}(\xi,r)=0$ for $\omega^{+}$-a.e. $\xi\in E$. But this follows almost immediately because, for almost every $\xi\in \Gamma$ and any $r_{j}\downarrow 0$, we may pass to a subsequence so that, by Lemma \ref{l:azmoto} (a) and (f), $\lim_{j\rightarrow \infty}\Theta_{\d\Omega^{+}}^{\cF}(\xi,r_{j})=0$. This finishes the proof.

\end{proof}

\section{$\BMO$, $\VMO$, and vanishing $A_{\infty}$}\label{s:vmo}

In this section, we will prove some estimates relating the logarithm of a Radon-Nikodym derivative to the mutual absolute continuity continuity properties of two measures. We will apply them to the specific case of harmonic measure, but we will prove them for general measures.

\begin{definition}
Let $\mu$ be a non-negative measure on a metric space $X$. We say that a function $f \in L^1_{loc}(\mu)$ is of {\it bounded mean oscillation} and write $f \in \BMO(\mu)$, if there exists a constant $C>0$ such that
\begin{equation}\label{def:BMO}
\sup_{r \in (0, \infty)} \sup_{x \in \supp{\mu}} \,\,\avint_{B(x,r)} | f - f_{B(x,r)} | d\mu \leq C,
\end{equation}
where $f_{A}:= \avint_A f \,d\mu := \mu(A)^{-1} \int_A f\,d\mu$, for any $A \subset X$ with $\mu(A)>0$. 
We define the space of {\it vanishing mean oscillation} $\VMO(\mu)$ to be the closure in the $\BMO(\mu)$ norm of the set of bounded uniformly continuous functions defined on $X$. Equivalently, we say $f \in \VMO(\mu)$ if $f \in L^1_{loc}(\mu)$ and 
\begin{equation}\label{def:VMO}
 \lim_{r \to 0} \sup_{x \in \supp{\mu}} \,\,\avint_{B(x,r)} | f - f_{B(x,r)} | d\mu =0,
\end{equation}
\end{definition}
\begin{definition}
For two measures $\mu$ and $\nu$ on a metric space $X$, we will say $\nu\in A_{\infty}(\mu)$ if $\mu\ll\nu$ there is $K=K(\mu,\nu)$ so that for any ball $B$ centered on the support of $\mu$,
\begin{equation}\label{e:K}
\avint_{B} \frac{d\nu}{d\mu}d\mu\, \exp\ps{-\avint_{B}\log \frac{d\nu}{d\mu}d\mu}\leq K(\mu,\nu). 
\end{equation}
We will say $\nu\in A_{\infty}'(\mu)$ if there are $\ve,\delta\in (0,1)$ so that for all $B\subseteq X$ and $E\subseteq B$,
\begin{equation}
\label{e:afinclassic}
\frac{\mu(E)}{\mu(B)} <\delta \mbox{ implies } \frac{\nu(E)}{\nu(B)}<\ve.
\end{equation}
We will say $\nu\in VA_{\infty}(\mu)$ (or {\it vanishing $A_{\infty}$ with respect to $\mu$}) if 
\begin{equation}
\lim_{r\rightarrow 0}\sup_{\xi\in \supp \mu}\,\, \avint_{B} \frac{d\nu}{d\mu}d\mu\,  \exp\ps{-\avint_{B}\log \frac{d\nu}{d\mu}d\mu}=1
\end{equation}
and $\nu\in VA_{\infty}'(\mu)$ if for all $r>0$ there is $\ve_{r}\in (0,1)$ so that $\lim_{r\rightarrow 0} \ve_{r}= 0$ and $\delta_{r}>0$ so that for all balls $B\subset X$ with $r_{B}<r$ and $E\subset B$,
\begin{equation}\label{e:xafinB}
\frac{\mu(E)}{\mu(B)}<\delta_{r} \;\;\; \Rightarrow \;\;\; \frac{\nu(E)}{\nu(B)}<\ve_{r}.
\end{equation}
\end{definition}

In the case that $X=\bR^{n+1}$ and $\mu$ is equal to the $(n+1)$-dimensional Lebesgue measure, $A_{\infty}$ equivalence is the same as $A_{\infty}'$-equivalence, and this is from Reimann and Rychener \cite{RR75}, although it was also shown later by Hru\v{s}\v{c}ev in \cite{Hru84} and Garc\'{i}a-Cuerva and Rubio de Francia in \cite{GCRF}.

We recall a notion introduced by Korey \cite{Kor98}.

\begin{definition}
A probability space $(X,\mu)$ is {\it halving} if every subset $E\subset X$ of positive measure has a subset $F\subset E$ so that $\mu(F)=\mu(E)/2$.
\end{definition}

We will first focus on proving the following after a series of other lemmas.

\begin{lemma}\label{l:vmolemma}
Let $(X,\mu)$ be a metric measure space, $\nu\ll \mu$, and $f=\frac{d\nu}{d\mu}$. 
\begin{enumerate}
\item If $\nu\in A_{\infty}'(\mu)$ and $\log f \in \BMO(\mu)$, then $\nu \in A_{\infty}(\mu)$. If $X$ is also halving, then $\nu \in A_{\infty}(\mu)$ implies $\nu\in A_{\infty}'(\mu)$ and $\log f \in \BMO(\mu)$. 
\item If $\nu\in VA_{\infty}'(\mu)$ and $\log f \in \VMO(\mu)$, then $\nu \in VA_{\infty}(\mu)$. If $X$ is also halving, then $\nu \in VA_{\infty}(\mu)$ implies $\nu\in VA_{\infty}'(\mu)$ and $\log f \in \VMO(\mu)$. 
\end{enumerate}
\end{lemma}

%Showing classic $A_{\infty}$ implies $A_{\infty}$ usually requires some doubling condition on $\mu$, but in the other direction we have the following estimate.
The first implication of the second half of (1) of the lemma follows from the following theorem.
\begin{theorem}[Theorem 1, \cite{Hru84}]\label{t:hru}
Suppose $\nu\ll \mu$, $B$ is a ball centered on $\supp \mu$, and 
\[
\avint_{B} \frac{d\nu}{d\mu}d\mu  \exp\ps{-\avint_{B}\log \frac{d\nu}{d\mu}d\mu}\leq C.\]
Then there are $\ve,\delta>0$ so that, for any $F\subset B\cap \supp \mu$, 
\begin{equation}\label{e:oldafin}
\frac{\mu(F)}{\mu(B)} <\delta \mbox{ implies } \frac{\nu(F)}{\nu(B)}<\ve . 
\end{equation}
Moreover, there is $\delta>0$ so that 
\begin{equation}\label{e:hru}
\frac{\mu(F)}{\mu(B)}<\delta \mbox{ implies }  \frac{\nu(F)}{\nu(B)}< 2(C-1).
\end{equation}
In particular, if $\nu\in A_{\infty}(\mu)$, then $\nu\in A_{\infty}'(\nu)$, and if $\nu\in VA_{\infty}(\mu)$, then $\nu\in VA_{\infty}'(\mu)$. 
\end{theorem}

\begin{proof}
We follow the proof from \cite[Theorem 1]{Hru84}, since he proves \eqref{e:oldafin} but not \eqref{e:hru}. Let $\delta\in (0,1)$ to be chosen later, $F\subseteq B$ and suppose $\mu(F)=\delta \mu(B) $, we will pick $\delta$ later. Let $f=\frac{d\nu}{d\mu}$, $E=B\backslash F$, and set
\[
t=\frac{\nu(E)}{\nu(F)}.\]
Let $g_{B}=\avint_{B} fd\mu$. Then
\begin{equation}\label{e:hru1}
\log C
\geq (\log f^{-1} )_{B}+\log f_{B}
= \frac{\mu(E)}{\mu(B)} (\log f^{-1})_{E}+\frac{\mu(F)}{\mu(B)}(\log f^{-1})_{F}+\log f_{B}.
\end{equation}
By Jensen's inequality, for any set $S$
\[
(\log f^{-1})_{S}=-(\log f)_{S}\geq - \log f_{S}\]
and applying this to $S=E,F$, we have 
\begin{align*}
\log C
& \geq -\frac{\mu(E)}{\mu(B)} \log f_{E} - \frac{\mu(F)}{\mu(B)} \log f_{F}+\log f_{B}\\
& \geq -\frac{\mu(E)}{\mu(B)} \log f_{E} - \frac{\mu(F)}{\mu(B)} \log f_{E} + \frac{\mu(F)}{\mu(B)} \log \frac{\mu(F)}{\mu(E)} + \frac{\mu(F)}{\mu(B)} \log t   \\
& \qquad  +\log f_{B}\\
& = -\log f_{E} + \frac{\mu(F)}{\mu(B)} \log \frac{\mu(F)}{\mu(E)} + \frac{\mu(F)}{\mu(B)} \log t   +\log f_{B}.\\
\end{align*}
Now observe that
\[
-\log f_{E} =  \log \ps{ \frac{\mu(E)}{\mu(B)} \frac{\mu(B)}{\nu(B)} \frac{\nu(B)}{\nu(E)}}
=\log \frac{\mu(E)}{\mu(B)} - \log f_{B}  +\log\ps{1+\frac{1}{t}}\]
and so we have 
\begin{align*}
\log C
& \geq\log \frac{\mu(E)}{\mu(B)} +  \log \ps{1+\frac{1}{t}} + \frac{\mu(F)}{\mu(B)} \log \frac{\mu(F)}{\mu(E)} +  \frac{\mu(F)}{\mu(B)} \log t  \\
& =  \frac{\mu(F)}{\mu(B)} \log \frac{\mu(F)}{\mu(B)} + \frac{\mu(E)}{\mu(B)}\log \frac{\mu(E)}{\mu(B)}+  \log (1+t) + \frac{\mu(E)}{\mu(B)} \log \frac{1}{t} \\
& =\underbrace{\delta \log \delta + (1-\delta)\log (1-\delta) }_{=:\phi(\delta)} +\log (1+t) +\frac{\mu(E)}{\mu(B)} \log \frac{1}{t}. \\
\end{align*}

Note that $\lim_{\delta\rightarrow 0}\phi(\delta)=0$. Let $\alpha>0$ and pick $\delta>0$ so that $|\phi(\delta)|<\alpha \log C$. Then
\begin{equation}\label{e:logeq}
(1+\alpha)\log C 
\geq  \log (1+t) + \frac{\mu(E)}{\mu(B)}\log\frac{1}{t}.
\end{equation}
We restrict $\delta$ further so that $\delta<\alpha$. If $t>1$, then $\frac{\mu(E)}{\mu(B)}\log\frac{1}{t}\geq \log\frac{1}{t}$; otherwise, $\frac{\mu(E)}{\mu(B)}\log\frac{1}{t}\geq (1-\alpha)\log\frac{1}{t}$ since $\frac{\mu(E)}{\mu(B)}=1-\delta>1-\alpha$. Thus, in any case, we have
\begin{equation}
\frac{1+\alpha}{1-\alpha}\log C > \log\frac{1}{t}.
\end{equation}
This implies $t\geq c=C^{-(1+\alpha)/(1-\alpha)}$, and so 
\[
\nu(F)
=\frac{\nu(F)}{1+t}  + \frac{t\nu(F)}{1+t} 
=\frac{\nu(F)+\nu(E)}{1+t}
=\frac{\nu(B)}{1+t} \leq \frac{\nu(B)}{1+c} .\]
This proves \eqref{e:oldafin} with $\ve= (1+c)^{-1}$. To prove \eqref{e:hru}, we go back to \eqref{e:logeq} with the same bound on $\delta$. Then, since $t\geq c$,
\begin{align*}
(1+\alpha)\log C 
& \geq  \log (1+t) + \frac{\mu(E)}{\mu(B)}\log\frac{1}{t}
= \log\ps{1+\frac{1}{t}} +  \frac{\mu(F)}{\mu(B)}\log t \\
& \geq \log\ps{1+\frac{1}{t}}-\delta \frac{1+\alpha}{1-\alpha} \log C.
\end{align*}
Since $\delta<\alpha$, this implies 
\begin{align*}
 \log\ps{1+\frac{1}{t}} & < \ps{1+\alpha +\delta \frac{1+\alpha}{1-\alpha}} \log C
 =(1+\alpha)\ps{1+\frac{\delta}{1-\alpha}}\log C \\
 & <\frac{1+\alpha}{1-\alpha} \log C,
 \end{align*}
and so
\[
C^{(1+\alpha)/(1-\alpha)} -1> \frac{1}{t}.\]
We now pick $\alpha$ so that $C^{(1+\alpha)/(1-\alpha)} -1=2(C-1)$, and we are done.

%
%
%
%If $C<3/2$ and $t\geq 1/(C-1)$, then $\log (C-1)<0$, so picking $\delta<-\frac{\log C}{2\log (C-1)}$, we get 
%\[
%\frac{1}{2} \log C \geq   \delta \log (C-1)+ \log \ps{1+\frac{1}{t}}
%> -\frac{1}{2}\log C + \log \ps{1+\frac{1}{t}}\]
%and so 
%\[
%\log C > \ps{1+\frac{1}{t}}\]
%hence
%\[
%C>e^{1+\frac{1}{t}} \geq 1+\frac{1}{t}
%\]
%\[
%\frac{1}{4} \log C \geq \log\ps{1+\frac{1}{t}}\]
%thus
%\[
%C\geq \ps{1+\frac{1}{t}}^{4} \geq 1+\frac{4}{t}.\]
%Therefore, we have 
%\[
%\frac{\nu(F)}{\nu(B)}\leq \frac{1}{t} \leq \frac{C-1}{4}.\]
\end{proof}

%For the case of $\mu$ equal to Lebesgue measure, this is \cite[Theorem 1]{Hru84} (in particular, look at equations (5) and (7)), but the proof for general $\mu$ is exactly the same. 

Showing that $VA_{\infty}$ implies the logarithm of the density is $\VMO$ was shown by Korey.

\begin{theorem}[Theorem 4 and Section 3.5 \cite{Kor98}] \label{t:kor}
There is a universal constant $c>0$ so that the following holds. Let $(X,\mu)$ be a halving probability space, and suppose that
\begin{equation}
\ps{\int_{X} \exp gd\mu}/ \exp\ps{\int_{X} gd\mu }\leq K.
\end{equation}
Then
\begin{equation}
\int_{X} \av{g-\int_X gd\mu}d\mu\leq \log 2K
\end{equation}
and as $K\rightarrow 1$,
\begin{equation}\label{eq:korKto1}
\int_{X} \av{g-\int_X gd\mu}d\mu\leq c\sqrt{K-1}
\end{equation}
\end{theorem}

\begin{lemma}\label{l:afincafin}
let $(X,\mu)$ be a metric probability space and suppose $\nu\ll\mu$. Let $\ve,\delta\in (0,1)$ be so that for any $E\subset X$,
\begin{equation}\label{e:xafin}
\mu(E)<\delta\mu(X) \;\;\; \Rightarrow \;\;\; \nu(E)<\ve \nu(X).
\end{equation}
Set $f=\frac{d\nu}{d\mu}$ and assume
\begin{equation}\label{e:f<eta}
\avint_{X}\av{\log f - \avint_{X} \log f d\mu} d\mu <\eta.
\end{equation}
Then
\begin{equation}\label{e:e^c+eta/d}
1\leq \avint_{X} fd\mu \exp\ps{-\avint_{X}\log fd\mu} \leq \frac{e^{\eta/\delta}}{1-\ve}. 
\end{equation}
\end{lemma}

\begin{proof}
Without loss of generality, we may assume $\mu(X)=\nu(X)=1$. Let $\ve>0$ and pick $\delta$ so that \eqn{xafin} holds.

%Since $\log f\in BMO$, there is $r_{0}>0$ so that if $B=B(\xi,r)$ and $r<r_{0}$,

Let $c=\int_{X}\log f\,d\mu$ and
\begin{equation}\label{e:gdef}
G=\{|\log f-c|< \rho:=\eta\delta^{-1}\}, \;\;\; F=G^{c}.\end{equation}
Then, by Chebysev's inequality and \eqn{f<eta}, we infer that $\mu(F)<\delta$, which, in turn, by \eqn{xafin}, implies
\begin{equation}\label{e:f-}
\nu(F)<\ve.
\end{equation}
Moreover, on the set $G$,
\[
\frac{\eta}{\delta}>|\log f-c|\]
and so
\begin{equation}\label{e:ong}
f\leq e^{c+\eta/\delta} \mbox{ on } G.
\end{equation}
%\max\{f,f^{-1}\}\leq \min\{e^{c}, e^{-c}\}\, e^{\eta/\delta}\mbox{ on }G.\end{equation}
%=\av{\log\frac{f}{e^{c}}}>\frac{1}{2} \max\ck{\av{\frac{f}{e^{c}}-1},\av{\frac{e^{c}}{f}-1}}. \]
Then, 
\begin{align*}
1=\frac{\nu(X)}{\mu(X)}=\int_{X} fd\mu
& \stackrel{\eqn{ong}}{\leq} \ps{\int_{G} e^{c+\eta/\delta}d\mu+\int_{F}fd\mu}\\
& \leq e^{c+\eta/\delta}+\nu(F)
\stackrel{\eqn{f-}}{<}e^{c+\eta/\delta}+ \ve.
\end{align*}
Thus,
\[
(1-\ve)\int_{X} fd\mu  = 1-\ve < e^{c+\eta/\delta}\]
and so
\[
\int_{X}fd\mu < \frac{e^{c+\eta/\delta}}{1-\ve}.\]
This and Jensen's inequality imply
\begin{equation}\label{e:ec}
1\leq e^{-c}\int_{X}fd\mu<e^{-c} \frac{1}{1-\ve}e^{c+\eta/\delta}=\frac{1}{1-\ve}e^{\eta/\delta}.
\end{equation}

\end{proof}

\begin{corollary}\label{c:afincafin}
let $(X,\mu)$ be a metric measure space. Set $f=\frac{d\nu}{d\mu}$ and assume that for some sequence of balls $B_{j}$ in $X$,
\begin{equation}\label{e:f<etaseq}
\lim_{j} \avint_{B_{j}}\av{\log f - \avint_{B_{j}} \log f d\mu} d\mu =0.
\end{equation}
and for all $\ve>0$ there is $\delta>0$ so that for $j$ sufficiently large,
\begin{equation}\label{e:afinsequence}
\frac{\mu(E)}{\mu(B_{j})} <\delta \mbox{ implies } \frac{\nu(E)}{\nu(B_{j})} <\ve.
\end{equation}
Then
\begin{equation}\label{e:lmsupe}
\lim_{j\rightarrow \infty} \avint_{B_{j}} fd\mu \exp\ps{-\avint_{B_{j}}\log fd\mu} =1. 
\end{equation}
In particular, if $\log f\in \VMO(d\mu)$ and $\nu\in VA_{\infty}'(\mu)$, then $\nu\in VA_{\infty}(\mu)$.
\end{corollary}

\begin{proof}
Let $\ve,\eta>0$ and let $\delta>0$ be so that \eqref{e:afinsequence} holds for $j$ large enough. Then \eqn{f<eta} holds (with $B_{j}$ in place of $X$ and $\mu|_{B_{j}}$ in place of $\mu$). Then \eqn{e^c+eta/d} must hold. In particular,
\[
\limsup_{j\rightarrow \infty}  \avint_{B_{j}} fd\mu \exp\ps{-\avint_{B_{j}}\log fd\mu}\leq \frac{e^{\eta/\delta}}{1-\ve}.\]
As $\ve$ and $\delta$ did not depend on $\eta$, we can send $\eta\rightarrow 0$, and then $\ve\rightarrow 0$ since $\delta$ now vanishes from the inequality, and then we obtain \eqn{lmsupe}. 
\end{proof}

\begin{proof}[Proof of \Lemma{vmolemma}]
The secolnd halves of (1) and (2)  follow from Theorems \ref{t:hru} and \ref{t:kor}. The first half of (1) follows from Lemma \ref{l:afincafin}, and the first half of (2) is from  \Corollary{afincafin}.
\end{proof}

\begin{lemma}\label{l:whalving}
Let $\Omega\subset \bR^{n+1}$ be any connected domain and $\omega=\omega_{\Omega}^{L_{A},x}$ where $A\in \cA(\Omega)$. Then $\omega$ is halving.
\end{lemma}

\begin{proof}
Suppose there is $E\subset \d\Omega$ with $\omega(E)>0$ that is not halving. For $t\in \bR$ and $v\in \bS^{n-1}$, let $H_{t,v}=\{x\in \bR^{n+1}: x\cdot v \geq t\}$. Then $t\mapsto \omega(H_{t,v}\cap E)$ is not continuous for any $v\in \bS^{n}$, and so there is $t_{v}$ so that $\omega(\d H_{t_{v},v}\cap E)>0$. Let $V_{v}=\d H_{t_{v},v}$, which is an $n$-dimensional plane. Since $\bS^{n}$ is uncountable, there is $\ve>0$ so that $\omega( V_{v}\cap E)>\ve>0$ for all $v$ in some uncountable set $A\subset \bS^{n}$. Let $A'\subset A$ be countable. Note that for any $u,v\in A'$ distinct, $V_{u}\cap V_{v}$ is an $(n-1)$-dimensional subspace. This implies $V_{u}\cap V_{v}$ has $2$-capacity zero \cite[Theorem 2.27]{HKM}, hence is a polar set for $\omega$ \cite[Theorem 10.1]{HKM} and polar sets have $L_{A}$-harmonic measure zero \cite[Theorem 11.15]{HKM}. Thus, if we set
\[
W_{u}:=V_{u}\backslash \bigcup_{v\in A' \atop v\neq u}V_{v},\]
 we have that $\omega(W_{u}\cap E)=\omega(V_{u}\cap E)\geq \ve$ and $W_{u}$ are mutually disjoint. But since $A'$ is infinite, this implies $\omega(E)=\infty$, which is a contradiction.
\end{proof}

\begin{lemma}\label{l:harmvmolemma}
Let $\Omega^{+}\subset \bR^{n+1}$ be a connected domain with connected complement $\Omega^{-}=\ext(\Omega^{+})$ and let $L_A$ be a uniformly elliptic operator with real coefficients. If $\omega^{\pm}$ denote the $L_A$-harmonic measures of  $\Omega^{\pm}$ with fixed poles $x^\pm \in \Omega^\pm$, then $\omega^{-}\in A_{\infty}(\omega^{+})$ if and only if $\omega^{-}\in A_{\infty}'(\omega^{+})$ and $\log \frac{d\omega^{-}}{d\omega^{+}} \in \BMO(d\omega^{+})$. Moreover, $\omega^{-}\in VA_{\infty}(\omega^{+})$ if and only if $\omega^{-}\in VA_{\infty}'(\omega^{+})$ and $\log \frac{d\omega^{-}}{d\omega^{+}} \in \VMO(d\omega^{+})$.
\end{lemma}

%If $\omega^{+}$ is doubling, then $\log f\in VMO$ implies $\omega^{-}\in A_{\infty}'(\omega^{+})$ and $A_{\infty}'(\omega^{+})=A_{\infty}(\omega^{+})$ by the John-Nirenberg inequality, and so we have the following corollary.
%
%\begin{lemma}\label{l:harmvmolemma2}
%Let $\Omega^{+}\subset \bR^{n+1}$ be a connected domain with connected complement $\Omega^{-}=\ext(\Omega^{+})$ and let $\omega^{\pm}$ denote their harmonic measures. Assume $\omega^{+}$ is doubling. Then $\omega^{-}\in A_{\infty}(\omega^{+})$ if and only if $\log \frac{d\omega^{-}}{d\omega^{+}} \in BMO(d\omega^{+})$. Moreover, $\omega^{-}\in VA_{\infty}(\omega^{+})$ if and only if $\log \frac{d\omega^{-}}{d\omega^{+}} \in VMO(d\omega^{+})$. 
%\end{lemma}

\begin{proof}
This follows from Lemmas \ref{l:vmolemma} and \ref{l:whalving}.
\end{proof}

\section{Proofs of  Theorems \ref{t:main} and \ref{t:main2}}

\begin{lemma}\label{l:vmo1}
Let $\omega^{\pm}$ be two halving Radon measures with equal supports 
%such that $\omega^{-}\in VA_{\infty}'(\omega^{+})$, 
and set $f=\log \frac{d\omega^{-}}{d\omega^{+}}$. Suppose there are $r_{j}\downarrow 0$ and $\xi_{j}\in \d\Omega^{+}$ so that $\omega_{j}^{+}=T_{\xi_{j},r_{j}}[\omega^{+}]/\omega(B(\xi_{j},r_{j}))$ converges weakly to some measure $\omega$ with $\omega(\bB)>0$. Further assume that for all $M>0$
\begin{equation}\label{e:xivmo}
\lim_{j} \avint_{B(\xi_{j},Mr_{j})}f \, d\omega^{+} \exp\ps{- \avint_{B(\xi_{j},Mr_{j})} \log f \,d\omega^{+}}=1.
\end{equation}
Then $\omega_{j}^{-}\rightharpoonup \omega$ as well.
\end{lemma}

The proof is similar to that of \cite[Theorem 4.4]{KT06}, though using the techniques of the previous section, we no longer require the doubling assumption.

\begin{proof}
Let $B_{j}=B(\xi_{j},r_{j})$ and for a ball $B$ set $c_{B}= \avint_{B} \log f$. 
By assumption, for each $M>0$,
\begin{equation}
e^{-c_{MB_{j}}}\frac{\omega^{-}(MB_{j})}{\omega^{+}(MB_{j})}\rightarrow 1\mbox{ as }j\rightarrow \infty.
\label{e:elim}
\end{equation}

Let $\vphi\in C_{c}^{\infty}(\bR^{n+1})$ with support in $B(0,M)$ for some $M>0$ and let $\vphi_{j}=\vphi\circ T_{\xi_{j},r_{j}}$. Then $\supp \vphi_{j}\subset MB_{j}$. Let $\ve>0$. By \eqn{elim}, for $j$ large enough, we have that
\begin{equation}\label{e:ebound}
0\leq e^{-c_{B_{j}}}\frac{\omega^{-}(B_{j})}{\omega^{+}(B_{j})} -1<\ve \;\;\; \mbox{ and }\;\;\; 0\leq e^{-c_{MB_{j}}}\frac{\omega^{-}(MB_{j})}{\omega^{+}(MB_{j})} -1<\ve.
\end{equation}
Let now $\eta= c \sqrt{1-\ve}$, where $c$ is the constant in \eqref{eq:korKto1}. For $j$ large enough, \Theorem{kor} %\Lemma{whalving}, 
and \eqn{elim} imply
\begin{equation}\label{e:f-cb<eta}
\avint_{B_j}|\log f-c_{B_j}| \,d\omega^{+}<\eta  \;\;\; \mbox{ and }\;\;\; \avint_{MB_j}|\log f-c_{MB_j}|\,d\omega^{+}<\eta.
\end{equation}
Note that $\ve$ is independent of $\eta$. For fixed $\delta>0$ and for a ball $B$, we set
\[
G_{B}=\{\xi\in B\cap \d\Omega^{+}: |\log f(\xi)-c_{B}|\leq \eta/\delta\}, \;\;\; F_{B}=B\backslash G_{B}.\]
Then, Chebyshev's inequality and \eqn{f-cb<eta} imply
\begin{equation}\label{e:w+d}
\omega^{+}(F_{B_j})<\delta \, \omega^{+}({B_j}) \quad\textup{and}\quad \omega^{+}(F_{MB_j}) <\delta\,  \omega^{+}({MB_j}),
\end{equation}
and for $\delta>0$ small enough and $j$ large enough, Theorem \ref{t:hru} and \eqref{e:elim} imply
\begin{equation}\label{e:w-e}
\omega^{-}(F_{B_j})<\ve\, \omega^{-}({B_j}) \quad\textup{and}\quad  \omega^{-}(F_{MB_j}) <\ve \, \omega^{-}({MB_j}).
\end{equation}
Let $C=2\frac{\omega(\cnj{M\bB})}{\omega(\bB)}$. Since $\omega(\bB)>0$, we know
\[
\limsup_{j\rightarrow\infty}
\frac{\omega^{+}(MB_{j})}{\omega^+(B_{j})}
=\limsup_{j\rightarrow\infty}\frac{\omega_{j}^{+}(M\bB)}{\omega_{j}^{+}(\bB)}
\leq \frac{\omega(\cnj{M\bB})}{\omega(\bB)}=C/2,\]
and so for $j$ large enough,
\begin{equation}\label{e:2b+}
\omega^{+}(MB_{j})\leq C\omega^{+}(B_{j}).
\end{equation}

Also, note that for $j$ large enough,
\begin{align}
|c_{B_{j}}-c_{MB_{j}}|& =\av{\avint_{B_{j}}(c_{B_{j}}-c_{MB_{j}})} \, d\omega^{+}\notag \\
&\leq \avint_{B_{j}}|c_{B_{j}}-\log f| \, d\omega^{+}+\avint_{B_{j}}|\log f-c_{MB_{j}}| \,d\omega^{+} \notag \\
& \stackrel{\eqn{f-cb<eta}}{<}\eta + \frac{\omega^{+}(MB_{j})}{\omega^{+}(B_{j})} \avint_{MB_{j}}|\log f-c_{MB_{j}}| \,d\omega^{+}\notag\\
&\stackrel{\eqn{f-cb<eta} \atop \eqn{2b+}}{<}(1+C)\eta.\label{e:cb-cmb}
\end{align}
Hence,
\begin{align}
\omega^{-}(MB_{j})
& \stackrel{\eqn{ebound}}{\leq} \omega^{+}(MB_{j})(1+\ve)e^{c_{MB_{j}}}\stackrel{\eqn{2b+} \atop \eqn{cb-cmb}}{<} C\omega^{+}(B_{j})(1+\ve)e^{c_{B_{j}}+(1+C)\eta} \notag \\
& \stackrel{\eqn{ebound}}{\leq} C\omega^{-}(B_{j})(1+\ve)e^{(1+C)\eta} \leq 2Ce^{(1+C)}\omega^{-}(B_{j})\notag\\
& \lec_{C} \omega^{-}(B_{j}).
\label{e:2b-}
\end{align}
Then
\begin{align*}
\int  & \vphi d\omega_{j}^{-}  -\int \vphi d\omega_{j}^{+}
 =\frac{1}{\omega^{-}(B_{j})}\int_{MB_{j}}\vphi_{j} d\omega^{-} -\frac{1}{\omega^{+}(B_{j})}\int_{MB_{j}}\vphi_{j} d\omega^{+}   \\
& =\underbrace{ \frac{1}{\omega^{-}(B_{j})}  \int_{MB_{j}\cap F_{MB_{j}}} \vphi_{j} f d\omega^{+}}_{=:I_{1}}\\
& \qquad + \underbrace{ \frac{1}{\omega^{-}(B_{j})} \int_{MB_{j}\cap G_{MB_{j}}}  \ps{ f -e^{c_{MB_{j}}} } \vphi_{j}  d\omega^{+} }_{=:I_{2}} \\
&  \qquad -\underbrace{ \frac{e^{c_{MB_{j}}} }{\omega^{-}(B_{j})} \int_{MB_{j}\cap F_{MB_{j}}} \vphi_{j}  d\omega^{+}  }_{=:I_{3}}
\\
& \qquad + \underbrace{\frac{e^{c_{MB_{j}}} }{\omega^{-}(B_{j})} \int_{MB_{j}}\vphi_{j} d\omega^{+}-\frac{1}{\omega^{+}(B_{j})}\int_{MB_{j}}\vphi_{j} d\omega^{+} }_{=:I_{4}} \\
& = I_{1}+I_{2}-I_{3}+I_{4}.
\end{align*}

We will estimate each of these terms separately, with the understanding that $j$ is large enough (depending on  $M$ and $\eta$). 
\begin{align*}
|I_{1}|
&  \leq  \frac{||\vphi||_{\infty}}{\omega^{-}(B_{j})}  \int_{{MB_{j}}} \one_{F_{MB_{j}}}  f d\omega^{+}\\
& =
  \frac{||\vphi||_{\infty}\omega^{-}(F_{MB_{j}})}{\omega^{-}(B_{j})} 
 = \frac{\omega^{-}(MB_{j})} {\omega^{-}(B_{j})}  \frac{||\vphi||_{\infty}{\omega^{-}(F_{MB_{j}})}}{\omega^{-}(MB_{j})} 
\stackrel{\eqn{w-e} \atop \eqn{2b-}}{\lec}_{C,M, ||\vphi||_{\infty}} \ve.
\end{align*}

Next, for points in $G_{MB_{j}}$,
\[
e^{-\eta/\delta}e^{c_{MB_{j}}} \leq  f\leq e^{\eta/\delta}e^{c_{MB_{j}}}
\]
and so
\[
e^{c_{MB_{j}}}(e^{-\eta/\delta}-1) \leq f-e^{c_{MB_{j}}} \leq e^{c_{MB_{j}}}(e^{\eta/\delta}-1).
\]
Thus, for $\eta>0$ small enough (i.e., for $j$ large enough), we can make 
\[
|f-e^{c_{MB_{j}}}| < \delta e^{c_{MB_{j}}}\;\; \mbox{ on }G_{MB_{j}}.
\]
Therefore,
\begin{align*}
|I_{2}| & \leq \frac{\delta e^{c_{MB_{j}}}||\vphi||_{\infty}}{\omega^{-}(B_{j})}\omega^{+}(G_{MB_{j}})
\leq  \frac{\delta e^{c_{MB_{j}}}||\vphi||_{\infty}}{\omega^{-}(B_{j})}\omega^{+}(MB_{j})
\\
& =e^{c_{MB_{j}}} \frac{\omega^{+}(MB_{j})}{\omega^{-}(MB_{j})} \frac{\delta ||\vphi||_{\infty}\omega^{-}(MB_{j})}{\omega^{-}(B_{j})}
\stackrel{\eqn{2b-} \atop \eqn{ebound}}{ \lec}_{||\vphi||_{\infty},C,M} \delta .
\end{align*}

\begin{align*}
|I_{3}|
& \leq \frac{e^{c_{MB_{j}}}||\vphi||_{\infty} }{\omega^{-}(B_{j})}  \omega^{+}(F_{MB_{j}})
\stackrel{\eqn{w+d}}{<}\delta \frac{e^{c_{MB_{j}}}||\vphi||_{\infty} }{\omega^{-}(B_{j})} \omega^{+}(MB_{j})\\
& =\delta \frac{e^{c_{MB_{j}}}||\vphi||_{\infty} \omega^{-}(MB_{j}) }{\omega^{-}(B_{j})} \frac{\omega^{+}(MB_{j})}{ \omega^{-}(MB_{j})}
\stackrel{\eqn{ebound}\atop\eqn{2b-}}\lec_{C,M, ||\vphi||_{\infty}} \delta.
\end{align*}

Finally,
\[
|I_{4}|
\leq\ps{e^{c_{MB_{j}}}\frac{\omega^{+}(B_{j})}{\omega^{-}(B_{j})}-1}\frac{\omega^{+}(MB_{j})}{\omega^{+}(B_{j})} \avint_{MB_{j}} \vphi_{j} d\omega^{+}
\stackrel{\eqn{ebound} \atop \eqn{2b+}}{\lec}_{C,||\vphi||_{\infty},M} \ve.
\]
Since these estimates hold for all $j$ large enough, we can conclude
\[\limsup_{j\rightarrow\infty} \av{\int   \vphi d\omega_{j}^{-}  -\int \vphi_{j} d\omega_{j}^{+}}
\lec_{C,M, ||\vphi||_{\infty}}\ve+\delta.\]
Now send $\delta$ to zero since it only had to be small enough depending on $\ve$. Finally, $\ve$ was arbitrarily chosen, which implies that the above limit is zero. Since this holds for all $\vphi$, we get that $\omega_{j}^{\pm}$ have the same weak limit.

\end{proof}

%\begin{lemma}\label{l:vmo}
%Let $\Omega^{+}$ and $\Omega^{-}=\ext(\Omega^{+})$ be two connected $\Delta$-regular domains. Suppose  that $\omega^{\pm}$ are mutually $A_{\infty}$-equivalent elliptic harmonic measures for which 
%
%and \eqref{e:vmo} holds at $\xi\in \d\Omega^{+}$. If $\Tan(\omega^{+},\xi)\neq\varnothing$, then $\Tan(\omega^{+},\xi)\subset \cH$.
%\end{lemma}

\begin{proof}[Proof of \Theorem{main}]
Let $\omega\in \Tan(\omega^{+},\xi)$. We claim that $\omega\in \mathscr{H}_{\cC}$. By \Lemma{nicetan}, $\omega=c\,T_{0,r}(\mu)$ for some constants $c,r>0$ and some measure $\mu$ of the form $\mu=\lim_{j\rightarrow 0} T_{\xi,r_{j}}[\omega^{+}]/\omega^{+}(B(\xi,r_{j}))$ for some $r_{j}\downarrow 0$ where $\mu(\bB)>0$. By \Lemma{vmo1}, $\mu=\lim_{j\rightarrow 0} T_{\xi,r_{j}}[\omega^{-}]/\omega^{-}(B(\xi,r_{j}))$ as well. By Lemma \ref{azmotovo2} (or \Lemma{azmoto}(g) if $\Omega^{\pm}$ have the CDC), $\mu\in \mathscr{H}_{\cC}$, and since $\mathscr{H}_{\cC}$ is a $d$-cone by \Lemma{Pdcone}, we also have that $\omega\in \mathscr{H}_{\cC}$, which proves the claim.

Hence, $\omega=\omega_{u}$ for some $u\in H_{A}$ and some $A\in \cC$. By \Lemma{taylor}, for some $k>0$,
 \[\Tan(\omega_{u},0)=\{c\omega_{u_{k}}:c>0\}\subset \cF_{A}(k)\subset \cF_{\cC}(k),\]
and since $\Tan(\omega_{u},0)\subset \Tan(\omega^{+},\xi)$ by \Theorem{ttt}, we now know that $\Tan(\omega^{+},\xi)\cap \cF_{\cC}(k)\neq\varnothing$ as well. By Lemma \ref{logw}, $\Tan(\omega^{+},\xi)\subset  \cF_{\cC}(k)$. The proof that $\Theta_{\d\Omega^{+}}^{\cF_{\Sigma,\cC}(k)}(\xi,r)\rightarrow 0$ if $\Omega^{\pm}$ have the CDC is similar to the proof of Theorem \ref{newKPT}. 

%
% If $\Tan(\omega^{+},\xi)\not\subset  \cF_{\cA}(k)$,  \Corollary{pstep} implies we can find $\nu\in \Tan(\omega^{+},\xi)$ so that, for any $\ve>0$, there is some $r_{0}>0$ so that
%\begin{equation}\label{e:drcontra}
%d_{r_{0}}(\nu,\cF_{\cA}(k))=\ve>0, \;\;\; d_{r}(\nu,\cF_{\cA}(k))\leq \ve \mbox{ for }r\geq r_{0},
%\end{equation}
%and $\nu(B(0,r))\leq Cr^{\beta}$ for all $r\geq r_{0}$ for some $r_{0}>0$ and some constants $C,\beta$ depending only on $k$ and $n$. Again, $\nu=\omega_{h}$ for some $h\in H$. For a multiindex $\alpha$, we have by the Cauchy estimates and \eqn{u<wr} that, for $\ell\in \bN$, 
%\begin{align}
%|\d_{\alpha} h(0)|(2^{\ell} )^{|\alpha|}
%& \lec \sup_{B(0,2^{\ell})} |h|
%\stackrel{\eqn{u<wr}}{\lec} \omega_{h}(B(0,2\delta_0^{-1}2^{\ell}))(2^{\ell} )^{1-n} \notag \\
%& \leq C (2\delta_0^{-1}2^{\ell})^{\beta}  2^{\ell(1-n)}
%\lesssim 2^{\ell(1-n+\beta)}\label{e:similar}
%\end{align}
%Thus, if $|\alpha|>1-n+\beta$, letting $\ell\rightarrow\infty$ implies $|\partial_{\alpha} h(0)|=0$. Thus, $h$ is a polynomial of degree at most $m=\ceil{1-n+\beta}$. The important point is that $m$ does not depend on $\ve$, and hence for $\ve$ small enough (depending on $n,m,$ and $k$, but since $\beta$ depends on $k$ and $n$, $\ve$ really only depends on $n$ and $k$), by \eqn{drcontra} and \Lemma{bad}, $\nu\in \cF_{\cA}(m)$. The only way this is possible 

\end{proof}

\begin{proof}[Proof of Theorem \ref{t:main2}]

Let $K$ be any compact subset of $\d\Omega^{+}$. Suppose there was a sequence of radii $r_{j}\downarrow 0$ and $\xi_{j}\in K$ so that 
\begin{equation}\label{e:vmo-far}
d_{1}(T_{\xi_{j},r_{j}}[\omega^{+}],\mathscr{P}_{\cC}(d))\geq \ve >0
\end{equation}
where $d$ will be chosen later, but it will depend only on $n$ and the doubling constant of $\omega^{+}$. 

Since $\omega^{+}$ is doubling, we may pass to a subsequence so that $\omega_{j}^{+}:=T_{\xi_{j},r_{j}}[\omega^{+}]/\omega^{+}(B(\xi_j,r_j))$ converges weakly to some measure $\omega$.  

If $f= \frac{d\omega^{-}}{d\omega^{+}}$ satisfies $\log f\in \VMO(\omega^{-})$, then doubling also implies that $\omega^{-}\in VA_{\infty}'(\omega^{+})$. Indeed, if $\omega^{+}$ is doubling, then the John-Nirenberg theorem holds, and the $\VMO$ condition tells us that on small enough balls, $f$ is a traditional $A_{p}$-weight (c.f. \cite[Chapter 6.2]{BAF}). This easily implies $fd\omega^{+}=d\omega^{-}\in VA_{\infty}'(\omega^{+})$. Thus, by Corollary \ref{c:afincafin}, we know $\omega^{-}\in VA_{\infty}(\omega^{+})$ that \eqref{e:xivmo} holds for every $M>0$. By Lemma \ref{l:vmo1}, $\omega_{j}^{-}\rightharpoonup \omega$ as well. Thus, we can pass to a subsequence so that the conclusions of Lemma \ref{azmotovo2} hold. In particular, $\omega=\omega_{h}$ for some $L_{0}$-harmonic function $h$, where $L_{0}$ is a uniformly elliptic operator with constant coefficients, and also, for any $\vphi\in C_{c}^{\infty}(\bR^{n+1})$, \eqref{e:ibp} holds. 

Now we apply the same standard trick from \cite{KT06}. Notice that  since $\omega^{+}$ is doubling, so is $\omega_{h}$, which combined with Cauchy estimates, implies that there exists $\beta>0$ such that for any $\ell\in \bN$ and any multi-index $\alpha$,
	\begin{align}
|\d_{\alpha} h(0)| &  \lec 2^{-|\alpha| \ell}  ||h||_{L^{\infty}(2^{\ell}\bB)}
 \stackrel{\eqref{e:ufinbound2}}{\lec} 2^{\ell(-|\alpha|+1-n)} \omega_{h}(B(0,2^{\ell+1})) \\
& \lec 2^{\ell(-|\alpha|+1-n+\beta)} \omega_{h}(B(0,2)).
\end{align}
Hence,  if $|\alpha|>1-n+\beta$, letting $\ell\rightarrow \infty$ gives $|\d_{\alpha}h(0)|=0$, which implies $h$ is a polynomial of degree at most $1-n+\beta$. Setting $d=\ceil{1-n+\beta}$ gives a contradiction to \eqref{e:vmo-far}.  The proof of \eqref{klimsup} is similar to the proof of Theorem \ref{newKPT}, where we use instead Lemma \ref{l:azmoto} instead of \ref{azmotovo2}.
	\end{proof}

\section{Proof of Theorem \ref{t:newABHM}}

All elliptic operators in this section will be assumed to satisfy \eqref{eqelliptic1} and \eqref{eqelliptic2}. We will require a few lemmas about elliptic measures in uniform domains as well as some new notation.

\begin{definition} \label{d:uniform}
	Let $\Omega\subseteq \bR^{n+1}$.
	\begin{itemize}
		\item We say $\Omega$ 
	satisfies the {\it corkscrew condition} if for some uniform constant $c>0$ and every ball $B$ centered on $\d\Omega$ with $0<r_{B}<\diam(\partial\Omega)$, there is a ball $B(x_{B},cr_{B})\subseteq \Omega\cap B$. The point $x_{B}$ is called
	a {\it corkscrew point relative to} $B$.
	\item We say $\Omega$ satisfies the {\it Harnack chain condition} if there is a uniform constant $C$ such that
	for every $\rho >0,\, \Lambda\geq 1$, and every pair of points
	$x,y \in \Omega$ with $\delta(x),\,\delta(y) \geq\rho$ and $|x-y|<\Lambda\,\rho$, there is a chain of
	open balls
	$B_1,\dots,B_N \subset \Omega$, $N\leq C(\Lambda)$,
	with $x\in B_1,\, y \in B_N,$ $B_k\cap B_{k+1}\neq \emptyset$
	and $C^{-1}\diam (B_k) \leq \dist (B_k,\partial\Omega)\leq C\diam (B_k).$  The chain of balls is called
	a {\it Harnack chain}.
%	\item We say a domain $\Omega$ is {\it $b$-John} if there is $x_{0}\in \Omega$ so that for any $\in \Omega$, there is a path $\gamma$ connecting $x$ to $x_{0}$ so that 
%	\[
%	\delta(z)\geq b |z-x| \mbox{ for all }z\in \gamma.
%	\]
	\end{itemize}
\end{definition}
%
%\begin{remark}
%	It is well known that uniform domains are John domains quantitatively.
%\end{remark}
\begin{definition}%[\bf 1-sided NTA]\label{def1.1nta}
	If $\Omega$ satisfies both the corkscrew and the Harnack chain conditions, then we say that
	$\Omega$ is a {\it uniform domain}.%{\it 1-sided NTA domain}, also known in the literature as.
\end{definition}\label{def1.1nta} 
\begin{theorem}
	Let $\Omega\subset \bR^{n+1}$ be a uniform domain with the CDC and $u$ a nonnegative $L_{A}$-elliptic function vanishing on $2B\cap \d\Omega$ where $B$ is a ball with $r_{B}<\diam \d\Omega$ and $A\in \cA(\Omega)$. Then 
	\begin{equation}
	\label{bharnack}
	\sup_{x\in B\cap \Omega} u(x)\lec u(x_{B}).
	\end{equation}
\end{theorem}

This was originally shown in section 4 of \cite{JK82} for NTA domains, but the proof only uses the H\"older continuity of $u$ at the boundary and the fact that NTA domains are uniform, and so the proof of the above result is exactly the same. 

\begin{theorem}
	Let $\Omega\subset \bR^{n+1}$ be a uniform domain with the CDC and $L_{A}$ an elliptic operator satisfying  \eqref{eqelliptic1} and \eqref{eqelliptic2}. Then, for all $B$ centered on $\d\Omega$,
	\begin{equation}\label{wGuniform}
	\omega^{L_{A},x}(B)\approx r_{B}^{n-1} G_{\Omega}(x,x_{B}) \mbox{ for all }x\in \Omega\backslash 2B.
	\end{equation}
\end{theorem}

This follows from the work of Aikawa and Hirata \cite{AH08}. Their proof is originally for harmonic measures, but an inspection of the proof shows that it carries through for elliptic measure as  well. 

\begin{theorem}
Let $\Omega\subset \bR^{n+1}$ be a uniform domain with the CDC. If  $L_{A}$ is an elliptic operator satisfying  \eqref{eqelliptic1} and \eqref{eqelliptic2}, $B$ is a ball centered on $\d\Omega$ and $E\subset B\cap \d\Omega$ is Borel, then 
\begin{equation}\label{markov}
\omega_{\Omega}^{L_{A},x_{B}}(E)\approx \frac{\omega_{\Omega}^{L_{A},x}(E)}{\omega_{\Omega}^{L_{A},x}(B)}.
\end{equation}
\end{theorem}

{ Again, this is \cite[Lemma 4.11]{JK82}, and since the previous two lemmas are available, the proof is exactly the same for elliptic  measures modulo the proof of \cite[Lemma 4.10]{JK82}. The latter can also be proved  as in \cite{JK82} to build a sub-uniform domain, and then showing as in \cite[Lemma 2.26]{AAM16} hat the resulting domain is also CDC (all of this instead of a geometric localization theorem due to Jones, which only works for NTA domains).}
\begin{lemma}\label{harmblow}
	Let $\Omega\subset \bR^{n+1}$ be a uniform domain with the CDC and $L_{A}$ an elliptic operator satisfying  \eqref{eqelliptic1} and \eqref{eqelliptic2}, and also  \eqref{e:boundarylimit} at $\xi$. If $\xi\in \d\Omega$ and $\omega_{j} =\omega^{L_{A},x_{0}}(B(\xi,r_{j}))^{-1}T_{\xi,r_{j}}(\omega^{L_{A},x_{0}})$ converges weakly to a tangent measure $\omega_{\infty}\in \Tan(\omega^{L_{A},x_{0}},\xi)$. Then there is a uniform domain $\Omega_{\infty}$ and a constant matrix $A_{0}\in \cC$ such that, for each $x\in \Omega_{\infty}$, $\omega_{\Omega_{j}}^{x}\warrow\omega_{\Omega_{\infty}}^{x}$ and, for all balls $B'\subset B$ centered on $\d\Omega_{\infty}$, if $x_{B}$ is a corkscrew point in $\Omega_{\infty}\cap B$,
	\begin{equation}\label{limratio}
	\omega_{\Omega_{\infty}}^{L_{A_{0}},x_{B}}(B') \approx \frac{\omega_{\infty}(B')}{\omega_{\infty}(B)}.
	\end{equation}
\end{lemma}

This was originally shown in \cite{AM15} for harmonic measure.  In our situation, the proof is much shorter, so we provide it here.

\begin{proof}
	By Lemma \ref{l:azmoto}, there is $A_{0}\in \cC$ so that we can pass to a subsequence so that $u_{j}(x) = c_{j} u(xr_{j}+\xi)r_{j}^{n-1}$ converges uniformly in $\bR^{n+1}$ to a nonzero $L_{A_{0}}$-elliptic function $u_{\infty}$ and also so that, if $\Omega_{j}= T_{\xi,r_{j}}(\Omega)$, then $\d\Omega_{j}$ converges in the Hausdorff metric on compact subsets. Let $\Omega_{\infty}=\{u_{\infty}>0\}$. 
	
	\Claim $\Omega_{\infty}$ is uniform. If $x,y\in \Omega_{\infty}$ with $\dist(\{x,y\},\d\Omega) \geq \ve |x-y|$, then they are contained in $\Omega_{j}$ and $\dist(\{x,y\},\d\Omega_{j}) \geq \frac{\ve}{2} |x-y|$ for sufficiently large $j$. Since the $\Omega_{j}$ are uniform, for each $j$ we can find a Harnack chain of length $N=N(\ve)$ contained in $\Omega_{j}$. By passing to a subsequence, we can assume the length of this chain is constant and their centers and radii are converging, and hence the chain converges to a Harnack chain in $\Omega_{\infty}$ of length no more than $N$. A similar proof shows that $\Omega_{\infty}$ is a corkscrew domain. Hence, $\Omega_{\infty}$ is uniform.
	
	Suppose  $B'\subset \bB$ are centered on $\d\Omega_{\infty}$. Let 
	\[
	\omega^{T_{\xi,r_{j}}(x)}_{\Omega_{j}}= T_{\xi,r_{j}}[\omega^{L_{A},x}].\]
	If $x_{j}= T_{\xi,r_{j}}(x_{0})$, then	
	\begin{align*}
		 \omega^{x_{B}}_{\Omega_{j}}(B')
		 \approx \frac{\omega^{x_{j}}_{\Omega_{j}}(B')}{\omega^{x_{j}}_{\Omega_{j}}(B)}
		= \frac{\omega^{x_{j}}_{\Omega_{j}}(\bB)}{\omega^{x_{j}}_{\Omega_{j}}(B)} \frac{\omega^{x_{j}}_{\Omega_{j}}(B')}{\omega^{x_{j}}_{\Omega_{j}}(\bB)}
		=\frac{\omega_{j}(B')}{\omega_{j}(B)}.
	\end{align*}
	Since $\omega_{j}$ and $\omega_{\Omega_{j}}$ are doubling measures, we have 
	\[
	\omega_{\Omega_{\infty}}^{x_{B}}(B')
	\leq \liminf_{j\rightarrow\infty} \omega^{x_{B}}_{\Omega_{j}}(B')
	\lec  \limsup_{j\rightarrow\infty}\frac{\omega_{j}(B')}{\omega_{j}(B)}
	\leq  \frac{\omega_{\infty}(\cnj{B'})}{\omega_{\infty}(B)}
	\lec  \frac{\omega_{\infty}({B'})}{\omega_{\infty}(B)}.
	\]
	A similar estimate gives the reverse inequality, and hence proves \eqref{limratio}.
\end{proof}

We will use the following criterion for uniform rectifiability due to Hofmann, Martell, and Uriarte-Tuero. 

\begin{theorem}\label{HMUT}
	Let $\Omega\subset \bR^{n+1}$ be a uniform domain with $n$-regular boundary and let $\omega_{\om}$ be the harmonic measure defined in $\om$. Suppose there is $q>1$ so that, for any balls $B'\subset B$ centered on $\d\Omega$, if $k_{B}=\frac{d\omega_{\Omega}^{x_{B}}}{d\cH^{n}|_{\d\Omega}}$, then 
		\[
		\ps{\avint_{B'\cap \d\Omega}k_{B}^{q}\,d\cH^{n}}^{\frac{1}{q}}\lec 
		\avint_{B'\cap \d\Omega} k_{B}\, d\cH^{n}.\]
		Then $\d\Omega_{\infty}$ is uniformly rectifiable.
\end{theorem}

%\begin{lemma}
%		Let $\Omega\subset \bR^{n+1}$ be a uniform domain with the CDC and so that $\d\Omega$ has locally finite $\cH^{n}$-measure. Let $E\subset \d\Omega$ be a set where $\omega^{L_A,x_{0}}\ll \cH^{n}$, $L_A \in \VMO(\om, \xi)$ satisfies \eqref{eqelliptic1} and \eqref{eqelliptic2}, and $\theta_{*}^{n}(\cH^{n}|_{\d\Omega},\xi)>0$ at $\cH^{n}$-a.e. $\xi\in E$. Then $E$ may be covered up to $\omega$-measure zero by Lipschitz graphs. 
%\end{lemma}

Now we prove Theorem \ref{t:newABHM}. Assume $\omega(E)>0$. Then we may find a subset $E'$ of full measure where $\omega$ and $\cH^{n}$ are mutually absolutely continuous. For almost every $\xi\in E'$, we also have 
\begin{equation}\label{e:lowden}
0<\theta^{n}_{*}(\cH^{n}|_{\d\Omega},\xi)\leq \theta^{n,*}(\cH^{n}|_{\d\Omega},\xi)<\infty.
\end{equation} 
By \cite[Theorem 14.7]{Mattila}, for almost every $\xi\in E'$, $\Tan(\cH^{n}|_{\d\Omega},\xi)$ consists of Ahlfors-David $n$-regular measures. By  \cite[Lemma 14.5]{Mattila} and \cite[Lemma 14.6]{Mattila}, for a.e. $\xi\in E'$,
\[
\Tan(\cH^{n}|_{\d\Omega},\xi) = \Tan(\cH^{n}|_{E'},\xi)=\Tan(\omega,\xi)
\]
and $\Tan(\omega,\xi)$ consists only of Ahlfors-David $n$-regular measures. Let $E''\subset E'$ be the set of points where this holds. 

By the Besicovitch decomposition theorem, we can split $E''$ into two sets $F_{1}$ and $F_{2}$ where $F_{1}$ is $n$-rectifiable and $F_{2}$ is purely $n$-unrectifiable. Suppose $\cH^{n}(F_{2})>0$.  Let $\xi\in F_{2}$ be a point of density of $F_{2}$ with respect to $\cH^{n}$. 

Let $r_{j}\downarrow 0$ be so that $\omega_{j} := \omega^{L_{A},x_{0}}(B(\xi,r_{j}))^{-1}T_{\xi,r_{j}}(\omega^{L_{A},x_{0}})$ converges weakly to some Ahlfors-David $n$-regular measure $\omega_{\infty}\in \Tan(\omega,\xi)$.  By Lemma \ref{harmblow}, we may find a uniform domain $\Omega_{\infty}$ so that $\supp \omega_{\infty}= \d\Omega_{\infty}$ and, for any balls $B'\subset B$ centered on $\d\Omega$, 
\[
\omega_{\Omega_{\infty}}^{L_{A_{0}},x_{B}}(B')\approx \frac{\omega_{\infty}(B')}{\omega_{\infty}(B)}
\approx \frac{r_{B'}^{n}}{r_{B}^{n}},
\]
for some $A_{0}\in \cC$. If $\sigma= \cH^{n}|_{\d\Omega_{\infty}}$, then $\sigma$ is Ahlfors-David $n$-regular and so if we set 
\[
k_{B}:=\frac{d \omega_{\Omega_{\infty}}^{L_{A_{0}},x_{B}}}{d\sigma},\]
then we have that for a.e. $x\in B\cap \d\Omega$,
\[
k_{B}(x) = \lim_{r\rightarrow 0} \frac{\omega_{\Omega_{\infty}}^{L_{A_{0}},x_{B}}(B(x,r))}{\sigma(B(x,r))}
\approx \frac{r^{n}/r_{B}^{n}}{r^{n}} = r_{B}^{-n}.
\]
Hence, if $B'\subset B$ is centered on $\d\Omega$,
\[
\ps{\avint_{B'}k_{B}^{2}\, d\sigma }^{\frac{1}{2}}
\approx r_{B}^{-n}\approx \avint_{B'}k_{B}\, d\sigma.
\]

Thus, by Theorem \ref{HMUT} (which still holds for constant coefficients since UR is invariant under bi-lipschitz maps), $\d\Omega_{\infty}$ is uniformly rectifiable. By the main result of \cite{AHMNT}, $\Omega_{\infty}$ is an NTA domain. In particular, we can find corkscrew balls $B_{1}\subset \bB\cap \Omega_{\infty}$ and $B_{2}\subseteq \bB\backslash \Omega_{\infty}$.  Also, for all $j$ sufficiently large, $B_{1}\subset \Omega_{j}\cap \bB$ and $B_{2}\subset \bB\backslash \Omega_{j}$. By the Besicovitch-Federer projection theorem, $\d\Omega_{j}\cap \bB$ contains an $n$-rectifiable set of $\cH^{n}$-measure $\gec 1$ (with constant depending on the corkscrew constant for $\Omega_{\infty}$ and hence on the corkscrew constant for $\Omega$). Thus, 
\[
\liminf_{j\rightarrow\infty} \frac{\cH^{n}( B(\xi,r_{j})\cap \d\Omega \backslash F_{2})}{\cH^{n}(B(\xi,r_{j}) \cap \d\Omega)} 
\gec \liminf_{j\rightarrow\infty} \frac{r_{j}^{n}}{ \cH^{n}(B(\xi,r_{j}) \cap \d\Omega)} \stackrel{\eqref{e:lowden}}{>}0 .\]
But this contradicts $\xi$ being a point of density for $F_{2}$. Therefore, $\cH^{n}(F_{2})=0$, and we have now shown that $\cH^{n}$-almost all of $E'$ is rectifiable, and thus $\omega^{x_{0}}$-almost all of $E$ is contained in a countable union of Lipschitz graphs. This finishes the proof of Theorem \ref{t:newABHM}.

\section{Proof of Proposition \ref{carlesonprop}}
Assume the conditions of the proposition. We recall the following result.

\begin{theorem}
	\cite[Theorem 1.3]{H-S94} Suppose that $\Omega\subset \R^{n+1}$ is a bounded $C$-uniform domain. If
	\[
	p\leq q \leq \frac{(n+1)p}{n+1-p(1-\delta)} \mbox{ and }p(1-\delta)<n+1,\]
	then for all $u\in L_{\loc}^{1}(\Omega)$ such that $\grad u(x)d(x,\d\Omega)^{\delta}\in L^{p}(\Omega)$,
	\begin{equation}\label{poincare1}
	\inf_{a\in \R}||u(x)-a||_{L^{q}(\Omega)} \lec_{n,p,q,\delta,C}  |\Omega|^{\frac{1-\delta}{n+1}+\frac{1}{q}-\frac{1}{p}} ||\grad u \dist(\cdot,\Omega^{c})^{\delta}||_{L^{p}(\Omega)}.
	\end{equation}
%		\begin{equation}\label{poincare1}
%	\inf_{a\in \R}||u(x)-a||_{L^{q}(\Omega)} \lec_{n,p,q,\delta} b^{n+1} |\Omega|^{\frac{1-\delta}{n+1}+\frac{1}{q}-\frac{1}{p}} ||\grad u \dist(\cdot,\Omega^{c})^{\delta}||_{L^{p}(\Omega)}.
%	\end{equation}
\end{theorem} 
(The explicit constant in \eqref{poincare1} is written at the end of the proof on page 218 of \cite{H-S94}.) We will use this in the case that $\delta=\frac{1}{2}$ and $p=q=2$, so \eqref{poincare1} becomes 
	\begin{equation}\label{poincare2}
\inf_{a\in \R}||u(x)-a||_{L^{2}(\Omega)} \lec_{n,p,q,\delta,C} |\Omega|^{\frac{1}{2(n+1)}} ||\grad u \dist(\cdot,\Omega^{c})^{\frac{1}{2}}||_{L^{2}(\Omega)}.
\end{equation}

\begin{lemma}
	Suppose $E\subset \bR^{n+1}$ and $\ve:E^{c}\rightarrow [0,\infty]$ is a function such that for some ball $B_{0}$ centered on $\d\Omega$, there exists a constant $C>0$
	\[
	\int_{B_{0} \cap \Omega}\ve(z)dz\leq Cr^{n}.\]
	Then for almost every $x\in E\cap B_{0}$,
	\[
	\lim_{r\rightarrow 0} r^{-n} \int_{B(x,r)}\ve(z)dz=0
	\]
\end{lemma}

\begin{proof}
	Without loss of generality, we can assume $E\subseteq B_{0}$. Let $d\mu(z) =\ve(z)\, dz|_{E^c}$. For $x\in E$ and $r>0$, set
	\[
	a(x,r) =\frac{\mu(B(x,r))}{ r^{n}} = r^{-n} \int_{B(x,r)}\ve(z)dz .\]
	Suppose there is $F\subset E$ with $\cH^{n}(F)>0$ such that 
	\[
	\limsup_{r\rightarrow 0}a(x,r)>0.
	\]
	Without loss of generality, we can assume $F$ is bounded. Then there is $t>0$ and a compact set $G\subset F$ with $\cH^{n}(G)>0$ and 
	\[
		\limsup_{r\rightarrow 0} a(x,r)> t>0 \;\; \mbox{ for all }x\in G. 
	\]
	For each $x$, pick $r_{x,1}>0$ so that  $B(x,r_{x,1})\subseteq B_{0}$ and $ a(x,r_{x,1})>t$. Let $B_{j}^{1}$ be a Besicovitch subcovering from $\cG_{1}:=\{B(x,r_{x}^{1}):x\in G\}$, that is, a countable collection of balls in $\cG_{1}$ so that 
	\[
	\one_{G} \leq \sum_{j} \one_{B_{j}^{1}}\lec_{n} 1. 
	\]
	Since the $B_{j}^{1}$ come from $\cG$, we have that for all $j$,
	\[
	\frac{\mu(B_{j}^{1})}{r_{B_{j}^{1}}^{n}}
	=a(x_{B_{j}^{1}},r_{B_{j}^{1}}) >t.
	\]
	Let 
	\[
	L_1 = \bigcup B_{j}^{1} \backslash E.\]
	Then since the $B_{j}^{1}$ have bounded overlap and come from $\cG_{1}$,
	\begin{align*}
	\mu(L_{1})
	& = \int_{L_{1}} d\mu 
	 \gec \int_{L_{1}} \sum_{j} \one_{B_{j}^{1}}d\mu 
	= \sum_j \mu(B_{j}^{1})>t\, \sum_j r_{B^1_{j}}^{n}
	\geq t\,\cH_{\infty}^{n}(G).
	\end{align*}
	Since $\mu(G)=0$, there is $\delta_{1}>0$ so that if $G_{\delta_{1}} = \{x\in \R^{n}:\dist(x,G)<\delta\})$
	and $L^{1} = L_{1}\backslash G_{\delta_{1}}$, then
	\[
	\mu(L^{1})> \frac{\mu(L_{1})}{2} \geq \frac{t}{2} \cH_{\infty}^{n}(G).\]
	Now inductively, suppose we have constructed disjoint sets $L^{1},...,L^{k}\subseteq B_{0}$ where
	\[
	\mu(L^{j})\gec t \cH_{\infty}^{n}(G) \;\; \mbox{ for all }j=1,2,...,k.
	\]
	and there is $\delta_{k}>0$ so that $L^{1}\cup\cdots \cup L^{k}\cap G_{\delta_{k}}=\emptyset$. 
	
	For each $x\in G$, we may find $r_{x,k+1}\in (0,\delta_{k})$ so that $B(x,r_{x,k+1})\subseteq B_{0}$ and  $a(x,r_{x,k+1})>t$. Let $\{B_{j}^{k+1}\}$ be a Besicovitch subcovering of the collection $\cG_{k+1} = \{B(x,r_{x,k+1}): x\in G\}$, so 
		\[
	\one_{G} \leq \sum_{j} \one_{B_{j}^{k+1}}\lec_{n} \one_{L_{k+1}},
	\]
	where $L_{k+1}=\bigcup_{j} B_{j}^{k+1}$. Since $G$ has $\mu(G)=0$, there is $\delta_{k+1}\in (0,\delta_{k})$ so that $L^{k+1} = L_{k+1}\backslash G_{\delta_{k+1}}$ has,
	\begin{align*}
	\mu(L^{k+1}) 
	& \geq \frac{\mu(L_{k+1})}{2} 
	=\frac{1}{2} \int \one_{L_{k+1}} d\mu \gec \int  \sum_j \mu(B^{k+1}_{j})\geq t\,\sum_j r_{B_{j}^{k+1}}^{n} \\
	& \gec t\, \cH^{n}(G).
	\end{align*}
	Also note that by our induction hypothesis
	\[
	L^{k+1}\subseteq L_{k+1}\subseteq G_{\delta_{k}}\subseteq (L^{1}\cup\cdots \cup L^{k})^{c}.\] 
	Thus, by induction, we can come up with a sequence of disjoint sets $L^{k}\subseteq B_{0}$ so that $\mu(L^{k})\gec t\, \cH^{n}(G)$ for all $k$, which contradicts the finiteness of $\mu$ since $\ve$ is locally integrable. 
\end{proof}

Now we finish the proof of Proposition \ref{carlesonprop}. By the previous lemma, we have that for $\cH^{n}$-a.e. $\xi \in B_{0}\cap \d\Omega$ that 

\begin{equation}
\lim_{r\rightarrow 0} r^{-n}\int_{B(\xi,r)}|\grad A|^{2} \dist(z,\Omega^{c})\,dz=0.
\end{equation}
Let $\xi \in B_{0}\cap \d\Omega$  be such a point. There is a universal constant $M$ depending on the uniformity constants so that, for all $r>0$, there is a $MC$-uniform domain $\Omega_{r}$ such that 
\[
\Omega\cap B(\xi,r)\subset \Omega_{r}\subset \Omega\cap B(\xi,Mr).
\]
This follows from the proof of \cite[Lemma 3.61]{HM14}. See also \cite[Lemma 4.1]{Azz16} or \cite[Lemma 6.3]{JK82}. 

%In particular, since $\Omega_{r}$ is $MC$-uniform, there is a corkscrew point $x_{r}\in \Omega_{r}$ for which $\Omega_{r}$ is also $CM$-John. 
Hence, by Cauchy-Schwarz inequality,
\begin{align*}
\inf_{C}r^{-(n+1)} &  \int_{B(\xi,r)\cap \Omega }|A-C| \\
&\lesssim  \inf_{C} \left(r^{-(n+1)} \int_{B(\xi,r)\cap \Omega }|A-C|^{2} \right)^{1/2}\\
& \leq \inf_{C} \left(r^{-(n+1)} \int_{\Omega_{r}}|A-C|^{2} \right)^{1/2} \\
&  \stackrel{\eqref{poincare2}}{\lec} |\Omega_{r}|^{\frac{1}{2(n+1)}} \left(\frac{1}{r^{n+1}} \int_{\Omega_{r}}|\grad A|^{2} \dist(z,\Omega_{r}^{c})dz\right)^{1/2} \\
& \lec \left( r^{-n} \int_{\Omega\cap B(\xi,Mr)}|\grad A|^{2} \dist(z,\Omega^{c})dz \right)^{1/2}\rightarrow 0,\,\,\textup{as} \,\,r \to 0.
\end{align*}

\frenchspacing
\bibliographystyle{alpha}
%\bibliography{reference}

\begin{thebibliography}{GaCRdF85}

\bibitem[AAM16]{AAM16}
M.~Akman, J.~Azzam, and M.~Mourgoglou.
\newblock Absolute continuity of harmonic measure for domains with lower
  regular boundaries.
\newblock {\em arXiv preprint arXiv:1605.07291}, 2016.

\bibitem[ABHM17]{ABHM17}
M.~Akman, M.~Badger, S.~Hofmann, and J-M. Martell.
\newblock Rectifiability and elliptic measures on 1-sided {NTA} domains with
  {A}hlfors-{D}avid regular boundaries.
\newblock {\em Trans. Amer. Math. Soc.}, 369(8):5711--5745, 2017.

\bibitem[ACF84]{ACF84}
H.W. Alt, L.A. Caffarelli, and A.~Friedman.
\newblock Variational problems with two phases and their free boundaries.
\newblock {\em Trans. Amer. Math. Soc.}, 282(2):431--461, 1984.

\bibitem[AGMT17]{AGMT17}
J.~Azzam, J.~Garnett, M.~Mourgoglou, and X.~Tolsa.
\newblock Uniform rectifiability, elliptic measure, square functions, and
  $\varepsilon$-approximability.
\newblock {\em arXiv preprint arXiv:1612.02650}, 2017.

\bibitem[AH08]{AH08}
H.~Aikawa and K.~Hirata.
\newblock Doubling conditions for harmonic measure in {J}ohn domains.
\newblock {\em Ann. Inst. Fourier (Grenoble)}, 58(2):429--445, 2008.

\bibitem[AHM{\etalchar{+}}16]{AHMMMTV16}
J.~Azzam, S.~Hofmann, J.~M. Martell, S.~Mayboroda, M.~Mourgoglou, X.~Tolsa, and
  A.~Volberg.
\newblock Rectifiability of harmonic measure.
\newblock {\em Geom. Funct. Anal.}, 26(3):703--728, 2016.

\bibitem[AHM{\etalchar{+}}17]{AHMNT}
J.~Azzam, S.~Hofmann, J.~M. Martell, K.~Nystr{\"o}m, and T.~Toro.
\newblock A new characterization of chord-arc domains.
\newblock {\em J. Eur. Math. Soc.}, 19:967--981, 2017.

\bibitem[AM15]{AM15}
J.~Azzam and M.~Mourgoglou.
\newblock Tangent measures and densities of harmonic measure.
\newblock {\em to appear in Rev. Math.}, 2015.

\bibitem[AMT16]{AMT16}
J.~Azzam, M.~Mourgoglou, and X.~Tolsa.
\newblock Mutual absolute continuity of interior and exterior harmonic measure
  implies rectifiability.
\newblock {\em To appear in Comm. Pure Appl.}, 2016.

\bibitem[AMT17]{AMT17}
J.~Azzam, M.~Mourgoglou, and X.~Tolsa.
\newblock Singular sets for harmonic measure on locally flat domains with
  locally finite surface measure.
\newblock {\em International Mathematics Research Notices},
  2017(12):3751--3773, 2017.

\bibitem[AMTV16]{AMTV16}
J.~Azzam, M.~Mourgoglou, X.~Tolsa, and A.~Volberg.
\newblock Mutual absolute continuity of interior and exterior harmonic measure
  revisited.
\newblock {\em arXiv preprint arXiv:1609.06133}, 2016.

\bibitem[Azz16]{Azz16}
J.~Azzam.
\newblock Sets of absolute continuity for harmonic measure in {NTA} domains.
\newblock {\em Potential Anal.}, 45(3):403--433, 2016.

\bibitem[Bad11]{Bad11}
M.~Badger.
\newblock Harmonic polynomials and tangent measures of harmonic measure.
\newblock {\em Rev. Mat. Iberoam.}, 27(3):841--870, 2011.

\bibitem[Bad12]{Bad12}
M.~Badger.
\newblock Null sets of harmonic measure on {NTA} domains: {L}ipschitz
  approximation revisited.
\newblock {\em Math. Z.}, 270(1-2):241--262, 2012.

\bibitem[Bad13]{Bad13}
M.~Badger.
\newblock Flat points in zero sets of harmonic polynomials and harmonic measure
  from two sides.
\newblock {\em Journal of the London Mathematical Society}, 87(1):111--137,
  2013.

\bibitem[BCGJ89]{BCGJ88}
C.~J. Bishop, L.~Carleson, J.~B. Garnett, and P.~W. Jones.
\newblock Harmonic measures supported on curves.
\newblock {\em Pacific J. Math.}, 138(2):233--236, 1989.

\bibitem[BET17]{BET17}
M.~Badger, M.~Engelstein, and T.~Toro.
\newblock Structure of sets which are well approximated by zero sets of
  harmonic polynomials.
\newblock {\em Anal. PDE}, 10(6):1455--1495, 2017.

\bibitem[BJ90]{BJ90}
C.~J. Bishop and P.~W. Jones.
\newblock Harmonic measure and arclength.
\newblock {\em Ann. of Math. (2)}, 132(3):511--547, 1990.

\bibitem[Bou87]{Bou87}
J.~Bourgain.
\newblock On the {H}ausdorff dimension of harmonic measure in higher dimension.
\newblock {\em Invent. Math.}, 87(3):477--483, 1987.

\bibitem[Bre11]{Br}
H.~Brezis.
\newblock {\em Functional analysis, {S}obolev spaces and partial differential
  equations}.
\newblock Universitext. Springer, New York, 2011.

\bibitem[CFK81]{CFK81}
Luis~A. Caffarelli, Eugene~B. Fabes, and C.~E. Kenig.
\newblock Completely singular elliptic-harmonic measures.
\newblock {\em Indiana Univ. Math. J.}, 30(6):917--924, 1981.

\bibitem[CNV15]{CNV15}
J.~Cheeger, A.~Naber, and D.~Valtorta.
\newblock Critical sets of elliptic equations.
\newblock {\em Comm. Pure Appl. Math.}, 68(2):173--209, 2015.

\bibitem[DJ90]{DJ90}
G.~David and D.~S. Jerison.
\newblock Lipschitz approximation to hypersurfaces, harmonic measure, and
  singular integrals.
\newblock {\em Indiana Univ. Math. J.}, 39(3):831--845, 1990.

\bibitem[FJK84]{FJK84}
E.B. Fabes, D.~S. Jerison, and C.~E. Kenig.
\newblock Necessary and sufficient conditions for absolute continuity of
  elliptic-harmonic measure.
\newblock {\em Ann. of Math. (2)}, 119(1):121--141, 1984.

\bibitem[FKP91]{FKP91}
R.~A. Fefferman, C.~E. Kenig, and J.~Pipher.
\newblock The theory of weights and the {D}irichlet problem for elliptic
  equations.
\newblock {\em Ann. of Math. (2)}, 134(1):65--124, 1991.

\bibitem[GaCRdF85]{GCRF}
J.~Garc\'\i~a Cuerva and J.L. Rubio~de Francia.
\newblock {\em Weighted norm inequalities and related topics}, volume 116 of
  {\em North-Holland Mathematics Studies}.
\newblock North-Holland Publishing Co., Amsterdam, 1985.
\newblock Notas de Matem\'atica [Mathematical Notes], 104.

\bibitem[Gar07]{BAF}
J.~B. Garnett.
\newblock {\em Bounded analytic functions}, volume 236 of {\em Graduate Texts
  in Mathematics}.
\newblock Springer, New York, first edition, 2007.

\bibitem[GM08]{Harmonic-Measure}
J.~B. Garnett and D.~E. Marshall.
\newblock {\em Harmonic measure}, volume~2 of {\em New Mathematical
  Monographs}.
\newblock Cambridge University Press, Cambridge, 2008.
\newblock Reprint of the 2005 original.

\bibitem[GW82]{GW82}
M.~Gr{\"u}ter and K-O. Widman.
\newblock The {G}reen function for uniformly elliptic equations.
\newblock {\em Manuscripta Math.}, 37(3):303--342, 1982.

\bibitem[HK07]{HK07}
S.~Hofmann and S.~Kim.
\newblock The green function estimates for strongly elliptic systems of second
  order.
\newblock {\em manuscripta mathematica}, 124(2):139--172, 2007.

\bibitem[HKM06]{HKM}
J.~Heinonen, T.~Kilpel{\"a}inen, and O.~Martio.
\newblock {\em Nonlinear potential theory of degenerate elliptic equations}.
\newblock Dover Publications, Inc., Mineola, NY, 2006.
\newblock Unabridged republication of the 1993 original.

\bibitem[HM14]{HM14}
S.~Hofmann and J.~M. Martell.
\newblock Uniform rectifiability and harmonic measure {I}: {U}niform
  rectifiability implies {P}oisson kernels in {$L^p$}.
\newblock {\em Ann. Sci. \'Ec. Norm. Sup\'er. (4)}, 47(3):577--654, 2014.

\bibitem[HMT16]{HMT16}
S.~Hofmann, J.~M. Martell, and T.~Toro.
\newblock ${A}_\infty$ implies {NTA} for a class of variable coefficient
  elliptic operators.
\newblock {\em arXiv preprint arXiv:1611.09561}, 2016.

\bibitem[Hru84]{Hru84}
S.~V. Hru{\v{s}}{\v{c}}ev.
\newblock A description of weights satisfying the {$A_{\infty }$} condition of
  {M}uckenhoupt.
\newblock {\em Proc. Amer. Math. Soc.}, 90(2):253--257, 1984.

\bibitem[HS94]{H-S94}
R.~Hurri-Syrjänen.
\newblock An improved poincar{\'e} inequality.
\newblock {\em Proceedings of the American Mathematical Society},
  120(1):213--222, 1994.

\bibitem[JK82]{JK82}
D.~S. Jerison and C.~E. Kenig.
\newblock Boundary behavior of harmonic functions in nontangentially accessible
  domains.
\newblock {\em Adv. in Math.}, 46(1):80--147, 1982.

\bibitem[Kor98]{Kor98}
M.~B. Korey.
\newblock Ideal weights: asymptotically optimal versions of doubling, absolute
  continuity, and bounded mean oscillation.
\newblock {\em J. Fourier Anal. Appl.}, 4(4-5):491--519, 1998.

\bibitem[KP01]{KP01}
C.~E. Kenig and J.~Pipher.
\newblock The {D}irichlet problem for elliptic equations with drift terms.
\newblock {\em Publ. Mat.}, 45(1):199--217, 2001.

\bibitem[KP02]{KP}
S.~G. Krantz and H.~R. Parks.
\newblock {\em A primer of real analytic functions}.
\newblock Springer Science \& Business Media, 2002.

\bibitem[KPT09]{KPT09}
C.~E. Kenig, D.~Preiss, and T.~Toro.
\newblock Boundary structure and size in terms of interior and exterior
  harmonic measures in higher dimensions.
\newblock {\em J. Amer. Math. Soc.}, 22(3):771--796, 2009.

\bibitem[KT06]{KT06}
C.~E. Kenig and T.~Toro.
\newblock Free boundary regularity below the continuous threshold: 2-phase
  problems.
\newblock {\em J. Reine Angew. Math.}, 596:1--44, 2006.

\bibitem[Mar79]{Mar79}
O.~Martio.
\newblock Capacity and measure densities.
\newblock {\em Ann. Acad. Sci. Fenn. Ser. A I Math.}, 4(1):109--118, 1979.

\bibitem[Mat95]{Mattila}
P.~Mattila.
\newblock {\em Geometry of sets and measures in {E}uclidean spaces}, volume~44
  of {\em Cambridge Studies in Advanced Mathematics}.
\newblock Cambridge University Press, Cambridge, 1995.
\newblock Fractals and rectifiability.

\bibitem[Pom86]{Pom86}
Ch. Pommerenke.
\newblock On conformal mapping and linear measure.
\newblock {\em J. Analyse Math.}, 46:231--238, 1986.

\bibitem[Pre87]{Pr87}
D.~Preiss.
\newblock Geometry of measures in {${\bf R}^n$}: distribution, rectifiability,
  and densities.
\newblock {\em Ann. of Math. (2)}, 125(3):537--643, 1987.

\bibitem[RR16]{RR16}
F.~Riesz and M.~Riesz.
\newblock \"{U}ber randwerte einer analytischen funcktionen.
\newblock pages 27--44. European Mathematical Society, 1916.

\bibitem[RR75]{RR75}
H.~M. Reimann and T.~Rychener.
\newblock {\em Funktionen beschr{\"a}nkter mittlerer Oszillation}.
\newblock Springer, 1975.

\bibitem[Swe92]{Swe92}
C.~Sweezy.
\newblock The {H}ausdorff dimension of elliptic measure---a counterexample to
  the {O}ksendahl conjecture in {${\bf R}^2$}.
\newblock {\em Proc. Amer. Math. Soc.}, 116(2):361--368, 1992.

\bibitem[TV16]{TV16-prep}
X.~Tolsa and Alexander Volberg.
\newblock On tsirelson's theorem about triple points for harmonic measure.
\newblock {\em arXiv preprint arXiv:1608.04022}, 2016.

\bibitem[TZ17]{TZ17}
T.~{Toro} and Z.~{Zhao}.
\newblock {Boundary rectifiability and elliptic operators with $W^{1,1}$
  coefficients}.
\newblock {\em ArXiv e-prints}, August 2017.

\bibitem[Wu86]{Wu86}
J-M. Wu.
\newblock On singularity of harmonic measure in space.
\newblock {\em Pacific J. Math.}, 121(2):485--496, 1986.

\bibitem[Wu94]{Wu94}
J-M. Wu.
\newblock Porous sets and null sets for elliptic harmonic measures.
\newblock {\em Trans. Amer. Math. Soc.}, 346(2):455--473, 1994.

\bibitem[Zie74]{Zie74}
W.~P. Ziemer.
\newblock Some remarks on harmonic measure in space.
\newblock {\em Pacific J. Math.}, 55:629--637, 1974.

\end{thebibliography}

\newcommand{\etalchar}[1]{$^{#1}$}
\def\cprime{$'$}

\end{document}